\documentclass[a4paper, 10pt, american,reqno]{amsart}
\usepackage[utf8]{inputenc}
\usepackage{amsmath,amsthm,amssymb}     
\usepackage{lmodern}
\usepackage[T1]{fontenc}        
\usepackage[a4paper]{geometry}
\usepackage{babel}
\usepackage{dsfont}
\usepackage{graphicx}
\usepackage{url}
\usepackage{cases}
\usepackage{tikz}
\usepackage{tikz-cd} 
\usepackage{xcolor}
\usepackage{bbold}
\usepackage{titletoc}

\usepackage{stmaryrd}
\usepackage{csquotes}
\usepackage[colorlinks,pdfpagelabels,pdfstartview = FitH,bookmarksopen
= true,bookmarksnumbered = true,linkcolor = blue,plainpages =
false,hypertexnames = false,citecolor = red,pagebackref=false]{hyperref}
\hypersetup{linkcolor=blue,linktoc=page}

\makeatletter
\newcommand*{\mint}[1]{%
	\mint@l{#1}{}%
}
\newcommand*{\mint@l}[2]{%
	\@ifnextchar\limits{%
		\mint@l{#1}%
	}{%
		\@ifnextchar\nolimits{%
			\mint@l{#1}%
		}{%
			\@ifnextchar\displaylimits{%
				\mint@l{#1}%
			}{%
				\mint@s{#2}{#1}%
			}%
		}%
	}%
}
\newcommand*{\mint@s}[2]{%
	\@ifnextchar_{%
		\mint@sub{#1}{#2}%
	}{%
		\@ifnextchar^{%
			\mint@sup{#1}{#2}%
		}{%
			\mint@{#1}{#2}{}{}%
		}%
	}%
}
\def\mint@sub#1#2_#3{%
	\@ifnextchar^{%
		\mint@sub@sup{#1}{#2}{#3}%
	}{%
		\mint@{#1}{#2}{#3}{}%
	}%
}
\def\mint@sup#1#2^#3{%
	\@ifnextchar_{%
		\mint@sup@sub{#1}{#2}{#3}%
	}{%
		\mint@{#1}{#2}{}{#3}%
	}%
}
\def\mint@sub@sup#1#2#3^#4{%
	\mint@{#1}{#2}{#3}{#4}%
}
\def\mint@sup@sub#1#2#3_#4{%
	\mint@{#1}{#2}{#4}{#3}%
}
\newcommand*{\mint@}[4]{%
	\mathop{}%
	\mkern-\thinmuskip
	\mathchoice{%
		\mint@@{#1}{#2}{#3}{#4}%
		\displaystyle\textstyle\scriptstyle
	}{%
		\mint@@{#1}{#2}{#3}{#4}%
		\textstyle\scriptstyle\scriptstyle
	}{%
		\mint@@{#1}{#2}{#3}{#4}%
		\scriptstyle\scriptscriptstyle\scriptscriptstyle
	}{%
		\mint@@{#1}{#2}{#3}{#4}%
		\scriptscriptstyle\scriptscriptstyle\scriptscriptstyle
	}%
	\mkern-\thinmuskip
	\int#1%
	\ifx\\#3\\\else_{#3}\fi
	\ifx\\#4\\\else^{#4}\fi  
}
\newcommand*{\mint@@}[7]{%
	\begingroup
	\sbox0{$#5\int\m@th$}%
	\sbox2{$#5\int_{}\m@th$}%
	\dimen2=\wd0 %
	\let\mint@limits=#1\relax
	\ifx\mint@limits\relax
	\sbox4{$#5\int_{\kern1sp}^{\kern1sp}\m@th$}%
	\ifdim\wd4>\wd2 %
	\let\mint@limits=\nolimits
	\else
	\let\mint@limits=\limits
	\fi
	\fi
	\ifx\mint@limits\displaylimits
	\ifx#5\displaystyle
	\let\mint@limits=\limits
	\fi
	\fi
	\ifx\mint@limits\limits
	\sbox0{$#7#3\m@th$}%
	\sbox2{$#7#4\m@th$}%
	\ifdim\wd0>\dimen2 %
	\dimen2=\wd0 %
	\fi
	\ifdim\wd2>\dimen2 %
	\dimen2=\wd2 %
	\fi
	\fi
	\rlap{%
		$#5%
		\vcenter{%
			\hbox to\dimen2{%
				\hss
				$#6{#2}\m@th$%
				\hss
			}%
		}%
		$%
	}%
	\endgroup
}

\newcommand{\lyxaddress}[1]{
	\par {\raggedright #1
	\vspace{1.4em}
	\noindent\par}
}

\newtheorem{Lem}{Lemma}

\newtheorem{Theo}{Theorem}

\theoremstyle{remark}
\newtheorem{remark}{Remark}

\newcommand{\N}{\mathbb{N}}

\newcommand{\R}{\mathbb{R}}

\allowdisplaybreaks
\numberwithin{equation}{section}
\numberwithin{Lem}{section}

\begin{document}

\title[Partial gradient regularity]{\textbf{Partial gradient regularity for parabolic systems with degenerate diffusion and Hölder continuous coefficients}}
\author{Fabian Bäuerlein}
\date{September 24, 2024}

\begin{abstract} We consider vector valued weak solutions $u:\Omega_T\to \mathbb{R}^N$ with $N\in \mathbb{N}$ of degenerate or singular parabolic systems of type
\begin{equation*}
    \partial_t u - \mathrm{div} \, a(z,u,Du) = 0 \qquad\text{in}\qquad \Omega_T= \Omega\times (0,T),
\end{equation*}
where $\Omega$ denotes an open set in $\mathbb{R}^{n}$ for $n\geq 1$ and $T>0$ a finite time. Assuming that the vector field $a$ is not of Uhlenbeck-type structure, satisfies $p$-growth assumptions and $(z,u)\mapsto a(z,u,\xi)$ is Hölder continuous for every $\xi\in \mathbb{R}^{Nn}$, we show that the gradient $Du$ is partially Hölder continuous, provided the vector field degenerates like that of the $p$-Laplacian for small gradients.
\end{abstract}



\maketitle
\tableofcontents

\section{Introduction}\label{Intro}

In this paper we prove partial local Hölder continuity for the gradient of a possibly vector valued weak solution $u:\Omega_T\to \R^N$ with $N\in \N$ for degenerate or singular parabolic systems of type
\begin{align}\label{Sys1}
    \partial_t u - \mathrm{div} \, a(z,u,Du) = 0 \qquad\text{in}\qquad \Omega_T.
\end{align}
Here $\Omega$ denotes an open set of $\R^n$ with $n\geq 1$ and $\Omega_T$ the cylindrical domain $\Omega\times (0,T)$ for some time $T>0$. The vector field $a: \Omega_T \times \R^N \times \R^{Nn} \to \R^N$ satisfies for $\frac{2n}{n+2} < p$ degenerate ($p>2$) or singular ($p<2$) $p$-growth conditions with respect to the gradient variable and behaves like the $p$-Laplacian vector field for small gradients. This will be made precise in Section \ref{Assumptions}. Furthermore, we assume that $a$ is Hölder continuous with respect to the space-time and solution variables, namely that $(z,u)\mapsto a(z,u,\xi)$ is Hölder continuous for every $\xi\in \R^{Nn}$ and impose an additional regularity assumption on the gradient variable. The main novelty of our approach is that we do not assume a Uhlenbeck-type structure condition for the vector field, which is, to the knowledge of the author, the only condition that implies everywhere regularity for solutions to systems of type (\ref{Sys1}). Uhlenbeck-type structure is thereby understood as follows: There exists a scalar function $\Tilde{a}: \Omega \times \R^N \times \R \to \R$ such that 
\begin{equation*}
    a(z,u,Du)=\Tilde{a}(z,u,|Du|) Du,
\end{equation*}
namely that the vector field points in the direction of the gradient. The stated condition can be slightly relaxed, see for instance \cite{BOGELEIN2022113119}, nonetheless there is still a quite restrictive assumption on the direction of the vector field imposed. \\
In order to understand the considered continuity of the gradient, let us briefly summarize what type of regularity we can expect for systems of type (\ref{Sys1}). For stationary solutions, i.e. solutions without time dependency, where $a$ is of Uhlenbeck-type structure, it is known that the gradient is locally Hölder continuous. This was proven in the classical works by Ural'tseva \cite{Ural'tseva} for degenerate equations ($N=1$), see also \cite{Ural'tseva2}, and Uhlenbeck \cite{Uhlenbeck} for degenerate systems ($N\geq 1$). The latter were also treated by Giaquinta and Modica in \cite{GiaquintaModica}. Furthermore, the same gradient regularity was obtained for singular equations by Tolksdorf in \cite{tolksdorf1984regularity} and thereafter by Acerbi and Fusco in \cite{Acerbi} for singular systems. For stationary solutions to systems of type (\ref{Sys1}) for which the vector field $a$ does not possess Uhlenbeck-type structure, we cannot expect the gradient to be Hölder continuous in the entire domain $\Omega$. Namely, it can even occur that solutions $u$ possess singularities. For counterexamples, we refer to the survey \cite{Darkside} by Mingione. Instead, one can prove that the gradient is locally partial Hölder continuous. This is to be understood as follows: There exists a set $\Sigma \subset \Omega$ with $n$-dimensional Lebesgue-measure zero such that $\Omega\setminus \Sigma$ is open and $Du \in C^{\alpha}_\mathrm{loc}(\Omega\setminus \Sigma , \R^N)$, for some $\alpha\in (0,1)$. Using the method of $p$-caloric approximation, this was established by Duzaar and Mingione \cite{duzaar2004regularity}. For a general account of the partial regularity theory, we refer to \cite{giaquinta1983multiple,GuistiDirect}. Further contributions to partial regularity include~\cite{bogelein2012partial,irving2023bmo} and a list of references can be found in~\cite{Darkside}. \\
The local Hölder continuity for evolutionary systems of Uhlenbeck-type structure was established in the groundbreaking works of DiBenedetto and Friedman \cite{DiBenedettoFriedman1,DiBenedettoFriedman2,DiBenedettoFriedman3}, where the nowadays widespread concept of intrinsic geometry is extensively used to derive appropriate local estimates. This approach is at length discussed in the monograph \cite{DiBenedettoDPE} by DiBenedetto. Again, everywhere local Hölder continuity for the gradient cannot be expected if the vector field $a$ does not possess Uhlenbeck structure. Nevertheless, local partial regularity, i.e., the existence of a subset $\Sigma \subset \Omega_T$ with $n+1$-dimensional Lebesgue-measure zero such that $\Omega_T\setminus \Sigma$ is open and $Du \in C^{\alpha}_\mathrm{loc}(\Omega_T\setminus \Sigma , \R^N)$ holds. This was established for singular and degenerate systems by Bögelein, Duzaar and Mingione \cite{bogelein2013regularity}, where the vector field does not depend on the space-time point $z$ and the solution $u$ and later generalized by Ok, Scilla and Stroffolini to vector fields that satisfy Orlicz-growth \cite{ok2024partial}. Both papers heavily rely on the machinery of the $\mathcal{A}$-caloric- and $p$-caloric approximation techniques. The former was first used for the treatment of non-degenerate linear ($p=2$) second order parabolic systems by Duzaar and Mingione in \cite{duzaar2005second}. Thereafter, both techniques were used for non-linear systems in the super-quadratic regime ($p\geq 2$) by Duzaar, Mingione and Steffen \cite{duzaar2011parabolic}, whereas the sub-quadratic regime ($2n/(n+2)<p<2$) was treated by Scheven \cite{Scheven2006}. Another paper for non-degenerate parabolic systems with Orlicz growth is the one by Foss, Isernia, Leone and Verde \cite{Foss2022}, which was later generalized in \cite{ok2024partial}. In the present paper, we generalize the partial regularity obtained in \cite{bogelein2013regularity} to vector fields that can depend on the space-time point as well as the solution and possess degenerate diffusion. Analogs to the results presented in this paper were obtained for non-degenerate diffusion in \cite{duzaar2011parabolic,Scheven2006}. Moreover, we simplify the approach given in \cite{bogelein2013regularity} with the help of a new a priori estimate, namely Lemma \ref{Lem7.02}, for the $p$-Laplacian. Furthermore, the proof does not require Hölder continuity for the solution $u$. Let us now take a closer look at the structure of the vector field $a$.

\section{Assumptions and results}\label{Assumptions}

Throughout the paper we assume $p>1$. The vector field $a: \Omega_T \times \R^N \times \R^{Nn} \to \R^N$ is assumed to be Borel-measurable and the partial map $\xi\mapsto a(\,\cdot\,,\,\cdot\,,\xi)$ is of class $C^1$ on $\R^{Nn}$ for $p\geq 2$ and on $\R^{Nn}\setminus \{0\}$ for $1<p<2$. Moreover, the vector field  satisfies the classical \textbf{$p$-growth and -ellipticity assumptions}
\begin{align}
    &|a(z,u,\xi)|\leq L |\xi|^{p-1} \label{As1}\\ 
    &\nu |\xi|^{p-2} |\zeta|^2 \leq \langle( D_\xi a)(z,u,\xi) \cdot \zeta,\zeta \rangle = \sum\limits_{i=1}^n\sum\limits_{j=1}^N \big( ( D_{\xi_i^j} a)(z,u,\xi) \cdot \zeta \big) \zeta_i^j\label{As2}
\end{align}
for almost every $z\in\Omega_T$ and all $u\in \R^N$, $\zeta, \xi\in \R^{Nn}$. Furthermore, given $M\geq 0$, we assume that there exists a number $\kappa_M>0$, such that we have for every $\xi\in \R^{Nn}$ with $0<|\xi|\leq M$ that
\begin{equation}\label{As3}
    \big| ( D_\xi a) (z,u,\xi) \big| \leq L \kappa_M |\xi|^{p-2}
\end{equation}
for some $L>0$ and almost every $z\in\Omega_T$ and all $u\in \R^N$. The function $\kappa: \R_{\geq 0}\to \R_{\geq 0}; M\mapsto \kappa_M$ is assumed to be non-decreasing. In the sub-quadratic case, we define $( D_\xi a) (z,u,\xi) = 0$, whenever $\xi=0$. Furthermore, we can assume $\kappa_1=1$ by re-scaling $L$. We additionally impose a continuity assumption on the derivative of the vector field $a$ with respect to the gradient variable $\xi$. Let $\omega_{M}: [0,\infty)\to [0,\infty)$ be a non-decreasing modulus of continuity such that $\omega_M^2$ is concave, then we require
\begin{equation}\label{ModCon}
    | ( D_\xi a) (z,u,\xi_1) - ( D_\xi a) (z,u,\xi_2) | \leq \begin{cases}
        L \omega_{M} \bigg( \frac{|\xi_1 - \xi_2|^2}{|\xi_1|^2 + |\xi_2|^2} \bigg) \big( |\xi_1|^2 + |\xi_2|^2 \big)^{\frac{p-2}{2}} &\text{for } p\geq 2, \\
        L \omega_{M} \bigg( \frac{|\xi_1 - \xi_2|^2}{|\xi_1|^2 + |\xi_2|^2} \bigg) \bigg( \frac{|\xi_1|^2 + |\xi_2|^2}{|\xi_1|^2 |\xi_2|^2} \bigg)^{\frac{2-p}{2}} &\text{for } p<2
    \end{cases} 
\end{equation}
for almost every $z\in \Omega_T$ and all $u\in \R^N$, $\xi_1,\xi_2\in \R^{Nn}$ satisfying $|\xi_1|,|\xi_2|\leq M$. With respect to the $(z,u)$-variables we assume Hölder-continuity of $a$, namely the existence of some $\beta\in (0,1)$ such that
\begin{equation}\label{ModConu}
    |a(z,u,\xi) - a(z_0,u_0,\xi)| \leq L \min\big\{ (|x-x_0|^2 + |t-t_0|)^{\frac{\beta}{2}} + |u-u_0|^\beta , 1 \big\} |\xi|^{p-1}
\end{equation}
for every $z=(x,t),z_0=(x_0,t_0) \in \Omega $, $u,u_0\in \R^n$ and all $\xi\in \R^{N n}$. Furthermore, we require that the vector field $a$ exhibits $p$-Laplace type behavior for small gradients in the sense that
\begin{equation}\label{Beh0}
    \lim\limits_{s\rightarrow 0} \frac{a(z,u,s\xi)}{s^{p-1}} = \Tilde{a}(z,u) |\xi|^{p-2} \xi
\end{equation}
holds uniformly in $\{ \xi \in \mathbb{R}^{Nn} \,|\, |\xi|=1 \}$ for almost every $z \in \Omega_T$ and $u \in \mathbb{R}^{N}$. Thereby, $\Tilde{a}: \Omega_T \times \mathbb{R}^{N}\to \mathbb{R}$ denotes a measurable function satisfying
\begin{align}\label{Asa}
    &\nu \leq \Tilde{a}(z,u) \leq L
\end{align}
for every $(z,u)\in \Omega_T\times \R^N$, which is Hölder-continuous with exponent $\beta$, i.e. 
\begin{equation}\label{ModConu2}
    |\Tilde{a}(z,u) - \Tilde{a}(z_0,u_0)| \leq L \min\big\{ (|x-x_0|^2 + |t-t_0|)^{\frac{\beta}{2}} + |u-u_0|^\beta , 1 \big\}
\end{equation}
holds for every $z,z_0 \in \Omega_T $ and $u,u_0\in \R^n$. The notion of (weak) solutions to (\ref{Sys1}) adopted throughout this paper is that of the usual distributional one. This is understood to be a map 
\begin{align}
    u\in  C\big([0,T], L^p(\Omega,\R^N)\big) \cap L^p \big(0,T,W^{1,p}\big(\Omega,\R^N\big)\big)
\end{align}
satisfying
\begin{align}\label{weakfrom}
    \int\limits_{\Omega_T} u\cdot \partial_t \phi - a(z,u,Du) \cdot D\phi \, dz = 0
\end{align}
for every $\phi\in C^\infty_0(\Omega_T,\R^N)$, where we require at least (\ref{As1}) for the integral to be finite. We are now in the position to state the main partial regularity result.
\begin{Theo}[Partial regularity]\label{Theo1}
    Let
    \begin{equation*}
        \frac{2n}{n+2} < p \neq 2
    \end{equation*}
    and
    \begin{align*}
        u\in  C\big([0,T], L^p(\Omega,\R^N)\big) \cap L^p \big(0,T,W^{1,p}\big(\Omega,\R^N\big)\big)
    \end{align*}
    be a weak solution of the parabolic system (\ref{Sys1}) that satisfies the assumptions (\ref{As1})-(\ref{ModConu2}). Then there exists $\alpha=\alpha(n,N,p,\nu,L,\kappa_3)\in (0,\beta)$ and an open set $Q_0\subset \Omega_T$ such that
    \begin{align}
        Du \in C_\mathrm{loc}^{\alpha,\alpha/2}(Q_0,\R^{Nn}) \qquad\text{and}\qquad |\Omega_T \setminus Q_0|=0.
    \end{align}
\end{Theo}
\begin{remark}
    We exclude the case $p=2$, since this has already been dealt with in \cite{duzaar2005second}.
\end{remark}
If we denote for $(x_0,t_0)=z_0\in\Omega_T$ and $\varrho>0$ by $Q^{(1)}_\varrho(z_0)\subset \Omega_T$ the standard parabolic cylinder $Q^{(1)}_\varrho(z_0)=B_\varrho(x_0) \times (t_0-\varrho^2,t_0+\varrho^2)$, where $B_\varrho(x_0)$ is an open ball of radius $\varrho$ with center $x_0$ in $\R^n$, we define the \textbf{excess-functional} 
\begin{align}\label{excess}
    \Phi_1(u;z_0,\varrho)= \frac{1}{|Q^{(1)}_\varrho|}\int\limits_{Q^{(1)}_\varrho(z_0)} \Big( \big|  (Du)_{Q^{(1)}_\varrho(z_0)} \big|^2 + \big| Du - (Du)_{Q^{(1)}_\varrho(z_0)} \big|^2 \Big)^{\frac{p-2}{2}} \big| Du - (Du)_{Q^{(1)}_\varrho(z_0)} \big|^2 \, dz.
\end{align}
Hereby,
\begin{align*}
    (Du)_{Q^{(1)}_\varrho(z_0)} = \frac{1}{|Q^{(1)}_\varrho|}\int\limits_{Q^{(1)}_\varrho(z_0)} Du \, dz
\end{align*}
denotes the \textbf{cylindrical mean} of $Du$ with respect to the cylinder $Q^{(1)}_\varrho(z_0)$. Then we have the following characterization of the singular set.
\begin{Theo}[Description of the singular set]\label{Theo2}
    Under the assumptions of Theorem \ref{Theo1}, we know that the singular set $\Sigma=\Omega_T \setminus Q_0$ is contained in $\Sigma_1\cup\Sigma_2$, where
    \begin{align}
        &\Sigma_1 \equiv \Big\{ z_0\in \Omega_T \,\Big|\, \liminf\limits_{\varrho\downarrow 0} \Phi_1 (u;z_0,\varrho) > 0 \Big\} ,\\
        &\Sigma_2 \equiv \Big\{ z_0\in \Omega_T \,\Big|\, \limsup\limits_{\varrho\downarrow 0} \big|(Du)_{Q^{(1)}_\varrho(z_0)}\big| = \infty \Big\}.
    \end{align}
\end{Theo}
\begin{remark}
    The set $Q_0$ contains the set of all Lebesgue points of $Du$ that are not in $\Sigma_2$, i.e. the set
    \begin{align}
        &Q_1 \equiv \Bigg\{ z_0\in \Omega_T \,\Big|\, \liminf\limits_{\varrho\downarrow 0} \frac{1}{|Q^{(1)}_\varrho|}\int\limits_{Q^{(1)}_\varrho(z_0)} \big| Du - (Du)_{Q^{(1)}_\varrho(z_0)} \big|^p \, dz = 0 \Bigg\} \setminus \Sigma_2.
    \end{align}
\end{remark}
 This can be seen as follows: In the case $p<2$ we have the inequality
 \begin{align*}
     \Phi_1(u;z_0,\varrho) \leq \frac{1 }{|Q^{(1)}_\varrho|} \int\limits_{Q^{(1)}_\varrho(z_0)} \big| Du - (Du)_{Q^{(1)}_\varrho(z_0)} \big|^p \, dz
 \end{align*}
 and in the case $p>2$, we obtain with the help of Hölder's inequality that
 \begin{align*}
     \Phi_1(u;z_0,\varrho) 
     &\leq \frac{2^p}{|Q^{(1)}_\varrho|} \int\limits_{Q^{(1)}_\varrho(z_0)} 
     \big| (Du)_{Q^{(1)}_\varrho(z_0)} \big|^{p-2}
    \big| Du - (Du)_{Q^{(1)}_\varrho(z_0)} \big|^2 
    +
    \big| Du - (Du)_{Q^{(1)}_\varrho(z_0)} \big|^p \, dz
    \nonumber \\ & \leq 
    2^p 
    \big| (Du)_{Q^{(1)}_\varrho(z_0)} \big|^{p-2}
    \Bigg[ \frac{1}{|Q^{(1)}_\varrho|} \int\limits_{Q^{(1)}_\varrho(z_0)}  \big| Du - (Du)_{Q^{(1)}_\varrho(z_0)} \big|^p \, dz \Bigg]^{\frac{2}{p}}
    \nonumber \\ & \qquad +
    \frac{2^p}{|Q^{(1)}_\varrho|} \int\limits_{Q^{(1)}_\varrho(z_0)} \big| Du - (Du)_{Q^{(1)}_\varrho(z_0)} \big|^p \, dz .
 \end{align*}
 We see that the first factor of the first term on the right-hand side is bounded since $z_0\notin \Sigma_2$. Both estimates of $\Phi_1$ immediately yield the desired characterization.

\subsection{Strategy of proof}\label{Strategy}

The proof will be achieved by means of the method of $p$-caloric approximation, which has been introduced in \cite{bogelein2013regularity}. The main ideas and techniques used are briefly presented in the present section. It should be mentioned that we distinguish between the super- and sub-quadratic cases, i.e., if $p>2$ or $2n/(n+2)<p<2$ holds for the parameter $p$. As this is a technical detail, we will not discuss it here. The ultimate goal is to show a decay estimate for the \textbf{excess functional} (\ref{excess}) of the type
\begin{equation}\label{decay}
    \Phi_1(u;z_0,\varrho) \leq C \varrho^\alpha,
\end{equation}
which implies the Hölder continuity of $Du$ by a standard characterization dating back to Campanato and Da Prato. To obtain this decay estimate, we distinguish between the \textbf{degenerate regime (DR)} and \textbf{non-degenerate regime (NDR)}, which are defined by considering the \textbf{hybrid excess functional}
\begin{align*}
     E_\lambda(z_0,\varrho) := \Phi_\lambda(u;z_0,\varrho) + H_\lambda(u;z_0,\varrho),
\end{align*}
where the functions $\Phi_\lambda$ and $H_\lambda$ are defined in \eqref{6.01} and \eqref{8.04}, respectively, and decay with respect to the radius $\varrho$. For $H_\lambda$, without going into the details, we will see that there holds
\begin{align*}
    H_\lambda(u;z_0,\varrho)\leq C \varrho^\beta \big| (Du)_{Q^{(\lambda)}_\varrho(z_0)} \big|^p,
\end{align*}
where $Q^{(\lambda)}_\varrho(z_0)$ denotes the  intrinsic geometric cylinder $B_\varrho(x_0) \times (t_0-\lambda^{2-p}\varrho^2,t_0+\lambda^{2-p}\varrho^2)$. Here the additional term $H_\lambda$ is due to the Hölder-continuity assumption (\ref{ModConu}) with respect to the space-time and solution variables $(z,u)$ of the vector field $a$. In analogy to the treatment of the excess functional in \cite{bogelein2013regularity}, we are in the non-degenerate regime, if the hybrid excess functional satisfies
\begin{align*}
    E_\lambda(z_0,\varrho) \ll \big| (Du)_{Q^{(\lambda)}_\varrho(z_0)} \big|^p,
\end{align*}
i.e., that the hybrid excess is very small compared to the mean of the gradient in a fixed intrinsic geometric cylinder. Furthermore, we additionally assume that the intrinsic scaling of the cylinder is coupled to the parameter $\lambda$. More precisely, we have
\begin{align*}
    \big| (Du)_{Q^{(\lambda)}_\varrho(z_0)} \big| \approx \lambda,
\end{align*}
due to which $Q^{(\lambda)}_\varrho(z_0)$ is referred to as an intrinsically scaled cylinder. This particular choice of scaling allows us to freeze the coefficients of the system around the center $z_0$ and the cylindrical mean of $u$ and compare the solution locally to the solution $v$ of a linear parabolic system. Thereby, the so-called $\mathcal{A}$-caloric approximation derived in \cite{bogelein2013regularity, duzaar2011parabolic,Scheven2006} is used. Then $u$ inherits the well-known a priori estimates for $v$ by comparison. This yields, for fixed $\vartheta\in (0,1)$, a quantified excess improvement of the form
\begin{align*}
    E_\lambda(z_0,\vartheta\varrho) \leq \vartheta^\beta E_\lambda(z_0,\varrho).
\end{align*}
Note that the improvement of the second term $H_\lambda$ of $E_\lambda$ is due to the upper bound $\varrho^\beta$. It should be mentioned that in the non-degenerate regime, the cylindrical mean of $Du$ remains approximately constant when transitioning to the smaller cylinder. We are therefore again in the non-degenerate regime on the next standard parabolic cylinder $Q^{(\lambda)}_{\vartheta\varrho}(z_0)$, which implies that the quantified excess improvement obtained can be iterated. From said iteration, we derive the excess decay (\ref{decay}). It remains to be investigated what happens in the degenerate regime.\\
As the degenerate regime refers to the case where the above assumptions of the non-degenerate case are not satisfied, we either have that
\begin{align}\label{Degen}
    \big| (Du)_{Q^{(\lambda)}_\varrho(z_0)} \big|^p \ll E_\lambda(z_0,\varrho) \qquad\text{or}\qquad 
    \big| (Du)_{Q^{(\lambda)}_\varrho(z_0)} \big| \ll \lambda,
\end{align}
where the alternative
\begin{align*}
     \lambda\ll \big| (Du)_{Q^{(\lambda)}_\varrho(z_0)} \big|
\end{align*}
is ruled out during the proof. Both inequalities of (\ref{Degen}) allow us to exploit the behavior of $a$ for small gradients, namely (\ref{Beh0}). Therefore, we can compare our solution $u$ to the solution $v$ of a suitable system of $p$-Laplacian type in the considered cylinder. This is done by means of the so-called $p$-caloric approximation procedure from \cite{bogelein2013regularity}, which allows us to transfer the a priori estimate of $v$ to $u$ and provides us with information about the excess and the cylindrical mean of $Du$ on a sub-cylinder $Q^{(\Tilde{\lambda})}_{\vartheta\varrho}(z_0)$, whose scaling parameter $\Tilde{\lambda}$ can be chosen within the interval $[\vartheta^{\alpha_1}\lambda,\lambda]$. In fact, we obtain
\begin{align*}
    E_{\Tilde{\lambda}}(u;z_0,\vartheta\varrho) \ll \Tilde{\lambda}^p \qquad\text{and}\qquad
    \big| (Du)_{Q^{(\Tilde{\lambda})}_{\vartheta\varrho}(z_0)}\big| \leq C \lambda.
\end{align*}
Note that we cannot iterate the degenerate regime and the previous estimate for the cylindrical mean of the gradient is not sufficient to establish the excess decay later since the right-hand side depends on $\lambda$ instead of $\Tilde{\lambda}$. However, the latter issue can be circumvented by considering the behavior of the solution on a sub-cylinder. Namely, we distinguish whether the degenerate or non-degenerate regime is applicable on said sub-cylinder. In this way, we obtain for $\lambda_1 = \vartheta^{\alpha_1}\lambda$ that
\begin{align*}
    \big| (Du)_{Q^{(\lambda_1)}_{\vartheta\varrho}(z_0)}\big| \leq \lambda_1,
\end{align*}
which, together with $E_{\lambda_1}(u;z_0,\vartheta\varrho) \ll \lambda_1^p$, provides the desired quantified excess improvement in the degenerate regime. On the new intrinsic geometric cylinder $Q^{(\lambda_1)}_{\vartheta\varrho}(z_0)$, we again distinguish between the non-degenerate and the degenerate regimes, which iteratively yields the desired excess decay (\ref{decay}) and completes the proof.

\subsection{Plan of the paper} 

We start the next Section \ref{Not} by fixing the notation and introducing some preliminary result, which will be used throughout the paper. Especially the therein introduced $V$-function and time independent minimizing affine functions will play an important role in many calculations. Furthermore, we introduce the $\mathcal{A}$-caloric and $p$-caloric approximations in both this and the thereafter following Sections \ref{Not} and \ref{pCal}. We continue by establishing suitable Caccioppoli- and Poincaré inequalities in Section \ref{SCacc}. Thereafter, the parabolic system is linearized in the Section \ref{SLin}, depending on whether we are in the non-degenerate or degenerate regime on some intrinsic geometric cylinder. After an quick interlude in Section \ref{SDI}, where we establish a priori estimates for solutions to linear parabolic systems as well as for solutions to the parabolic $p$-Laplacian system, we continue with the proof of partial regularity in Section \ref{SIN}.

\section{Notation and preliminaries}\label{Not}

\subsection{Notation}

Let us start this section by fixing the notation used throughout this paper. We denote by $C$ generally a positive constant that may vary on different occasions, even within the same line of estimates. If we want to highlight the dependencies on parameters of said constants, we will use parenthesis. In order to avoid an overburdened notation as a result of the dependencies of the constants from the structural parameters, we use the abbreviation $\mathfrak{C}$ in order to indicate the dependence on $n,N,p,\nu,L$ or some of these parameters; for example, when writing $C=C(\mathfrak{C},\gamma)$ we mean that $C$ depends on $n,N,p,\nu,L$ and also on $\gamma$. The norm on the Euclidean space $\R^k$, denoted by $|\cdot|$ will be the standard one, i.e. $|\xi|=\sqrt{\xi\cdot\xi}$ for $\xi \in \R^k$, where $\cdot$ denotes the inner product of $\R^k$. Furthermore, we will identify matrices $\R^{k \times l}$ with $\R^{kl}$.\\
For space-time points, we will use the abbreviation $z=(x,t)$ for the spatial variable $x\in\R^n$ and times $t\in\R$. By $B_\varrho(x_0)=\{ x\in\R^n \,|\, |x-x_0|<\varrho \}$ we denote the spatial open ball with radius $\varrho>0$ and center $x_0\in \R^n$. Moreover, we use the notation
\begin{align*}
    Q^{(\lambda)}_\varrho(z_0):= B_\varrho(x_0) \times \Lambda_\varrho^{(\lambda)}(t_0) := B_\varrho(x_0) \times (t_0-\lambda^{2-p}\varrho^2, t_0+\lambda^{2-p}\varrho^2),
\end{align*}
where $z_0=(x_0,t_0)\in \R^n\times \R$ and $\lambda,\varrho\in \R_+$, for the intrinsic geometric cylinder with vertex $(x_0,t_0)$ and width $\varrho$. The parameter $\lambda$ describes the scaling of the cylinder. If the center of balls and cylinders is clear from context, we shall simply omit it in the notation, which will be denoted by $B_\varrho\equiv B_\varrho(x_0)$ for example. We do the same for intrinsic geometric cylinders when $\lambda=1$ holds. \\
If $E\subset \R^k$ is a Lebesgue-measurable set, we denote by $|E|$ its $k$-dimensional Lebesgue measure. Furthermore, if $0<|E|<\infty$, the mean of a function $f\in L^1(E)$ is defined by
\begin{align*}
    (f)_E:= \mint{-}\limits_E f(y) \, dy = \frac{1}{|E|}\int\limits_E f(y) \, dy
\end{align*}
and we abbreviate $(f)_{z_0;\varrho}^{(\lambda)} \equiv (f)_{Q^{(\lambda)}_\varrho(z_0)}$. We again omit the center $z_0$ in the notation, when it is clear from the context. \\
Due to the differing behavior of solutions, we introduce the symbols
\begin{equation*}
    \mathds{1}_{\{p<2\}} \qquad\text{and}\qquad \mathds{1}_{\{p>2\}}
\end{equation*}
to unify the treatment of the sub-quadratic $p<2$ and super-quadratic case $p\geq 2$. The first symbol indicates that the terms that are multiplied by the symbol only occur in the sub-quadratic case, and the latter does the same for the super-quadratic case.

\subsection{The \texorpdfstring{$V$}{V}-function}

In the following we define the auxiliary function $V_\lambda : \R^k \to \R^k$ for $k\in \N$ and $\lambda\geq 0$ by
\begin{align}\label{2.21}
    V_\lambda(B) := \big( \lambda^2 + |B|^2 \big)^{\frac{p-2}{4}} B \qquad\text{for } B\in \R^k,
\end{align}
where we set $V_\lambda(B)=0$, if $\lambda$ and $B$ both assume zero in the sub-quadratic case. This so called $V$-function will be essential for most estimates. We start by collecting some useful algebraic properties of said function. For the following Lemma we refer to \cite[Lemma 1]{duzaar2004regularity}.
\begin{Lem}\label{Lem2.1}
    Let $p>1$, $k\in \mathbb{N}$ and $\lambda\geq 0$. Then for any $A,B\in \mathbb{R}^k$ there holds
    \begin{equation}\label{2.11}
        \big( \lambda^2 + |A|^2 \big)^{\frac{p-2}{2}}|A| |B| 
        \leq
        C(p) \big( |V_\lambda(A)|^2 + |V_\lambda(B)|^2 \big)
    \end{equation}
    and
    \begin{equation}\label{2.12}
        |V_\lambda(A+B)|
        \leq
        C(p) \big( |V_\lambda(A)| + |V_\lambda(B)| \big)
    \end{equation}
    and
    \begin{equation}\label{2.13}
        |V_\lambda(A) - V_\lambda(B)|
        \leq
        C(p,k,M) |V_\lambda(A-B)| \qquad\text{if } |B|\leq M
    \end{equation}
    Furthermore, in the case $1<p<2$ there exists $C=C(p)$ such that
    \begin{equation}\label{2.14}
        C^{-1} \min\big\{ \lambda^{\frac{p-2}{2}} |A| , |A|^{\frac{p}{2}} \big\} \leq |V_\lambda(A)| \leq C^{-1} \max\big\{ \lambda^{\frac{p-2}{2}} |A| , |A|^{\frac{p}{2}} \big\} \qquad\forall A\in \mathbb{R}^k
    \end{equation}
\end{Lem}
In addition, we will need the slight generalization of the \textit{iteration Lemma} in \cite[Lemma3.7]{habermann2008partial}. See also \cite{bogelein2013regularity,carozza1998partial}
\begin{Lem}\label{Lem2.2}
    Let $p\geq 1$, $0<\vartheta<1$, $a,b,c\geq 1$, $A\in \mathbb{R}^{Nn}$, $v\in L^p\big(Q_\varrho^{(\lambda)}(z_0),\mathbb{R}^N\big)$ and $\varphi: [r,\varrho]\to [0,\infty)$ be a bounded function satisfying 
    \begin{equation*}
        \varphi(s) \leq \vartheta \varphi(t) + a \int\limits_{Q_\varrho^{(\lambda)}(z_0)} \bigg| V_{|A|}\bigg( \frac{v}{t-s} \bigg) \bigg|^2 \, dz + \frac{b}{(t-s)^2} + c
    \end{equation*}
    for all $r\leq s < t \leq \varrho$. Then there exists a constant $C=C(\vartheta,p)$ such that
    \begin{equation*}
        \varphi(r) \leq C \Bigg[ a \int\limits_{Q_\varrho^{(\lambda)}(z_0)} \bigg| V_{|A|}\bigg( \frac{v}{\varrho-r} \bigg) \bigg|^2 \, dz + \frac{b}{(\varrho-r)^2} + c \Bigg].
    \end{equation*}
\end{Lem}
The following Lemma shows that the cylindrical mean is a quasi-minimizer of a certain functional involving the $V$-function. This will be crucial later and we refer to \cite[Lemma 2.4]{bogelein2013regularity} for a proof.
\begin{Lem}\label{Lem2.4}
    Let $p\geq 1$, $Q\subset\R^{n+1}$ such that $|Q|>0$ and $f\in L^p(Q,\R^k)$ and $k\geq 1$. Then, we have
    \begin{equation*}
        \mint{-}\limits_Q \big| V_{|(f)_Q|}(f-(f)_Q) \big|^2 \, dz
        \leq
        C(p) \mint{-}\limits_Q \big| V_{|A|}(f-A) \big|^2 \, dz 
        \qquad\forall A\in \R^k.
    \end{equation*}
    In the case $p\geq 2$ the inequality holds with the particular constant $C(p)=2^{2p}$.
\end{Lem}
The following algebraic fact can be retrieved from \cite{hamburger1992regularity}.
\begin{Lem}\label{Lem2.5}
    For every $\sigma\in (-\frac12,0)$ and $\mu\geq 0$ we have
    \begin{equation*}
        1\leq \frac{ \int\limits_0^1 \big( \mu^2 + |A + s(\Tilde{A}-A)|^2 \big)^\sigma \, ds }{\big( \mu^2 + |A|^2 + |\Tilde{A}|^2 \big)^\sigma} \leq \frac{8}{2\sigma + 1}
    \end{equation*}
    for any $A,\Tilde{A} \in \mathbb{R}^{Nn}$, not both zero in case $\mu=0$.
\end{Lem}
The following inequality immediately follows from Lemma \ref{Lem2.5}: For every $\sigma\in (-\frac12,0)$ and $\mu\geq 0$ we have
\begin{equation}\label{e2.6}
    \int\limits_0^1 \big( \mu^2 + |A + sB|^2 \big)^\sigma \, ds \leq \frac{24}{2\sigma+1} \big( \mu^2 + |A|^2 + |B|^2 \big)^\sigma
\end{equation}
for any $A,B \in \mathbb{R}^{Nn}$, not both zero if $\mu=0$.

\subsection{Basic deductions from the structure conditions}

Let us observe the following immediate consequence of hypothesis (\ref{Beh0}): There exists a function $\eta: (0,\infty) \to (0,\infty)$ such that for any $\delta>0$, $u\in \mathbb{R}^{N}$ and almost every $z\in \Omega_T$ we have
\begin{equation}\label{2.7}
    \big|a(z,u,\xi)-\Tilde{a}(z,u) |\xi|^{p-2}\xi \big| \leq \delta |\xi|^{p-1} 
    \qquad
    \text{for every } \xi \in \mathbb{R}^{Nn} \text{ with } |\xi|\leq \eta(\delta).
\end{equation}
Additionally, we have the following useful Lemma, the proof of which only uses assumptions (\ref{As1}) and (\ref{As3}). Therefore it can be obtained exactly as in \cite[Lemma 2.7]{bogelein2013regularity}.
\begin{Lem}\label{Lem2.7}
    Let $p>1$, $A,B\in \mathbb{R}^{Nn}$ satisfying $|A| \leq |B|$ and $a:\mathbb{R}^{Nn}\to\mathbb{R}^{Nn}$ satisfying the growth assumptions (\ref{As1}) \& (\ref{As3}). Then there holds for almost every $z\in \Omega_T$ and $l\in \R^N$
    \begin{equation*}
        |a(z,l,A+B)-a(z,l,A)| \leq C(p) L |B|^{p-1}.
    \end{equation*}
\end{Lem}

\subsection{Gagliardo-Nirenberg inequality}
The following form of the Gagliardo-Nirenberg inequality from \cite[Lemma 2.11]{bogelein2013regularity} will be used later. 
\begin{Lem}\label{Lem2.11}
    Let $B_\varrho(x_0)\subset \mathbb{R}^n$ with $\varrho\leq 1$, $1\leq q,p,r \leq \infty$ and $\vartheta\in (0,1)$ such that $-\frac nq\leq \vartheta ( 1- \frac n p) - (1-\vartheta) \frac nr$ and $u\in W^{1,p}(B_\varrho(x_0))$. Then there exists a constant $C=C(n,q)$ such that there holds
    \begin{equation*}
        \mint{-}\limits_{B_\varrho(x_0)} \bigg| \frac{u}{\varrho} \bigg|^q
        \leq
        C
        \Bigg( \mint{-}\limits_{B_\varrho(x_0)}  \bigg| \frac{u}{\varrho} \bigg|^p + |Du|^p \, dx \Bigg)^{\frac{\vartheta q}{p}}
        \Bigg( \mint{-}\limits_{B_\varrho(x_0)}  \bigg| \frac{u}{\varrho} \bigg|^r \, dx \Bigg)^{\frac{(1-\vartheta) q}{r}}.
    \end{equation*}
\end{Lem}

\subsection{Minimizing affine functions}\label{SMinaff}

In what follows we approximate the solution with affine functions, that depend only on the space variables. These are functions $L: \R^n \to \R^N, x \mapsto l + (Dl)x$ with $l\in \R^N$ and $(Dl)\in \R^{N\times n}$. This subsection is dedicated to summarize certain properties of these functions that will be applied later. For details we refer to \cite{bogelein2013regularity,duzaar2011parabolic,kronz2002partial}. Let $Q_\varrho^{(\lambda)}(z_0)\subset \R^{n+1}$ be an intrinsic geometric cylinder and $u\in L^2(Q_\varrho ^{(\lambda)}(z_0),\R^N)$, then we denote with $L_{z_0;\varrho}^{(\lambda)}: \R^n \to \R^N$ the unique affine map minimizing 
\begin{align}\label{2.5.1}
    L \mapsto \mint{-}\limits_{Q_\varrho^{(\lambda)}(z_0)} |u-L|^2 \, dz
\end{align}
amongst all affine maps $L(z)=L(x)$ which are independent of time $t$. The minimum of course exists and can be computed to be of the form $L_{z_0;\varrho}^{(\lambda)}(x) = l_{z_0;\varrho}^{(\lambda)} + (Dl)_{z_0;\varrho}^{(\lambda)}(x-x_0)$, where
\begin{align}\label{2.5.2}
    l_{z_0;\varrho}^{(\lambda)}= (u)_{z_0;\varrho}^{(\lambda)}
    \qquad\text{and}\qquad
    (Dl)_{z_0;\varrho}^{(\lambda)} = \frac{n+2}{\varrho^2} \mint{-}\limits_{Q_\varrho^{(\lambda)}} u\otimes (x-x_0) \, dz.
\end{align}
Moreover, for any $w\in \R^{Nn}$ there holds
\begin{align}\label{2.5.3}
    \big|(Dl)_{z_0;\varrho}^{(\lambda)} - w\big|^2 \leq \frac{n(n+2)}{\varrho^2} \mint{-}\limits_{Q_\varrho^{(\lambda)}} \big| u - (u)_{z_0;\varrho}^{(\lambda)} - w (x-x_0)\big|^2 \, dz.
\end{align}
Finally, $L_{z_0;\varrho}^{(\lambda)}$ is a quasi-minimum of the $L^p$-distance for $p\geq 2$; see \cite[Lemma 2.8]{bogelein2013regularity}.
\begin{Lem}\label{Lem2.8}
    Let $p\geq 2$, $Q_\varrho^{(\lambda)}(z_0)\subset \R^{n+1}$ be an intrinsic geometric cylinder and $u\in L^p(Q_\varrho ^{(\lambda)}(z_0),\R^N)$. For any affine function $L:\R^n\to\R^N$ independent of $t$, we then have
    \begin{align}\label{2.5.4}
    \mint{-}\limits_{Q_\varrho^{(\lambda)}(z_0)} \big|u- L_{z_0;\varrho}^{(\lambda)} \big|^p \, dz
    \leq C(n,p)
    \mint{-}\limits_{Q_\varrho^{(\lambda)}(z_0)} |u-L|^p \, dz,
    \end{align}
    where $L_{z_0;\varrho}^{(\lambda)}$ denotes the unique minimizer of (\ref{2.5.1}).
\end{Lem}

\subsection{\texorpdfstring{$\mathcal{A}$}{A}-caloric approximation Lemma}

In the non-degenerate case, we approximate the original solution with solutions to a linear parabolic system with constant coefficients. The approximation will be achieved by the following two $\mathcal{A}$-caloric approximation Lemmas, where we distinguish between the sub- and super-quadratic cases. First, we consider the super-quadratic case; for a statement, see \cite[Lemma 2.12]{bogelein2013regularity}; for a proof that requires minor modification, we refer to \cite[Lemma 3.2]{duzaar2011parabolic}.
\begin{Lem}\label{Lem2.12}
    Let $p\geq 2$. Given $0<\nu \leq L$ and $\varepsilon>0$, there exists a positive function $\delta_0=\delta_0(n,p,\nu,L,\varepsilon) \in (0,1]$ with the following property: Whenever $\mathcal{A}$ is a elliptic bilinear form on $\R^{Nn}$ with ellipticity constant $\nu$ and upper bound $L$ - i.e.
    \begin{align}\label{2.1201}
        \mathcal{A}(\xi, \xi) \geq \nu |\xi|^2
        \qquad\text{and}\qquad 
        \mathcal{A}(\xi,\Tilde{\xi}) \leq L |\xi| |\Tilde{\xi}|
    \end{align}
    hold for any $\xi,\Tilde{\xi}\in \R^{Nn}$ - and $\gamma\in (0,1]$ and a function
    \begin{equation*}
        w\in L^p\big(\Lambda_\varrho(t_0); W^{1,p}(B_\varrho(x_0),\R^N)\big)
    \end{equation*}
    with 
    \begin{align}\label{2.1202}
        \mint{-}\limits_{Q_{\varrho}(z_0)} & \bigg|\frac{w}{\varrho}\bigg|^2 + |Dw|^2 \, dz 
        +
        \gamma^{p-2}
        \mint{-}\limits_{Q_{\varrho}(z_0)} \bigg|\frac{w}{\varrho}\bigg|^p + |Dw|^p \, dz  \leq 1,
    \end{align}
    is approximately $\mathcal{A}$-caloric in the sense that there holds
    \begin{align}\label{2.1203}
        \Bigg| \mint{-}\limits_{Q_{\varrho}(z_0)} w \cdot\partial_t \phi - \mathcal{A}(Dw,D\phi) \, dz \Bigg|
        \leq
        \delta_0 \sup\limits_{Q_{\varrho}(z_0)} |D\phi|, 
        \qquad\forall \phi\in C^\infty_0(Q_{\varrho}(z_0),\R^N),
    \end{align}
    then there exists an $\mathcal{A}$-caloric function $h\in L^p\big(\Lambda_{\varrho/2}(t_0); W^{1,p}(B_{\varrho/2}(x_0),\R^N)\big)$, i.e.
    \begin{align}\label{2.1204}
        \mint{-}\limits_{Q_{\varrho/2}(z_0)} h \cdot\partial_t \phi - \mathcal{A}(Dh,D\phi) \, dz = 0 
        \qquad\forall \phi\in C^\infty_0(Q_{\varrho},\R^N),
    \end{align}
    such that
    \begin{align}\label{2.1205}
        \mint{-}\limits_{Q_{\varrho/2}(z_0)} & \bigg|\frac{h}{\varrho/2}\bigg|^2 + |Dh|^2 \, dz 
        +
        \gamma^{p-2}
        \mint{-}\limits_{Q_{\varrho/2}(z_0)} \bigg|\frac{h}{\varrho/2}\bigg|^p + |Dh|^p \, dz  \leq 2^{n+3+2p}
    \end{align}
    and
    \begin{align}\label{2.1206}
        \mint{-}\limits_{Q_{\varrho/2}(z_0)} & \bigg|\frac{w-h}{\varrho/2}\bigg|^2 \, dz 
        +
        \gamma^{p-2}
        \mint{-}\limits_{Q_{\varrho/2}(z_0)} \bigg|\frac{w-h}{\varrho/2}\bigg|^p \, dz  \leq \varepsilon.
    \end{align}
\end{Lem}
For the sub-quadratic case we refer to \cite{Scheven2006} for a $V_1$-version of the $\mathcal{A}$-caloric approximation Lemma.
\begin{Lem}\label{Lem2.13}
    Let $\frac{2n}{n+2} < p < 2$. Given $0<\nu \leq L$ and $\varepsilon>0$, there exists a positive function $\delta_0=\delta_0(n,p,\nu,L,\varepsilon) \in (0,1]$ with the following property: Whenever $\mathcal{A}$ is a elliptic bilinear form on $\R^{Nn}$ with ellipticity constant $\nu$ and upper bound $L$, as in (\ref{2.1201})  and a function
    \begin{equation*}
        w\in L^\infty\big(\Lambda_\varrho(t_0); L^2(B_\varrho(x_0),\R^N)\big) L^p\big(\Lambda_\varrho(t_0); W^{1,p}(B_\varrho(x_0),\R^N)\big)
    \end{equation*}
    with 
    \begin{align}\label{2.1302}
    \sup\limits_{t\in \Lambda_\varrho(t_0)} \mint{-}\limits_{B_\varrho(x_0)} \bigg| \frac{w(\,\cdot\,, t)}{\varrho} \bigg|^2 \, dx +
    \mint{-}\limits_{Q_{\varrho}(z_0)} & |V_1(Dw)|^2 \, dz 
    \leq 1,
    \end{align}
    is approximately $\mathcal{A}$-caloric in the sense there holds (\ref{2.1203}), then there exists an $\mathcal{A}$-caloric function
    \begin{equation*}
        h\in L^p\big(\Lambda_{\varrho/2}(t_0); W^{1,p}(B_{\varrho/2}(x_0),\R^N)\big),
    \end{equation*}
    i.e. (\ref{2.1204}) holds, such that
    \begin{align}\label{2.1305}
        \mint{-}\limits_{Q_{\varrho/2}(z_0)} & \bigg|\frac{h}{\varrho/2}\bigg|^2 + |V_1(Dh)|^2 \, dz \leq 2^{n+5}
    \end{align}
    and
    \begin{align}\label{2.1306}
        \mint{-}\limits_{Q_{\varrho/2}(z_0)} & \bigg|\frac{w-h}{\varrho/2}\bigg|^2 \, dz 
        \leq \varepsilon.
    \end{align}
\end{Lem}

\section{The \texorpdfstring{$p$}{p}-caloric approximation Lemma}\label{pCal}

In this section, we state the $p$-caloric approximation Lemma, which will be crucial when we linearize the solution for small gradients. This will be achieved by approximation with appropriate systems of parabolic $p$-Laplacian type, where the vector fields $\mathcal{A}: \R^{Nn}\to \R^{Nn}$ for $p>\frac{2n}{n+2}$ and  some $\gamma>0$ are given by
\begin{equation*}
    \mathcal{A}(\xi)= \gamma |\xi|^{p-2} \xi.
\end{equation*}
\begin{remark}
    Given the vector field $\mathcal{A}: \R^{Nn}\to \R^{Nn}: \xi \mapsto \gamma |\xi|^{p-2} \xi$, we call the vector valued function 
    \begin{align*}
        h\in  C\big([0,T], L^2(\Omega,\R^N)\big) \cap L^p\big(0,T,W^{1,p}(\Omega,\R^N)\big)
    \end{align*}
    $\mathcal{A}$-caloric in the cylinder $\Omega_T$ if it solves
    \begin{align}
        \partial_t h - \mathrm{div} \, \mathcal{A}(Dh) = 0
    \end{align}
\end{remark}
\noindent
weakly. With these structure conditions in mind, we can state the $p$-caloric approximation Lemma in the super-quadratic case. For the proof, we refer to \cite[Theorem 4.2]{bogelein2013regularity}.
\begin{Lem}\label{Lem4.1}
    Let $p\geq 2$, $\mu\in[0,1]$, $0<\nu\leq 1\leq L$ and $C_p\geq 1$ be fixed. Then for every $\varepsilon>0$ there exits a positive function $\delta_0(n,N,p,\nu,L,\mu,C_p,\varepsilon)\leq 1$ with the following property: Let $z_0=(x_0,t_0)\in \R^{n+1}$, $\varrho>0$ and $\mathcal{A}: \R^{Nn}\to \R^{Nn}$ be a vector field, such that $\nu\leq \gamma\leq L$. If $w\in L^p\big(\Lambda_\varrho(t_0); W^{1,p}(B_\varrho(x_0),\R^N)\big)$ fulfills
    \begin{align}\label{4.101}
        \mint{-}\limits_{Q_\varrho(z_0)} \bigg| \frac{w}{\varrho} \bigg|^p \, dz
        +
        \mint{-}\limits_{Q_\varrho(z_0)} |Dw|^p \, dz
        \leq 1
    \end{align}
    and is approximately ($\mathcal{A},p$)-caloric in the sense that
    \begin{align}\label{4.102}
        \Bigg| \mint{-}\limits_{Q_{\varrho}(z_0)} w \cdot\partial_t \phi - \mathcal{A}(Dw)\cdot D\phi \, dz \Bigg|
        \leq
        \delta_0 \sup\limits_{Q_{\varrho}(z_0)} |D\phi|, 
        \qquad\forall \phi\in C^\infty_0(Q_{\varrho}(z_0),\R^N),
    \end{align}
    and moreover the Poincaré inequality 
    \begin{align}\label{4.103}
        \Bigg| \mint{-}\limits_{Q^{(\lambda)}_{\varrho}(z')} \big| w - (w)_{z',r}^{(\lambda)} \big| \, dz \Bigg|
        \leq
        C_p r \Bigg[ \mint{-}\limits_{Q^{(\lambda)}_{\varrho}(z')} |Dw| + \lambda^{2-p} (1+|Dw|)^{p-1} \, dz \Bigg],
    \end{align}
    holds for any intrinsic geometric cylinder $Q^{(\lambda)}_{\varrho}(z')\subset Q_{\varrho}(z_0)$, then there exists an ($\mathcal{A},p$)-caloric function
    \begin{equation*}
        h\in L^p\big(\Lambda_{\varrho/2}(t_0); W^{1,p}(B_{\varrho/2}(x_0),\R^N)\big),
    \end{equation*}
    i.e.
    \begin{align}\label{4.104}
        \partial_t h - \mathrm{div} \, \mathcal{A}(z,Dh) = 0
        \qquad\text{in } Q_{\varrho/2}(z_0),
    \end{align}
    such that
    \begin{align}\label{4.105}
        \mint{-}\limits_{Q_{\varrho/2}(z_0)} \bigg| \frac{h}{\varrho/2} \bigg|^p \, dz
        +
        \mint{-}\limits_{Q_{\varrho/2}(z_0)} |Dh|^p \, dz
        \leq H
    \end{align}
    and
    \begin{align}\label{4.106}
        \mint{-}\limits_{Q_{\varrho/2}(z_0)} \bigg| \frac{w-h}{\varrho/2} \bigg|^p 
        +
        \bigg| \frac{w-h}{\varrho/2} \bigg|^2 \, dz
        \leq \varepsilon
    \end{align}
    holds, where $H$ is a non-decreasing function of the arguments $n,p,L/\nu$.
\end{Lem}
In the sub-quadratic case $\frac{2n}{n+2}<p<2$, an additional $L^\infty$-$L^2$-norm bound is required to hold for $w$. This is necessary to ensure compactness in $L^2$. Keep in mind, though, that the kinds of solutions we are looking at in Theorem \ref{Theo1} always satisfy such a constraint. The proof of the following Lemma can be found in \cite[Theorem 4.5]{bogelein2013regularity}.
\begin{Lem}\label{Lem4.2}
    Let $\frac{2n}{n+2}<p<2$, $\mu\in[0,1]$, $0<\nu\leq 1\leq L$ and $C_p\geq 1$ be fixed. Then for every $\varepsilon>0$ there exits a positive function $\delta_0(n,N,p,\nu,L,\mu,C_p,\varepsilon)\leq 1$ with the following property: Let $z_0=(x_0,t_0)\in \R^{n+1}$, $\varrho>0$ and $\mathcal{A}: \R^{Nn}\to \R^{Nn}$ be a vector field satisfying $\nu\leq \gamma\leq L$. If 
    \begin{align*}
        w\in L^\infty\big(\Lambda_\varrho(t_0); L^2 (B_\varrho(x_0),\R^N)\big) \cap L^p\big(\Lambda_\varrho(t_0); W^{1,p}(B_\varrho(x_0),\R^N)\big)
    \end{align*}
    fulfills (\ref{4.101})-(\ref{4.103}) and
    \begin{align}\label{4.201}
        \sup\limits_{t\in \Lambda_\varrho(t_0)} \mint{-}\limits_{B_\varrho(x_0)} \bigg| \frac{w(\,\cdot\,,t)}{\varrho} \bigg|^2 \, dx \leq 1,
    \end{align}
    then there exists an ($\mathcal{A},p$)-caloric function
    \begin{align*}
        h\in C\big(\Lambda_{\varrho/2}(t_0); L^2(B_{\varrho/2}(x_0),\R^N)\big) \cap L^p\big(\Lambda_{\varrho/2}(t_0); W^{1,p}(B_{\varrho/2}(x_0),\R^N)\big),
    \end{align*}
    i.e.
    \begin{align}\label{4.202}
        \partial_t h - \mathrm{div} \, \mathcal{A}(Dh) = 0
        \qquad\text{in } Q_{\varrho/2}(z_0),
    \end{align}
    such that
    \begin{align}\label{4.203}
        \mint{-}\limits_{Q_{\varrho/2}(z_0)} \bigg| \frac{h}{\varrho/2} \bigg|^p \, dz
        +
        \mint{-}\limits_{Q_{\varrho/2}(z_0)} |Dh|^p \, dz
        \leq H
    \end{align}
    and
    \begin{align}\label{4.204}
        \mint{-}\limits_{Q_{\varrho/2}(z_0)} \bigg| \frac{w-h}{\varrho/2} \bigg|^p 
        +
        \bigg| \frac{w-h}{\varrho/2} \bigg|^2 \, dz
        \leq \varepsilon
    \end{align}
    holds, where $H$ is a non-decreasing function of the arguments $n,p,L/\nu$.
\end{Lem}

\section{Caccioppoli inequality and Poincaré inequalities}\label{SCacc}

In this section, we establish Caccioppoli and Poincaré inequalities suitable to our setting. To shorten the notation, we introduce the functional
\begin{equation}\label{6.03}
    \Psi_{\lambda}(z_0,\varrho,l)= \mint{-}\limits_{Q_\varrho^{(\lambda)}(z_0)} | u - l|^p \, dz
\end{equation} 
and the auxiliary function
\begin{equation}\label{5.01}
    H^*_{\lambda}(z_0,\varrho,l) = \Psi^{\frac{\beta}{p}}_{\lambda}(z_0,\varrho,l) + \varrho^\beta (1+\lambda^{2-p})^{\frac{\beta}{2}}.
\end{equation}
Then we can establish the following Caccioppoli inequality.
\begin{Lem}\label{Lem5.1}
    Let $p > \frac{2n}{n+2}$, $M\geq 1$ and assume that
    \begin{equation*}
        u\in  C\big([0,T], L^2(\Omega,\R^N)\big) \cap L^p \big(0,T,W^{1,p}(\Omega,\R^N)\big)
    \end{equation*}
    is a weak solution to (\ref{Sys1}) satisfying the growth conditions (\ref{As1})-(\ref{As3}). Then there exists a constant $C=C(p,\nu,L,\kappa_{2M})$ such that for every time independent affine function $l: \mathbb{R}^n \to \mathbb{R}^N$ with $|Dl|\leq M$ and every intrinsic geometric cylinder $Q^{(\lambda)}_\varrho(z_0)\subset \Omega_T$ with $z_0=(x_0,t_0) \in \Omega_T$ and $\varrho/2\leq r<\varrho$ and $\lambda>0$ there holds
    \begin{align*}
        &\sup_{t \in \Lambda^{\lambda}_r(t_0)} \mint{-}\limits_{B_r(x_0)} \lambda^{p-2} \bigg|\frac{u(\,\cdot\, , t)-l}{r}\bigg|^2 \, dx 
        +
        \mint{-}\limits_{Q^{(\lambda)}_r(z_0)} \big| V_{|Dl|}(Du-Dl) \big|^2 \, dz
        \\ & \leq
        C
        \mint{-}\limits_{Q^{(\lambda)}_\varrho(z_0)} \bigg| V_{|Dl|}\bigg(\frac{u-l}{\varrho-r}\bigg) \bigg|^2 + \lambda^{p-2} \bigg|\frac{u-l}{\varrho-r}\bigg|^2 \, dz
        +
        C |Dl|^p  H^*_{\lambda,p}(z_0,\varrho,l(z_0)).
    \end{align*}
\end{Lem}
\begin{proof}
    By translation, we can assume that $z_0=(0,0)$ and define for $\tau\in \R$ the cylinders $Q^{(\lambda)}_{\sigma,\tau}= Q^{(\lambda)}_\sigma \cap \{ (x,t) \in \mathbb{R}^{n+1} \,|\, t\leq \tau \}$, to shorten the notation. Since we need to apply the iteration Lemma \ref{Lem2.2} later, we consider arbitrary radii $r\leq s<\sigma\leq \varrho$ and choose, in accordance with the radii, the cut-off function $\eta\in C^\infty_0(B_\sigma, [0,1])$ in space and $\zeta\in C^1_0(\Lambda^{(\lambda)}_\sigma,[0,1])$ in time, such that $\eta \equiv 1$ on $B_s$ and $|D\eta| \leq C/(\sigma-s)$ as well as $\zeta \equiv 1$ on $\Lambda^{(\lambda)}_s$ and $|D\zeta| \leq 2\lambda^{p-2}/(\sigma^2-s^2)$. Furthermore, for $\tau\in \Lambda^{(\lambda)}_r$, we define a second cut-off function $\chi_\vartheta \in W^{1,\infty}(\mathbb{R})$ in time as follows
    \begin{equation*}
        \chi_\vartheta(t)=\begin{cases}
            1 &\text{on }(-\infty,\tau], \\ 
            1- \frac{1}{\vartheta}(t - \tau) &\text{on }(\tau,\tau+\vartheta], \\ 
            0 &\text{on }(\tau+\vartheta,\infty).
        \end{cases} 
    \end{equation*}
    To further reduce the burden of notation, we employ the auxiliary functions $v(x,t)=u(x,t)-l(x)$, $\phi_\vartheta(x,t)= \chi_\vartheta(t)\eta^p(x)\zeta(t)v(x,t)$ and $\phi(x,t)=\eta^p(x)\zeta(t)v(x,t)$. In the following, we will proceed formally since $\phi$ is not a valid test function due to the lack of differentiability in time. However, every calculation can be easily made precise by means of Steklov averages or some alternative smoothing of $u$. We immediately observe that 
    \begin{equation}\label{5.02}
        |D\phi| = \big| (Du-Dl)\eta^p \zeta + p D\eta \otimes (u-l) \eta^{p-1}\zeta \big| \leq |Du-Dl| + C \bigg|\frac{u-l}{\sigma-s}\bigg|.
    \end{equation}
    If we now choose $D\phi_\vartheta$ as the test function in (\ref{weakfrom}), we obtain that
    \begin{align*}
        \int\limits_{Q^{(\lambda)}_\sigma} u \cdot \partial_t \phi_\vartheta - a(z,u,Du)\cdot D\phi_\vartheta \, dz=0
    \end{align*}
    We start with the treatment of the term that involves the time derivative. Thereby, we utilize the fact that the affine function does not depend on time and obtain after (formally) integrating by parts twice in the limit $\vartheta\to 0$ that
    \begin{align}\label{5.04}
        \int\limits_{Q^{(\lambda)}_\sigma} u \cdot \partial_t \phi_\vartheta \, dz 
        &=
        \int\limits_{Q^{(\lambda)}_\sigma} (u-l) \cdot \partial_t \phi_\vartheta \, dz
        = \frac{1}{2}
        \int\limits_{Q^{(\lambda)}_\sigma} |v|^2 \eta^p \partial_t( \chi_\vartheta \zeta) \, dz
        \nonumber\\ &
        \underset{\vartheta\to 0}{\longrightarrow}
        \frac{1}{2}
        \int\limits_{Q^{(\lambda)}_\sigma} |v|^2 \eta^p \partial_t \zeta \, dz
        -
        \frac{1}{2}
        \int\limits_{B_\sigma} |v(\,\cdot\, , \tau)|^2 \eta^p \zeta(\tau) \, dx
        \nonumber\\ & \leq
        \int\limits_{Q^{(\lambda)}_{\sigma,\tau}} \lambda^{p-2} \frac{|v|^2 }{\sigma^2-s^2} \, dz
        -
        \frac{1}{2}
        \int\limits_{B_s} |v(\,\cdot\, , \tau)|^2 \, dz,
    \end{align}
    where we used the properties of $\eta$ and $\zeta$ for the latter inequality. In the limit $\vartheta\to 0$, we can therefore combine the previous two estimates to obtain
    \begin{align}\label{5.03}
        \frac{1}{2}
        \int\limits_{B_s} |v(\,\cdot\, , \tau)|^2 \, dz & 
        +
        \int\limits_{Q^{(\lambda)}_{\sigma,\tau}}a(z,u,Du) \cdot D\phi \, dz
        \leq
        \int\limits_{Q^{(\lambda)}_{\sigma,\tau}} \lambda^{p-2} \frac{|v|^2 }{\sigma^2-s^2} \, dz.
    \end{align}
    Now let us continue by treating the diffusion term. To this end, we note that 
    \begin{equation*}
        \int\limits_{Q^{(\lambda)}_{\sigma,\tau}}a(0,l(0),Dl) \cdot D\phi \, dz =0,
    \end{equation*}
    which implies in combination with (\ref{5.02}) that 
    \begin{align}\label{5.10}
        \frac{1}{2}
        \int\limits_{B_s} & |v(\,\cdot\, , \tau)|^2 \, dz
        +
        \int\limits_{Q^{(\lambda)}_{\sigma,\tau}} \big(a(z,u,Du)-a(z,u,Dl) \big) \cdot (Du-Dl) \eta^p \zeta \, dz
        \nonumber\\ & \leq
        \int\limits_{Q^{(\lambda)}_{\sigma,\tau}} \lambda^{p-2} \frac{|v|^2 }{\sigma^2-s^2} \, dz +
        \int\limits_{Q^{(\lambda)}_{\sigma,\tau}} \big( a(0,l(0),Dl) - a(z,u,Dl) \big)\cdot D\phi \, dz 
        \nonumber\\ & \qquad- p
        \int\limits_{Q^{(\lambda)}_{\sigma,\tau}} \big( a(z,u,Du) - a(z,u,Dl) \big)\cdot (D\eta \otimes (u-l)) \eta^{p-1} \zeta \, dz 
        =T + I + II
    \end{align}
    with the obvious definitions of the integrals $T$, $I$ and $II$. Now, we estimate the first integral $I$ on the left-hand side. To this end, we use (\ref{As2}) and Lemma \ref{Lem2.5} in the sub-quadratic case and a simple translation argument in the super-quadratic case to obtain
    \begin{align*}
        \int\limits_{Q^{(\lambda)}_{\sigma,\tau}} & \big(a(z,u,Du)-a(z,u,Dl) \big) \cdot (Du-Dl) \eta^p \zeta \, dz
        \\ & =
        \int\limits_{Q^{(\lambda)}_{\sigma,\tau}} \int\limits_0^1 \big(D_\xi a(z,u,Dl+s(Du-Dl) ) \big)(Du-Dl,Du-Dl) \, ds \, \eta^p \zeta \, dz
        \\ & \geq 
        \nu \int\limits_{Q^{(\lambda)}_{\sigma,\tau}} \int\limits_0^1|Dl+s(Du-Dl)|^{p-2} |Du-Dl|^2 \, ds \, \eta^p \zeta \, dz
        \\ & \geq 
        \frac{\nu}{C(p)} \int\limits_{Q^{(\lambda)}_{\sigma,\tau}} \big| V_{|Dl|}(Du- Dl) \big|^2 \eta^p \zeta \, dz
        \geq 
        \frac{\nu}{C(p)} \int\limits_{Q^{(\lambda)}_{s,\tau}} \big| V_{|Dl|}(Du- Dl) \big|^2 \, dz,
    \end{align*}
    where we used that $\eta^p\zeta=1$ on $Q^{(\lambda)}_{s,\tau}$. We now continue with the estimates for $I$ and $II$. To treat $I$, we use the assistance of the function
    \begin{equation*}
        f(x,t) :=\min\big\{ (|x|^2 + |t|)^{\frac{\beta}{2}} + |u(x,t)-l(0)|^\beta , 1 \big\},
    \end{equation*}
    assumption (\ref{ModConu}) and split the domain of integration into
    \begin{equation*}
        S_1:= \{ z\in Q^{(\lambda)}_{\sigma,\tau} \,|\, |D\phi|(z)\leq |Dl| \} 
        \qquad\text{and}\qquad
        S_2 := Q^{(\lambda)}_{\sigma,\tau}\setminus S_1
    \end{equation*}
    to deduce with Young's inequality and $f\leq 1$ that
    \begin{align*}
        I &= \int\limits_{Q^{(\lambda)}_{\sigma,\tau}} \big( a(0,l(0),Dl) - a(z,u,Dl) \big)\cdot D\phi \, dz 
        \leq
        L \int\limits_{Q^{(\lambda)}_{\sigma,\tau}} f |Dl|^{p-1} |D\phi| \, dz 
        \nonumber\\ & =
        L \int\limits_{S_1} f |Dl|^{p-1} |D\phi| \, dz 
        +
        L \int\limits_{S_2} f |Dl|^{p-1} |D\phi| \, dz 
        \leq
        C(\mathfrak{C},\varepsilon) |Dl|^p \int\limits_{Q^{(\lambda)}_\sigma} f \, dz 
        +
        \varepsilon \int\limits_{S_2} |D\phi|^{p} \, dz,
    \end{align*}
     for some $\varepsilon\in (0,1)$ that will be chosen later. To further estimate the first integral on the right-hand side, we use Young's inequality, which yields that
    \begin{align*}
        \int\limits_{Q^{(\lambda)}_\sigma} f \, dz
        &\leq
        \int\limits_{Q^{(\lambda)}_\varrho} f \, dz
        =
        |Q^{(\lambda)}_\varrho| \mint{-}\limits_{Q^{(\lambda)}_\varrho} f \, dz
        \leq
        |Q^{(\lambda)}_\varrho| \mint{-}\limits_{Q^{(\lambda)}_{\varrho}} \big( (|x|^2 + |t|)^{\frac{\beta}{2}} + |u-l(0)|^\beta \big) \, dz
        \nonumber\\ & \leq
        |Q^{(\lambda)}_\varrho| \Bigg[\varrho^\beta (1+\lambda^{2-p})^{\frac{\beta}{2}} 
        + 
        \Bigg( \mint{-}\limits_{Q^{(\lambda)}_\varrho} |u-l(0)|^p  \, dz \Bigg)^{\frac{\beta}{p}} \Bigg]
        =
        |Q^{(\lambda)}_\varrho| H^*_{\lambda,p}.
    \end{align*}
    To estimate $II$, we start with the observation that the integrand vanishes on $Q^{(\lambda)}_{s,\tau}$ by choice of the cut-off functions and decompose the domain of integration $Q^{(\lambda)}_{\sigma,\tau} \setminus Q^{(\lambda)}_{s,\tau}$ into $S_3 \cup S_4$, where $S_3$ and $S_4$ are defined by
    \begin{align*}
        S_3 = \{ z\in Q^{(\lambda)}_{\sigma,\tau} \setminus Q^{(\lambda)}_{s,\tau} \,|\, |Du-Dl| < |Dl| \}
        \qquad\text{and}\qquad
        S_4 =\big(Q^{(\lambda)}_{\sigma,\tau} \setminus Q^{(\lambda)}_{s,\tau}\big) \setminus S_3.
    \end{align*}
    Let us now start with the estimation of the integral restricted to the domain of integration $S_3$. We use the growth condition (\ref{As3}) for $Da(\, \cdot\, )$ and that $|Dl|\leq M$ implies $|Dl + \mu Du-Dl| \leq 2 |Dl| \leq 2M$ for every $0\leq \mu \leq 1$ on $S_3$. We obtain with an application of Lemma \ref{Lem2.5} in the case $p<2$ and (\ref{2.11}) that
    \begin{align*}
        \Bigg| \int\limits_{S_3} & \big( a(z,u,Du) - a(z,u,Dl) \big)\cdot ( D\eta \otimes (u-l) ) \eta^{p-1} \zeta \, dz  \Bigg|
        \\ & =
        \Bigg| \int\limits_{S_3} \int\limits_0^1 \big\langle (D_\xi a )(z,u,Dl+\mu (Du-Dl)) \cdot (Du-Dl),D\eta \otimes (u-l) \big\rangle \, d\mu dz \Bigg|
        \\ & \leq
        C \int\limits_{S_3} \int\limits_0^1 \big|(D_\xi a )(z,u,Dl+\mu(Du-Dl))\big| \, d\mu \, |Du-Dl| \, \bigg|\frac{u-l }{\sigma-s}\bigg| \, dz
        \\ & \leq
        C  L \kappa_{2M}
        \int\limits_{S_3} \int\limits_0^1 |Dl + \mu (Du-Dl) |^{p-2} \, d\mu |Du-Dl| \, \bigg|\frac{u-l }{\sigma-s}\bigg| \, dz
        \\ & \leq 
        C L \kappa_{2M}
        \int\limits_{S_3} \big( |Dl|^2 + |Du-Dl|^{2} \big)^{\frac{p-2}{2}} |Du-Dl| \, \bigg|\frac{u-l }{\sigma-s}\bigg| \, dz
        \\ & \leq C L \kappa_{2M}
        \int\limits_{Q^{(\lambda)}_{\sigma,\tau} \setminus Q^{(\lambda)}_{s,\tau}} \big| V_{|Dl|}(Du - Dl) \big|^2 + \bigg| V_{|Dl|}\bigg(\frac{u-l}{\sigma-s} \bigg) \bigg|^2\, dz.
    \end{align*}
    Since the argument of $(D_\xi a)(z,u,\,\cdot\,)$ could be $0$, we need to justify this calculation in the sub-quadratic case $p<2$. However, this can be done by a standard argument, which can be found in \cite[Remark 5.2]{bogelein2013regularity}. Now, we consider the domain of integration $S_4$. To this end, we will use the growth conditions (\ref{Asa})\textsubscript{1} on the vector field $a(z,u,\,\cdot\,)$ and the inequality $|Du| \leq |Dl| + |Du-Dl|\leq 2|Du-Dl|$, which holds on $S_4$. The equality $Du=Dl + D\phi + D\psi$ followed by Young's inequality with exponents $(p,p/(p-1))$ then allow us to estimate that
    \begin{align*}
        \Bigg|\int\limits_{S_4} & \big( a(z,u,Du) - a(z,u,Dl) \big)\cdot D\eta \otimes (u-l) \eta^{p-1} \zeta \, dz \Bigg|
        \\ & \leq
        C \int\limits_{S_4} |Du-Dl|^{p-1} \bigg| \frac{u-l}{\sigma-s} \bigg| \, dz
        \leq
        C
        \int\limits_{S_4} |Du-Dl|^{p} +  \bigg| \frac{u-l}{\sigma-s} \bigg|^{p} \, dz.
    \end{align*}
    To estimate the remaining integrals over $S_2$ and $S_4$ further, we distinguish again between the super-quadratic case and the sub-quadratic case. In the former case, we apply the upper bound of (\ref{5.02}) and note that we have
    \begin{align*}
        |Du- Dl|^p + \bigg|\frac{u-l}{\sigma-s} \bigg|^p
        \leq
        \big| V_{|Dl|}(Du- Dl) \big|^2 + \bigg| V_{|Dl|}\bigg(\frac{u-l}{\sigma-s} \bigg) \bigg|^2,
    \end{align*}
    while the latter case $p<2$ necessitates more care. There, we choose
    \begin{equation*}
        A \in \bigg\{Du-Dl, \frac{u-l}{\sigma-s} \bigg\}
        \qquad\text{such that}\qquad
        A= \max \bigg\{|Du-Dl|, \bigg|\frac{u-l}{\sigma-s}\bigg| \bigg\}.
    \end{equation*}
    and observe that $|D\phi| \geq |Dl|$ on $S_2$ with (\ref{5.02}) as well as $|Du-Dl|\geq |Dl|$ on $S_4$ imply $|A| \geq \frac{1}{C} |Dl|$. Therefore, we obtain with inequality (\ref{2.14}) the estimate
    \begin{equation*}
        |A|^p \leq C \big|V_{|Dl|}(|A|) \big|^2.
    \end{equation*}
    By the non-negativity of the involved terms, this yields on $S_2$ and $S_4$ the chain of inequalities
    \begin{equation*}
        |Du-Dl|^p + \bigg|\frac{u-l}{\sigma-s}\bigg|^p \leq 2 |A|^p \leq  C \big|V_{|Dl|}(|A|) \big|^2 \leq C \big| V_{|Dl|}(Du- Dl) \big|^2 + \bigg| V_{|Dl|}\bigg(\frac{u-l}{\sigma-s} \bigg) \bigg|^2.
    \end{equation*}
    As an immediate consequence, we infer
    \begin{align*}
        \varepsilon & \int\limits_{S_2} |D\phi|^{p} \, dz 
        \leq
        \varepsilon C
        \int\limits_{Q^{(\lambda)}_{\sigma,\tau}} \big| V_{|Dl|}(Du- Dl) \big|^2 + \bigg| V_{|Dl|}\bigg(\frac{u-l}{\sigma-s} \bigg) \bigg|^2 \, dz
        \\ & \leq
        \varepsilon C
        \int\limits_{Q^{(\lambda)}_{s,\tau}} \big| V_{|Dl|}(Du- Dl) \big|^2 
        +
        C
        \int\limits_{Q^{(\lambda)}_{\sigma,\tau} \setminus Q^{(\lambda)}_{s,\tau}} \big| V_{|Dl|}(Du- Dl) \big|^2
        +
        C
        \int\limits_{Q^{(\lambda)}_{\sigma,\tau}} \bigg| V_{|Dl|}\bigg(\frac{u-l}{\sigma-s} \bigg) \bigg|^2 \, dz
    \end{align*}
    and 
    \begin{align*}
        \int\limits_{S_4} |Du-Dl|^{p} +  \bigg| \frac{u-l}{\sigma-s} \bigg|^{p} \, dz
        \leq
        C
        \int\limits_{Q^{(\lambda)}_{\sigma,\tau} \setminus Q^{(\lambda)}_{s,\tau}} \big| V_{|Dl|}(Du- Dl) \big|^2
        +
        C
        \int\limits_{Q^{(\lambda)}_{\sigma,\tau}} \bigg| V_{|Dl|}\bigg(\frac{u-l}{\sigma-s} \bigg) \bigg|^2 \, dz.
    \end{align*}
    Collecting the previous estimates in inequality (\ref{5.10}) and choosing $\varepsilon=\frac{\nu}{2C}$ for $C$ large enough yields after absorption that
    \begin{align*}
        \frac{1}{2} \int\limits_{B_s} & |v(\,\cdot\,,\tau)|^2 \, dx
        +
        \frac{\nu}{2C(p)}
        \int\limits_{ Q^{(\lambda)}_{s,\tau}} \big| V_{|Dl|}(Du- Dl) \big|^2 \, dz
        \leq C
        \int\limits_{Q^{(\lambda)}_{\sigma,\tau} \setminus Q^{(\lambda)}_{s,\tau}} \big| V_{|Dl|}(Du- Dl) \big|^2 \, dz 
        \\ & \qquad + C
        \int\limits_{Q^{(\lambda)}_{\sigma,\tau}}
        \bigg| V_{|Dl|}\bigg(\frac{u-l}{\sigma-s} \bigg) \bigg|^2 
        + \lambda^{p-2} \frac{|v|^2 }{\sigma^2-s^2} 
        \, dz
        + C|Dl|^{p} |Q^{(\lambda)}_{\varrho}| H^*_{\lambda}
        ,
    \end{align*}
    where the constant $C$ is the same on both sides of the inequality and we have in terms of dependencies $C=C(n,p,\nu,L,\kappa_{2M})$. Since we want to iterate away the first term on the right-hand side using Lemma \ref{Lem2.2}, we add
    \begin{equation*}
        CL
        \int\limits_{ Q^{(\lambda)}_{s,\tau}} \big| V_{|Dl|}(Du- Dl) \big|^2 \, dz 
    \end{equation*}
    to both sides of the previous equation and divide by $\frac{\nu}{2C(p)} + CL$. This results in
    \begin{align*}
        \frac{1}{2(\nu/(2C(p)) + C L)} & \int\limits_{B_s} |v(\,\cdot\,,\tau)|^2  \, dx
        +
        \int\limits_{ Q^{(\lambda)}_{s,\tau}} \big| V_{|Dl|}(Du- Dl) \big|^2 \, dz
        \leq \vartheta
        \int\limits_{Q^{(\lambda)}_{\sigma,\tau} } \big| V_{|Dl|}(Du- Dl) \big|^2 \, dz 
        \\ & +
        C \int\limits_{Q^{(\lambda)}_{\sigma,\tau}}
        \bigg| V_{|Dl|}\bigg(\frac{u-l}{\sigma-s} \bigg) \bigg|^2 
        + \lambda^{p-2} \frac{|v|^2 }{\sigma^2-s^2} 
        \, dz
        + C|Dl|^{p} |Q^{(\lambda)}_{\varrho}| H^*_{\lambda},
    \end{align*}
    with $\vartheta = \frac{CL}{\nu/(2C(p)) + CL}<1$. Since the radii $s<\sigma$ were not further specified in the set-up of the proof, we observe that this inequality holds for all $r\leq s < \sigma \leq \varrho$. Therefore, we can apply the iteration Lemma \ref{Lem2.2} and afterwards take the supremum over $\tau \in \Lambda^{(\lambda)}_r$ to obtain the estimate for the "sup"-term in the Caccioppoli inequality and $\tau=\lambda^{2-p}r^2$ to deduce the estimate for the second term in the Caccoppoli inequality. This yields after dividing by $|Q^{(\lambda)}_r|$ that
    \begin{align*}
        \sup_{t \in \Lambda^{\lambda}_r} & \mint{-}\limits_{B_r} \lambda^{p-2} \bigg|\frac{u(\,\cdot\, , t)-l}{r^2}\bigg|^2\, dx 
        +
        \mint{-}\limits_{Q^{(\lambda)}_r} \big| V_{|Dl|}(Du-Dl) \big|^2 \, dz
        \\ & \leq
        C
        \mint{-}\limits_{Q^{(\lambda)}_\varrho} \bigg| V_{|Dl|}\bigg(\frac{u-l}{\varrho-r}\bigg) \bigg|^2 + \lambda^{p-2} \bigg|\frac{u-l}{\varrho-r}\bigg|^2 \, dz
        + C|Dl|^{p} H^*_{\lambda}.
    \end{align*}
    This concludes the proof.
\end{proof}
The intermediate Poincaré inequality for solutions to general parabolic systems will be used in the following. For the proof we refer to \cite[Lemma 3.1]{bogelein2013regularity}.
\begin{Lem}\label{Lem3.1}
    Let $p>1$, $t_1<t_2$ and $U\subset \mathbb{R}^n$. Suppose that the functions $\xi\in L^1\big(U\times(t_1,t_2), \mathbb{R}^N\big)$ and $w\in L^p\big(t_1,t_2;W^{1,p}\big(U,\mathbb{R}^N\big)\big)$ satisfy
    \begin{equation}\label{3.10}
        \int\limits_{U\times(t_1,t_2)} w \cdot \partial_t \phi - \xi \cdot D\phi \, dz =0 
        \qquad
        \forall \phi \in C^\infty_0\big(U\times(t_1,t_2),\mathbb{R}^N\big).
    \end{equation}
    Then for $r,\lambda>0$ and $q\in [1,p]$, there exists a constant $C=C(n,N,q)$ such that for any intrinsic geometric cylinder $Q^{(\lambda)}_{r}(z_0)\in U\times(t_1,t_2)$ there holds
    \begin{equation*}
        \mint{-}\limits_{Q^{(\lambda)}_{r}(z_0)} \big|w- (w)_{z_0;r}^{(\lambda)} \big|^q \, dz 
        \leq
        C r^q \Bigg[ \mint{-}\limits_{Q^{(\lambda)}_{r}(z_0)} |Dw|^q \, dz + \bigg( \lambda^{2-p} \mint{-}\limits_{Q^{(\lambda)}_{r}(z_0)} |\xi| \, dz \bigg)^q \Bigg].
    \end{equation*}
\end{Lem}
Since the previous Poincaré inequality holds irrespective of the structure of $\xi$, we can obtain a slightly stronger result, if we use the vector field $a$ explicitly. This is done in
\begin{Lem}\label{Lem5.3}
    Let $p> \frac{2n}{n+2}$ and $M\geq 1$. Suppose that
    \begin{equation*}
        u\in  C\big([0,T], L^2(\Omega,\R^N)\big) \cap L^p \big(0,T,W^{1,p}\big(\Omega,\R^N\big)\big)
    \end{equation*}
    is a weak solution to (\ref{Sys1}) fulfilling the structural assumptions (\ref{As1})-(\ref{As3}) and (\ref{ModConu}). Then, there exists a constant $C\geq 1$ such that for any $A\in \R^{Nn}$ with $|A|\leq M$, any intrinsic geometric cylinder $Q_\varrho^{(\lambda)}(z_0) \subset \Omega_T$ with center $z_0=(x_0,t_0)\in \Omega_T$ and parameters $\lambda,\varrho>0$ and any $1\leq q\leq s$ such that $|Du|\in L^q\big( Q_\varrho^{(\lambda)}(z_0)\big)$ there holds:
    \begin{align}\label{5.301}
        \mint{-}\limits_{Q_\varrho^{(\lambda)}(z_0)} & \big| u - (u)_{z_0;\varrho}^{(\lambda)} - A(x-x_0) \big|^q \, dz
        \leq
        C \varrho^q \Bigg[ \mint{-}\limits_{Q^{(\lambda)}_{\varrho}(z_0)} |Du-A|^q \, dz 
        +
        \Bigg( \lambda^{2-p} \mint{-}\limits_{Q^{(\lambda)}_{\varrho}(z_0)} |Du-A|^{p-1} 
        \nonumber \\ & \qquad + \mathds{1}_{\{p\geq 2\}} |A|^{p-2}|Du-A|  \, dz 
        + 
        \lambda^{2-p} |A|^{p-1} H^*_\lambda\big(z_0,\varrho,(u)_{z_0;\varrho}^{(\lambda)}\big)\Bigg)^q  \Bigg].
    \end{align}
    In the case $\frac{2n}{n+2}<p<2$ there additionally holds if $A\neq 0$ that
    \begin{align}\label{5.303}
        \mint{-}\limits_{Q_\varrho^{(\lambda)}(z_0)} \big| u & - (u)_{z_0;\varrho}^{(\lambda)} - A(x-x_0) \big|^q \, dz
        \leq
        C \varrho^q \bigg( 1 + \frac{\lambda^{q(2-p)}}{|A|^{q(2-p)}} \bigg) \Bigg] 
        \Bigg[ \mint{-}\limits_{Q^{(\lambda)}_{\varrho}(z_0)} |Du-A|^q \, dz + 
        \nonumber \\ & \qquad +
        \lambda^{q(2-p)}|A|^{q(p-1)} \big( H^*_\lambda\big(z_0,\varrho,(u)_{z_0;\varrho}^{(\lambda)}\big)\big)^q \Bigg].
    \end{align}
    Note, that in both cases the constant is of the form $C=C(n,N,p,s,\kappa_{2M}) L$.
\end{Lem}
\begin{proof}
    We again assume, without loss of generality, that $z_0=0$ and abbreviate $Q=Q_\varrho^{(\lambda)}$. From the weak formulation (\ref{weakfrom}) and the fact that
    \begin{equation*}
        \int\limits_{\Omega_T}\big( (u)_{\varrho}^{(\lambda)} + Ax \big) \cdot \partial_t\phi \, dz =0
        \qquad\text{and}\qquad
        \int\limits_{\Omega_T} a\big(0,(u)_{\varrho}^{(\lambda)}, A \big) \cdot D\phi \, dz = 0 
        \qquad\forall \phi\in C_0^\infty \big(\Omega_T , \mathbb{R}^N \big),
    \end{equation*}
    we observe in a first step
    \begin{align*}
        \int\limits_{\Omega_T} \big(u - (u)_{\varrho}^{(\lambda)} -Ax\big) \cdot \partial_t\phi & - \big( a(x,u,Du) -  a\big(0,(u)_{\varrho}^{(\lambda)}, A \big) \big) \cdot D\phi \, dz
        = 0.
    \end{align*}
    Therefore, we are in the position to apply Lemma \ref{Lem3.1} to $w=u -(u)_{\varrho}^{(\lambda)} -Ax$ with the particular choice $\xi = a(z,u,Du) -   a(0,(u)_{\varrho}^{(\lambda)}, A )$. By assumption, we have that 
    \begin{equation*}
        w\in L^q\big(-\lambda^{2-p}\varrho^p,\lambda^{2-p}\varrho^p,W^{1,q}\big(B_\varrho,\R^N\big)\big),
    \end{equation*}
    which implies that we can choose for the parameter $q$ of Lemma \ref{Lem3.1} the growth exponent $p$. Using that $u-(u)_{\varrho}^{(\lambda)}-Ax$ has mean value zero with respect to $Q$, we thus obtain
    \begin{align}\label{5.34}
        \mint{-}\limits_{Q} & \big| u - (u)_{\varrho}^{(\lambda)} - Ax \big|^q \, dz 
        \leq
        C \varrho^q \Bigg[ \mint{-}\limits_{Q} |Du-A|^q \, dz + \bigg( \lambda^{2-p} \mint{-}\limits_{Q} \big|a(z,u,Du) -  a\big(0,(u)_{\varrho}^{(\lambda)}, A \big) \big| \, dz \bigg)^q \Bigg].
    \end{align}
    Note that the constant $C$ can be modified so that it depends only on $n,N$ and $s$, since $q\in [1,s]$ and the constant of the Poincaré inequality of Lemma \ref{Lem3.1} is continuous with respect to the exponent $q$. It remains to estimate the second integral on the right-hand side of integral (\ref{5.34}). To this end, we treat the following integrals separately:
    \begin{align}
        \mint{-}\limits_{Q} & \big|a(z,u,Du) -  a\big(0,(u)_{\varrho}^{(\lambda)}, A \big) \big| \, dz
        \leq
        \mint{-}\limits_{Q} |a(z,u,Du) - a(z,u,A) |  \, dz
        \nonumber\\& \qquad + \mint{-}\limits_{Q} \big| a\big(z,u,A \big) - a\big(0,(u)_{\varrho}^{(\lambda)}, A \big) \big| \, dz =: J_1 + J_2,
    \end{align}
    with the obvious definitions of $J_1$ and $J_2$. We start by estimating $J_1$. To achieve this, we split the domain of integration into two sets:
    \begin{align*}
        &S_1:= \{ z\in Q \,|\, |Du(z)-A|<|A| \} \\
        &S_2:= \{ z\in Q \,|\, |Du(z)-A| \geq |A| \}.
    \end{align*}
    Let us start by estimating the integrand of $J_1$ on the set $S_2$. We obtain with Lemma \ref{Lem2.7} that
    \begin{align}\label{5.36}
        \frac{1}{|Q|}\int\limits_{S_2} |a(z,u,Du)  -  a(z,u, A ) | \, dz
         &\leq
        \frac{ C L }{|Q|}\int\limits_{S_2}  |Du-A|^{p-1}
        \leq
        C(p,\kappa_{2M}) L  \mint{-}\limits_{Q}  |Du-A|^{p-1} \, dz.
    \end{align}
    In order to continue with the domain of integration $S_1$ for $J_1$, we need to distinguish the sub- and super-quadratic cases to estimate the first integral on the right hand side. In the \textbf{super-quadratic case} $p\geq 2$, we use the growth assumptions (\ref{As3}) on $a$, $|Du| \leq |Du-A| + |A|\leq 2|A| \leq 2M$, and Jensen's inequality. This yields
    \begin{align}\label{5.37}
        & \frac{1}{|Q|} \int\limits_{S_1} \big| a(z,u,Du)  -  a(z,u, A) \big| \, dz \leq \frac{1}{|Q|}\int\limits_{S_1}   \int\limits_0^1 |D_\xi a|(z,u,A - s( Du -A))  |Du-A| \, ds
        \nonumber\\ &\leq
        \frac{1}{|Q|}\int\limits_{S_1}C L \kappa_{2M} \big( |A|^2 + |Du-A|^2 \big)^{\frac{p-2}{2}} |Du-A| \, dz
        \leq
        C(p,\kappa_{2M}) L |A|^{p-2} \mint{-}\limits_{Q} |Du-A| \, dz.
    \end{align}
    In the \textbf{sub-quadratic case} $p<2$, we perform a similar calculation to (\ref{5.37}). This time, we need to apply inequality (\ref{e2.6}), which yields
    \begin{align}\label{5.38}
        & \frac{1}{|Q|} \int\limits_{S_1} \big| a(z,u,Du)  -  a(z,u, A)) \big| \, dz \leq \frac{1}{|Q|}\int\limits_{S_1}   \int\limits_0^1 |D_\xi a|(z,u,A - s( Du -A))  |Du-A| \, ds
        \nonumber\\ &\leq
        C(p,\kappa_{2M}) L \frac{1}{|Q|}\int\limits_{S_1} \big( |A|^2 + |Du-A|^2 \big)^{\frac{p-2}{2}} |Du-A| \, dz.
    \end{align}
    We further estimate the right-hand side using the definition of $S_1$ as follows: The first inequality (\ref{5.301}) for the sub-quadratic is obtained by estimating
    \begin{equation}\label{5.39}
        \int\limits_{S_1} \big( |A|^2 + |Du-A|^2 \big)^{\frac{p-2}{2}} |Du-A| \, dz
        \leq
        \int\limits_{Q} |Du-A|^{p-1} \, dz.
    \end{equation}
    For (\ref{5.303}), we proceed similarly, using that $|A|>0$ implies 
    \begin{equation}\label{5.391}
        \int\limits_{S_1} \big( |A|^2 + |Du-A|^2 \big)^{\frac{p-2}{2}} |Du-A| \, dz
        \leq
        \int\limits_{Q} |A|^{p-2}|Du-A| \, dz
    \end{equation}
    and
    \begin{equation}\label{5.392}
        \int\limits_{S_2}  |Du-A|^{p-1} \, dz 
        \leq
        \int\limits_{S_2}  |A|^{p-2}|Du-A| \, dz
        \leq
        \int\limits_{Q}  |A|^{p-2}|Du-A| \, dz.
    \end{equation}
    This concludes the estimation for $J_1$. Let us resume the proof with the estimation of $J_2$. We obtain with the Hölder-continuity assumption (\ref{ModConu}) on $a$ and Hölder's inequality with exponents $\left(\frac{p}{\beta}, \frac{p}{p-\beta} \right)$ 
    \begin{align}\label{5.393}
        J_2 & \leq
        L |A|^{p-1} \mint{-}\limits_{Q} \big|u- (u)_\varrho^{(\lambda)} \big|^\beta \, dz + C L \varrho^\beta (1+\lambda^{2-p})^{\frac{\beta}{2}} |A|^{p-1}
        \nonumber \\ & \leq
        C(p) L |A|^{p-1} \Bigg[ \bigg( \mint{-}\limits_{Q}  \big|u- (u)_\varrho^{(\lambda)} \big|^{p} \, dz \bigg)^{\frac{\beta}{p}} + \varrho^\beta ( 1+\lambda^{2-p})^{\frac{\beta}{2}}  \Bigg]
    \end{align}
    Inserting (\ref{5.36}), (\ref{5.37}) and (\ref{5.393}) into (\ref{5.34}) yields the desired inequality (\ref{5.301}) of the Lemma in the super-quadratic case. To obtain (\ref{5.301}) in the sub-quadratic case, we estimate (\ref{5.34}) using (\ref{5.393}), (\ref{5.36}) and (\ref{5.37}) in combination with (\ref{5.39}). For the derivation of (\ref{5.303}) we proceed accordingly, with the minor change that we use (\ref{5.391}) as well as (\ref{5.392}) instead of (\ref{5.39}) and apply Minkowski's inequality followed by Young's inequality to the right-hand side.
\end{proof}
\begin{remark}
    If we use $A=0$ in the preceding inequality (\ref{5.301}), the terms derived from the continuity assumption on $(x,u)\mapsto a(x,u,\,\cdot\,)$ drop out. Using the obtained inequality, we can further estimate the mean oscillation of $u$ on the right-hand side of (\ref{5.301}). Moreover, if the vector field is independent of the space-time variable $z$ and the solution $u$, the $H_\lambda^*$ term can be ignored. This will come in handy for upcoming chapters.
\end{remark}
To conclude this chapter, we will prove the following Sobolev-Poincaré type inequality for the sub-quadratic case $p<2$, which will be used later.
\begin{Lem}\label{Lem5.4}
    Let $\frac{2n}{n+2}<p<2$ and $M,K\geq 1$. Suppose that
    \begin{equation*}
        u\in  C\big([0,T], L^2(\Omega,\R^N)\big) \cap L^p \big(0,T,W^{1,p}(\Omega,\R^N)\big)
    \end{equation*}
    is a weak solution to (\ref{Sys1}) fulfilling the assumptions (\ref{As1})-(\ref{As3}) and (\ref{ModConu}). Then, for any intrinsic geometric cylinder $Q_\varrho^{(\lambda)}(z_0)\subset \Omega_T$ with center $z_0=(x_0,t_0)\in \Omega_T$ and parameters $\varrho,\lambda >0$ and for any $A\in \mathbb{R}^{Nn}$ with
    \begin{equation}\label{5.401}
        |A| \leq M \qquad\text{and}\qquad \frac{\lambda}{K}\leq |A| 
    \end{equation}
    there holds:
    \begin{align*}
        \mint{-}\limits_{Q_{\frac\varrho 2}^{(\lambda)}(z_0)} \big|u - (u)_{z_0,\frac \varrho2}^{(\lambda)} - A(x-x_0) \big|^2 \, dz
        &\leq C \varrho^2
        \Bigg[ \lambda^{2-p} |A|^p H^*_\lambda\big(0,\varrho,(u)_{z_0;\varrho}^{(\lambda)}\big) + \Bigg( \mint{-}\limits_{Q^{(\lambda)}_{\varrho}(z_0)} |Du-A|^p \, dz 
        \\ & \qquad + 
        |A|^{p} \big( H^*_\lambda\big(0,\varrho,(u)_{z_0;\varrho}^{(\lambda)}\big) \big)^p \Bigg)^{\frac{2}{p}}\Bigg],
    \end{align*}
    where $C=C(\mathfrak{C},K,\kappa_{2M})$.
\end{Lem}
\begin{proof}
    We again assume, without loss of generality, that $z_0=0$. Furthermore, choose $\frac\varrho2 \leq r < R \leq \varrho$ and abbreviate $v_r(x,t):= u - (u)_r^{(\lambda)} - Ax$ as well as $H^*_\lambda(r,\varrho)=H^*_\lambda(0,r,(u)_{z_0;\varrho}^{(\lambda)})$. Now we apply the Gagliardo-Nirenberg inequality (\ref{Lem2.11}) "slicewise" to $v_r(\,\cdot\,,t)$, replacing $(q,p,r,\vartheta)$ with $(2,p,2,p/2)$. After integrating over $t\in \Lambda_r^{(\lambda)}$ and dividing by $2\lambda^{2-p}\varrho^2$, this results in
    \begin{align}\label{5.41}
        \mint{-}\limits_{Q_r^{(\lambda)}} \bigg| \frac{v_r}{r} \bigg|^2 \, dz
        & \leq C(n) 
        \mint{-}\limits_{\Lambda_r^{(\lambda)}} \mint{-}\limits_{B_r} \bigg| \frac{v_r}{r} \bigg|^p + |Dv_r|^p \, dx 
        \Bigg[ \mint{-}\limits_{B_r} \bigg| \frac{v_r}{r} \bigg|^2 \, dx \Bigg]^{1-\frac{p}{2}} \, dt
        \nonumber \\ & \leq C(n)
        \mint{-}\limits_{Q_r^{(\lambda)}} \bigg| \frac{v_r}{r} \bigg|^p + |Dv_r|^p \, dz 
        \cdot
        \Bigg[ \sup\limits_{t\in \Lambda_r^{(\lambda)}} \mint{-}\limits_{B_r} \bigg| \frac{v_r(\,\cdot\,,t)}{r} \bigg|^2 \, dx \Bigg]^{1-\frac{p}{2}}.
    \end{align}
    If we use the second assumption in (\ref{5.401}) in the Poincaré inequality (\ref{5.303}) of Lemma \ref{Lem5.3} with the parameter choice $q=p$, we obtain
    \begin{align}\label{5.42}
        \mint{-}\limits_{Q^{(\lambda)}_{r}} \bigg| \frac{v_r}{r} \bigg|^p \, dz
        &\leq
        C  \big( 1 + K^{p(2-p)} \big) \Bigg[ \mint{-}\limits_{Q^{(\lambda)}_{r}} |Dv_r|^p \, dz 
        +
        K^{p(2-p)} |A|^{p} \big( H^*_\lambda(r,r) \big)^p \Bigg],
    \end{align}
    where the dependency of the constant is given by $C=C(n,N,p,L,\kappa_{2M})$. Furthermore, we want to estimate the $\sup$-term of the right-hand side of inequality (\ref{5.41}). To do so, we apply the Caccioppoli inequality from Lemma \ref{Lem5.1} on the cylinders $Q_r^{(\lambda)}$ and $Q_R^{(\lambda)}$. Using hypotheses (\ref{5.401}) and (\ref{2.14}), this results in
    \begin{align}\label{5.43}
        \sup\limits_{t\in \Lambda_r^{(\lambda)}} \mint{-}\limits_{B_r} \bigg| \frac{v_r(\,\cdot\,,t)}{r} \bigg|^2 \, dx 
        &\leq C 
        \mint{-}\limits_{Q_R^{(\lambda)}} \lambda^{2-p} \bigg| V_{|A|} \bigg(\frac{v_r}{R-r}\bigg) \bigg|^2 + \bigg| \frac{v_r}{R-r} \bigg|^2 \, dz
        +
        L \frac{|A|^{p}}{\lambda^{p-2}} H^*_\lambda(R,r)
        \nonumber \\ & \leq C 
        \mint{-}\limits_{Q_R^{(\lambda)}} \bigg| \frac{v_r}{R-r} \bigg|^2 \, dz
        +
        L \lambda^{2-p} |A|^{p} H^*_\lambda(R,r),
    \end{align}
    where $C=C(p,\nu,L,K,\kappa_{2M})$. Let us now replace in (\ref{5.43}) the function $v_r$ on the right-hand side by $v_R$. To this end, we estimate the difference $|v_r-v_R|$ on $Q_R^{(\lambda)}$. By the definitions of $v_r$ and  $v_R$, an application of Hölder's inequality with exponents $(2,2)$ and an enlargement of the domain of integration from $Q_r^{(\lambda)}$ to $Q_R^{(\lambda)}$ (note that $|Q_R^{(\lambda)}|/|Q_r^{(\lambda)}| \leq 2^{n+2}$) we find that
    \begin{align}\label{5.44}
        |v_r - v_R|^2 &= \big| (u)_r^{(\lambda)} - (u)_R^{(\lambda)} \big|^2 = 
        \bigg| \mint{-}\limits_{Q_r^{(\lambda)}} \big( u - (u)_R^{(\lambda)} - Ax \big) \, dz\bigg|^2
        =
        \bigg| \mint{-}\limits_{Q_r^{(\lambda)}} v_R \, dz\bigg|^2
        \leq
        2^{n+2}  \mint{-}\limits_{Q_R^{(\lambda)}} |v_R|^2 \, dz.
    \end{align}
    Let us now estimate the first term of $H^*_\lambda(R,r)$. We obtain, after increasing the domain of integration, that
    \begin{align*}
        \mint{-}\limits_{Q^{(\lambda)}_{R}}  \big|u- (u)_r^{(\lambda)} \big|^{p} \, dz
        &\leq
        2^{p-1} \mint{-}\limits_{Q^{(\lambda)}_{R}}  \big|u- (u)_\varrho^{(\lambda)} \big|^{p} \, dz
        +
        2^{p-1} \mint{-}\limits_{Q^{(\lambda)}_{r}}  \big|u- (u)_\varrho^{(\lambda)} \big|^{p} \, dz
        \nonumber\\ & \leq
        2^{n+p+2} \mint{-}\limits_{Q^{(\lambda)}_{\varrho}}  \big|u- (u)_\varrho^{(\lambda)} \big|^{p} \, dz,
    \end{align*}
    where we used that $|Q_\varrho^{(\lambda)}|/|Q_r^{(\lambda)}| \leq |Q_\varrho^{(\lambda)}|/|Q_R^{(\lambda)}| \leq 2^{n+2}$. This yields
    \begin{align}\label{5.444}
        H^*_\lambda(R,r) \leq C(\mathfrak{C}) H^*_\lambda(\varrho,\varrho)
        \qquad\text{and}\qquad
        H^*_\lambda(r,r) \leq C(\mathfrak{C}) H^*_\lambda(\varrho,\varrho),
    \end{align}
    where the latter estimate is proven in exactly the same manner as the former one. If we join (\ref{5.41})-(\ref{5.444}), apply Young's inequality with exponents $\big(\frac{2}{2-p}, \frac{2}{p} \big)$ and enlarge the domain of integration, we arrive at
    \begin{align}\label{5.45}
        \mint{-}\limits_{Q_\varrho^{(\lambda)}} \bigg| \frac{v_r}{r} \bigg|^2 \, dz
        &\leq C
        \bigg( \frac{R}{R-r} \bigg)^{2-p} \Bigg[ \mint{-}\limits_{Q^{(\lambda)}_{r}} |Du-A|^p \, dz
        + 
        |A|^{p} \big(  H^*_\lambda(r,r) \big)^p \Bigg]
        \Bigg[ \mint{-}\limits_{Q_R^{(\lambda)}} \bigg| \frac{v_R}{R} \bigg|^2 \, dz
        \nonumber\\ & \qquad +
        L \frac{|A|^{p}}{\lambda^{p-2}} H^*_\lambda(R,r) \Bigg]^{1-\frac{p}{2}}
        \leq 
        \frac{1}{2} \mint{-}\limits_{Q_R^{(\lambda)}} \bigg| \frac{v_R}{R} \bigg|^2 \, dz
        +
        L \lambda^{2-p} |A|^p H^*_\lambda(\varrho,\varrho)
        \nonumber\\ & \qquad+ C
        \bigg( \frac{\varrho}{R-r} \bigg)^{\frac{2(2-p)}{p}}
        \cdot 
        \Bigg( \mint{-}\limits_{Q^{(\lambda)}_{\varrho}} \big|Du - A \big|^{p} \, dz 
        + 
        |A|^{p} \big( H^*_\lambda(\varrho,\varrho) \big)^p \Bigg)^{\frac{2}{p}},
    \end{align}
    where the constant is of the form $C=C(\mathfrak{C},K,\kappa_{2M})$. We obtain the stated inequality of the Lemma after applying Lemma \ref{Lem2.2}.
\end{proof}

\section{Approximate \texorpdfstring{$\mathfrak{A}$}{A}- and \texorpdfstring{$p$}{p}-caloricity}\label{SLin}

In this chapter, we linearize our non-linear parabolic system similar to \cite[Section 6]{bogelein2013regularity}. When linearizing, we distinguish between two different regimes (NDR) and (DR) described in Section \ref{Strategy}. The main reason for this is that the diffusion term behaves quite differently, when $|Du|$ is small or large.

\subsection{Linearization in the non-degenerate regime (NDR)}

Let $Q_\varrho^{(\lambda)}(z_0)\subset \Omega_T$ with parameters $\varrho,\lambda>0$ be a intrinsic geometric cylinder and $A\in \mathbb{R}^{Nn}$. Then we define the \textbf{excess functional in the non-degenerate regime} by
\begin{equation}\label{6.01}
    \Phi_\lambda(z_0,\varrho,A)= \mint{-}\limits_{Q_\varrho^{(\lambda)}(z_0)} \big| V_{|A|}(Du-A) \big|^2 \, dz.
\end{equation}
Due to the dependency of the vector field on the space-time point and solution, it will be useful to introduce the \textbf{preliminary hybrid excess functional}
\begin{equation}\label{6.02}
E^*_\lambda(z_0,\varrho,A)=\Phi_\lambda(z_0,\varrho,A)+|A|^{p} H^*_\lambda\big(z_0,\varrho,(u)_{z_0;\varrho}^{(\lambda)}\big),
\end{equation}
which will be estimated from above in dependence on whether we are in the sub- or super-quadratic regime. Note that $H^*_\lambda$ was defined in (\ref{5.01}). With this preliminary hybrid excess in mind, we can linearize the system inside an (intrinsic) cylinder in the following way:
\begin{Lem}\label{Lem6.1}
    Let $p>\frac{2n}{n+2}$, $M,K\geq 1$ and
    \begin{equation*}
        u\in C\big([0,T];L^2\big(\Omega,\mathbb{R}^{N}\big) \big) \cap L^p\big(0,T;W^{1,p}\big(\Omega,\mathbb{R}^{N}\big) \big)
    \end{equation*}
    be a weak solution to (\ref{Sys1}) in $\Omega_T$ that satisfies the assumptions (\ref{As1})-(\ref{ModConu}). Then for any intrinsic geometric cylinder $Q_\varrho^{(\lambda)}(z_0)\subset \Omega_T$ with parameters $\varrho, \lambda>0$ and $A\in \mathbb{R}^{Nn}$ satisfying
    \begin{align}
        |A|\leq M
        \qquad\text{and}\qquad \frac{\lambda}{K}<|A|\leq \lambda K
    \end{align}
    there holds
    \begin{align}\label{6.10}
        \Bigg| \mint{-}\limits_{Q_{\varrho}^{(\lambda)}(z_0)} & u\cdot \partial_t \phi  - \big(D_\xi a\big)\big(z_0 , (u)^{(\lambda)}_{z_0;\varrho}, A \big)(Du-A,D\phi) \, dz \Bigg| 
        \nonumber\\& \leq 
        C(p,\kappa_{2M}) L |A|^{\frac{p-2}{2}} \big(E^*_\lambda\big)^{\frac{1}{2}} \Bigg[ \bigg( \frac{E^*_\lambda}{|A|^p} \bigg)^{\left|\frac{1}{2}-\frac{1}{p}\right|} + \omega_{2M} \bigg( \frac{E^*_\lambda}{|A|^p} \bigg) +
        \bigg( \frac{E^*_\lambda}{|A|^p} \bigg)^{\frac{1}{2}} \Bigg] \sup\limits_{Q_\varrho^{(\lambda)}(z_0)} |D\phi| 
    \end{align}
    for all $\phi\in C^1_0\big( Q_{\varrho}^{(\lambda)}(z_0),  \mathbb{R}^{N} \big)$. Here we abbreviated $E^*_\lambda=E^*_\lambda(z_0,\varrho,A)$.
\end{Lem}
\begin{proof}
    To shorten the notation, we set $z_0=0$ and abbreviate in the following $ Q_\varrho^{(\lambda)}$ with $Q_\varrho$, $\Phi_\lambda(0,\varrho,A)$ with $\Phi_\lambda$ and $E^*_\lambda(0,\varrho,A)$ with $E^*_\lambda$. Furthermore, we define the vector field
    \begin{equation*}\label{6.11}
        \xi\mapsto \Tilde{a}(\xi) = a\big(0,(u)_{\varrho}^{(\lambda)},\xi\big).
    \end{equation*}
    Without loss of generality, we may assume $\| D\phi \|_\infty\leq 1$. The general case can be retrieved by linearity using $\phi/\| D\phi \|_\infty$ in the obtained, when $\| D\phi \|_\infty>0$. From the parabolic system and the fact that 
    \begin{equation*}\label{6.12}
        \mint{-}\limits_{Q_{\varrho}} \Tilde{a}(A) \cdot D\phi \, dz = 0,
    \end{equation*}
    we infer
    \begin{align*}\label{6.13}
        \Bigg| \mint{-}\limits_{Q_{\varrho}} & u\cdot \partial_t \phi  - \big(D_\xi \Tilde{a}\big)(A)(Du-A,D\phi) \, dz \Bigg|
        =
        \Bigg| \mint{-}\limits_{Q_{\varrho}} a(z,u,Du) \cdot D\phi - \big(D_\xi \Tilde{a}\big)(A)(Du-A,D\phi) \, dz \Bigg|
        \\ & =
        \Bigg| \mint{-}\limits_{Q_{\varrho}} \big[ a(z,u,Du) - \Tilde{a}(Du) + \Tilde{a}(Du) - \Tilde{a}(A) \big] \cdot D\phi - \big(D_\xi \Tilde{a}\big)(A)(Du-A,D\phi) \, dz \Bigg|
        \\ & \leq
        \Bigg| \mint{-}\limits_{Q_{\varrho}} \big[ a(z,u,Du)
         - \Tilde{a}(Du) \big] \cdot D\phi \, dz \Bigg|
        \\ & \qquad +
        \Bigg| \mint{-}\limits_{Q_{\varrho}} \big[\Tilde{a}(Du) - \Tilde{a}(A) \big] \cdot D\phi - \big(D_\xi \Tilde{a}\big)(A)(Du-A,D\phi) \, dz \Bigg|
        =: J_1 + J_2
    \end{align*}
    with the obvious definitions of $J_1$ and $J_2$. We start by estimating $J_2$. In order to do this, we once again decompose the domain of integration $Q$ into two pieces, $S_1$ and $S_2$, which are defined by
    \begin{equation*}\label{6.14}
        S_1= \{ z\in Q_{\varrho} \,|\, |Du(z)-A| < |A| \}, \qquad S_2= \{ z\in Q_{\varrho} \,|\, |Du(z)-A| \geq |A| \}.
    \end{equation*}
    With this decomposition, we split the estimation of $J_2$ into two parts, $J_{2,1}$ and $J_{2,2}$, corresponding to the domains of integration $S_1$ and $S_2$, respectively. Since the treatment of $J_{2,1}$ is verbatim that of \cite[Inequalities (6.3)-(6.6)]{bogelein2013regularity}, we just state the result
    \begin{align}\label{6.15}
        J_{2,1} \leq
        C(p) L |A|^{\frac{p-2}{2}} \omega_{2M} \bigg( \frac{E^*_\lambda}{|A|^p} \bigg) \big(E^*_\lambda \big)^{\frac{1}{2}} 
    \end{align}
    In the next step, we estimate the integral $J_{2,2}$. On the domain of integration $S_2$, we use Lemma \ref{Lem2.7} with $B\equiv Du - A$ and the growth assumption (\ref{As3}) for $D_\xi a$. Note that Lemma is applicable since we have $0<|A|\leq |Du-A|$ on $S_2$ and $ \kappa_M \leq \kappa_{2M}$ due to the monotonicity of $\kappa$. This results in
    \begin{align}\label{6.19}
        J_{2,2} 
        &= 
        \frac{1}{|Q_{\varrho}|} \Bigg| \int\limits_{S_2} \big[\Tilde{a}(Du) - \Tilde{a}(A) \big] \cdot D\phi - \big(D_\xi \Tilde{a}\big)(A)(Du-A,D\phi) \, dz \Bigg|
        \nonumber\\ & \leq
        \frac{1}{|Q_{\varrho}|} \int\limits_{S_2} | \Tilde{a}(Du) - \Tilde{a}(A)| + |D_\xi \Tilde{a}|(A)|Du-A| \, dz
        \nonumber\\ & \leq
        \frac{1}{|Q_{\varrho}|} \int\limits_{S_2} \big| a\big(0,(u)_{\varrho}^{(\lambda)},Du \big) - a\big(0,(u)_{\varrho}^{(\lambda)},A \big)\big| + |D_\xi a|\big(0,(u)_{\varrho}^{(\lambda)},A \big)|Du-A| \, dz 
        \nonumber\\ & \leq
        \frac{C(p,\kappa_{2M}) L}{|Q_{\varrho}|}\int\limits_{S_2}  |Du-A|^{p-1} + |A|^{p-2}|Du-A| \, dz.
    \end{align}
    We again distinguish between the sub- and super-quadratic cases. In the latter case, we use the fact that $|A|\leq |Du-A|$ on $S_2$ and Hölder's inequality with exponents $\big(p,\frac{p}{p-1}\big)$ to estimate the right-hand side of (\ref{6.19}). This yields
    \begin{align}\label{6.191}
        J_{2,2} 
        &\leq
        \frac{C L}{|Q_{\varrho}|}\int\limits_{S_2}  |Du-A|^{p-1} \, dz
        \leq
        C L \Bigg[\mint{-}\limits_{Q_{\varrho}}  |Du-A|^{p} \, dz \Bigg]^{1-\frac{1}{p}}
        \leq
        C(p,\kappa_{2M}) L \big(\Phi_\lambda\big)^{1-\frac{1}{p}}.
    \end{align}
    In the sub-quadratic case, we use that $ |Du-A|^{p-2} \leq |A|^{p-2}$ on $S_2$ and apply Hölder's inequality with exponents $\big(p,\frac{p}{p-1}\big)$ to obtain with inequality (\ref{2.14}) of Lemma \ref{Lem2.1} that
    \begin{align}\label{6.192}
        J_{2,2} 
        &\leq
        \frac{C L}{|Q_{\varrho}|}\int\limits_{S_2}  |A|^{p-2}|Du-A| \, dz
        \leq
        C(p,\kappa_{2M}) L |A|^{p-2} \big(\Phi_\lambda\big)^{\frac{1}{p}},
    \end{align}
    Joining (\ref{6.191}) and (\ref{6.192}), we have
    \begin{align}\label{6.193}
        J_{2,2} \leq
        C(p,\kappa_{2M}) L |A|^{\frac{p-2}{2}} \big( E^*_\lambda \big)^{\frac{1}{2}} \bigg( \frac{E^*_\lambda}{|A|^p} \bigg)^{\left|\frac{1}{2}-\frac{1}{p}\right|}.
    \end{align}
    It remains to estimate $J_1$. This will be done by splitting the integral into two different domains of integration. We have
    \begin{align}\label{6.194}
        J_1 & \leq  \frac{1}{|Q_{\varrho}|}\int\limits_{S_1} \big| a(z,u,Du) - a\big(0,(u)_{\varrho}^{(\lambda)},Du \big) \big| \, dz
        +
        \frac{1}{|Q_{\varrho}|}\int\limits_{S_2} \big| a(z,u,Du) - a\big(0,(u)_{\varrho}^{(\lambda)},Du\big)\big| \, dz
        \nonumber\\ & =: J_{1,1} + J_{1,2}
    \end{align}
    with the obvious definitions of $J_{1,1}$ and $J_{1,2}$. Let us start by estimating $J_{1,1}$. First, note that we have $|Du| \leq |Du-A| + |A| \leq 2|A|$ on $S_1$. We estimate with the Hölder-continuity assumption (\ref{ModConu}) on $a$ and Hölder's inequality with exponents $\big(\frac{p}{\beta},\frac{p}{p-\beta}\big)$ that
    \begin{align}\label{6.195}
        J_{1,1} & \leq
        L
        \frac{1}{|Q_{\varrho}|} \int\limits_{S_1} \big( (|x|^2 + |t|)^{\frac{\beta}{2}} + |u-(u)_{\varrho}^{(\lambda)}|^\beta \big) |Du|^{p-1} \, dz
        \nonumber \\ & \leq
        L 2^{p-1} |A|^{p-1} \frac{1}{|Q_{\varrho}|} \int\limits_{S_1} \big( (|x|^2 + |t|)^{\frac{\beta}{2}} + \big|u-(u)_{\varrho}^{(\lambda)}\big|^\beta \big) \, dz
        \nonumber \\ & \leq
        L 2^{p-1} |A|^{p-1} \mint{-}\limits_{Q_{\varrho}} \Big( (|x|^2 + |t|)^{\frac{\beta}{2}} + \big|u-(u)_{\varrho}^{(\lambda)}\big|^\beta \Big) \, dz
        \nonumber \\ & \leq
        L 2^{p-1} |A|^{p-1}\big( \varrho^\beta ( 1+\lambda^{2-p})^{\frac{\beta}{2}} + \Psi_\lambda^{\frac{\beta}{p}} \big)
        \leq
        C(p) L |A|^{\frac{p-2}{2}} \big( E^*_\lambda \big)^{\frac{1}{2}} \bigg( \frac{E^*_\lambda}{|A|^p} \bigg)^{\frac{1}{2}}.
    \end{align}
    Now, we continue with $J_{1,2}$. Here we apply Lemma \ref{Lem2.7} with $B\equiv Du - A$. This is possible on $S_2$. Afterwards, we estimate similar to (\ref{6.191}), (\ref{6.192}) and (\ref{6.195}), which results in
    \begin{align}\label{6.196}
        J_{1,2} & \leq
        \frac{1}{|Q_{\varrho}|}\int\limits_{S_2} | a(z,u,Du) - a(z,u,A)| \, dz
        +
        \frac{1}{|Q_{\varrho}|}\int\limits_{S_2} \big| a\big(0,(u)_{\varrho}^{(\lambda)},A\big) - a\big(0,(u)_{\varrho}^{(\lambda)},Du\big)\big| \, dz
        \nonumber \\ & \qquad+
        \frac{1}{|Q_{\varrho}|}\int\limits_{S_2} \big| a(z,u,A) - a\big(0,(u)_{\varrho}^{(\lambda)},A\big)\big| \, dz
        \nonumber \\ & \leq
        \frac{C(p,\kappa_{1}) L}{|Q_{\varrho}|}\int\limits_{S_2}  |Du-A|^{p-1} \, dz
        +
        L \frac{1}{|Q_{\varrho}|} \int\limits_{S_2} \Big( (|x|^2 + |t|)^{\frac{\beta}{2}} + \big|u-(u)_{\varrho}^{(\lambda)}\big|^\beta \Big) |A|^{p-1} \, dz
        \nonumber \\ & \leq
        C(p,\kappa_{2M}) L |A|^{\frac{p-2}{2}} \big( E^*_\lambda \big)^{\frac{1}{2}} \bigg( \frac{E^*_\lambda}{|A|^p} \bigg)^{\left|\frac{1}{2}-\frac{1}{p}\right|}
        +
        C(p) L |A|^{\frac{p-2}{2}} \big( E^*_\lambda \big)^{\frac{1}{2}} \bigg( \frac{E^*_\lambda}{|A|^p} \bigg)^{\frac{1}{2}}.
    \end{align}
    The Lemma now immediately follows, if we combine (\ref{6.15}), (\ref{6.193}) and (\ref{6.196}).
\end{proof}

\subsection{Linearization in the degenerate regime (DR)} In the degenerate case, the gradient is "small" compared to the excess. Since we assume that the vector field behaves like the $p$-Laplacian for small gradients, we expect the solution to behave similarly to that of the parabolic $p$-Laplacian, which will later allow us to apply the $p$-caloric approximation Lemma. We start by defining the \textbf{excess functional in the degenerate case}
\begin{equation}\label{defPsi}
    \Tilde{\Phi}_\lambda(z_0,\varrho)= \mint{-}\limits_{Q_\varrho^{(\lambda)}(z_0)} |Du|^p \, dz.
\end{equation}
Note that this is just the excess functional in the non-degenerate regime with the particular choice $A=0$. With the excess functional in the degenerate case in mind, we have the following
\begin{Lem}\label{Lem6.2}
    Let $p>\frac{2n}{n+2}$ and 
    \begin{equation*}
        u\in C\big([0,T];L^2\big(\Omega,\mathbb{R}^{N}\big) \big) \cap L^p\big(0,T;W^{1,p}\big(\Omega,\mathbb{R}^{N}\big) \big)
    \end{equation*}
    be a weak solution to (\ref{Sys1}) in $\Omega_T$ satisfying the assumptions (\ref{As1})-(\ref{As3}) and (\ref{Beh0})-(\ref{ModConu2}). Then for any intrinsic geometric cylinder $Q_\varrho^{(\lambda)}(z_0) \subset \Omega_T$ with $\varrho,\lambda>0$ and $\delta>0$ we have
    \begin{align}\label{6.201}
        \Bigg| \mint{-}\limits_{Q_\varrho^{(\lambda)}(z_0)} & u \cdot \partial_t \phi - \Tilde{a}\big(z_0,(u)_{z_0;\varrho}^{(\lambda)}\big) |Du|^{p-2} Du \cdot D\phi \, dz \Bigg| 
        \nonumber \\ & \leq
        C \Bigg[ \delta + \frac{\Tilde{\Phi}_\lambda^{\frac{1}{p}}}{\eta(\delta)}
        +
        \varrho^\beta \Big( ( 1+\lambda^{2-p})^{\frac{\beta}{2}} + 
        \big(  \Tilde{\Phi}_\lambda + \lambda^{p(2-p)} \Tilde{\Phi}_\lambda^{p-1}\big)^{\frac{\beta}{p}} \Big) \Bigg]
       \Tilde{\Phi}_\lambda^{1-\frac{1}{p}} \sup\limits_{Q_\varrho^{(\lambda)}(z_0)} |D\phi|
    \end{align}
    whenever $\phi \in C^1_0(Q_\varrho^{(\lambda)}(z_0),\mathbb{R}^N)$ with $C=C(\mathfrak{C})$. Here we abbreviated $\Tilde{\Phi}_\lambda = \Tilde{\Phi}_\lambda (z_0,\varrho)$ and $\eta(\,\cdot\,)$ is from (\ref{2.7}).
\end{Lem}
\begin{proof}
    Without loss of generality, assume that $z_0=0$. We abbreviate $Q\equiv Q_\varrho^{(\lambda)}$ and $l=(u)_{\varrho}^{(\lambda)}$. Furthermore, as in the previous Lemma, we can assume by scaling that $\phi \in C^1_0(Q,\mathbb{R}^N)$ fulfills the condition $|D\phi| \leq 1$. From the weak formulation (\ref{weakfrom}) of the parabolic system (\ref{Sys1}), we know that
    \begin{align}\label{6.21}
        \Bigg| \mint{-}\limits_{Q} u \cdot \partial_t \phi & - \Tilde{a}(0,l) |Du|^{p-2} Du \cdot D\phi \, dz \Bigg| 
        =
        \Bigg| \mint{-}\limits_{Q} \big[ a(z,u,Du) - \Tilde{a}(0,l) |Du|^{p-2} Du \big] \cdot D\phi  \, dz \Bigg|=: I 
    \end{align}
    Now we estimate the right-hand side. For this purpose, the decomposition of the domain of integration $Q$ into the sets $S_1 = \{ z\in Q \,|\, |Du(z)|\leq \eta(\delta) \}$ and its complement $S_2= Q\setminus S_1$ is used. We have
    \begin{align}\label{6.22}
        I &\leq
        \Bigg| \mint{-}\limits_{S_1} \big[ a(z,u,Du) - \Tilde{a}(z,u) |Du|^{p-2} Du \big] \cdot D\phi  \, dz \Bigg| 
        \nonumber \\ & \qquad +
        \Bigg| \mint{-}\limits_{S_2} \big[ a(z,u,Du) - \Tilde{a}(z,u) |Du|^{p-2} Du \big] \cdot D\phi  \, dz \Bigg|
        \nonumber \\ & \qquad +
        \mint{-}\limits_{Q} \big| \Tilde{a}(z,u) - \Tilde{a}(0,l) \big| |Du|^{p-1} \, dz
        =: I_1 + I_2 + I_3,
    \end{align}
    where $I_1$, $I_2$ and $I_3$ are defined in the obvious manner. We start by estimating $I_1$ with the help of inequality (\ref{2.7}). This yields
    \begin{align}\label{6.23}
        I_1 \leq \delta \mint{-}\limits_{Q} |Du|^{p-1} \, dz
        \leq
        \delta \Bigg[ \mint{-}\limits_{Q} |Du|^{p} \, dz \Bigg]^{1-\frac{1}{p}}
        =
        \delta \Tilde{\Phi}_\lambda.
    \end{align}
    Now we continue with the treatment of $I_2$. First, note that we can estimate the measure of $S_2$ by the $L^p$-norm of $Du$. As a matter of fact, we have
    \begin{align}\label{6.24}
        |S_2| \leq \frac{1}{\eta(\delta)^p} \int\limits_{S_2} |Du|^p \, dz.
    \end{align}
    Furthermore, we can estimate the integrand with the growth assumptions (\ref{As1}) and (\ref{Asa}) to obtain
    \begin{align}\label{6.25}
        \big| a(z,u,Du) - \Tilde{a}(z,u) |Du|^{p-2} Du \big| \leq
        L |Du|^{p-1} + L |Du|^{p-1} \leq L 2 |Du|^{p-1}.
    \end{align}
    This results in the following upper bound for the integral $I_2$ using Hölder's inequality with exponents $\big(p,\frac{p-1}{p}\big)$ and estimate (\ref{6.24}):
    \begin{align}\label{6.28}
        I_2 \leq \frac{C}{|Q|} \int\limits_{S_2} |Du|^{p-1} \, dz
        \leq \frac{C}{|Q|} |S_2|^{\frac{1}{p}} \Bigg[ \int\limits_{Q} |Du|^{p} \, dz \Bigg]^{1-\frac{1}{p}}
        \leq \frac{C(\mathfrak{C})}{\eta(\delta)} \Tilde{\Phi}_\lambda.
    \end{align}
    We will conclude the proof of the Lemma by estimating $I_3$. To this end, it should be noted that the Poincaré inequality (\ref{5.301}) applied with $A=0$ and $q=p$ implies:
    \begin{align*}
        \mint{-}\limits_{Q} |u-l|^{p} \, dz \leq C(n,N,p) L \varrho^p  \big(  \Tilde{\Phi}_\lambda + \lambda^{p(2-p)} \Tilde{\Phi}_\lambda^{p-1}\big).
    \end{align*}
    We now use the preceding inequality and the continuity assumption (\ref{ModConu2}) to finally estimate $I_3$. By applying Hölder's inequality with exponents $\big(p,\frac{p}{p-1}\big)$, we obtain
    \begin{align}\label{6.29}
        I_3 
        &\leq L \mint{-}\limits_{Q} \big( (|x|^2 + |t| )^{\frac{\beta}{2}} + |u-l|^\beta \big) |Du|^{p-1} \, dz
        \leq 
        C L \varrho^\beta ( 1+\lambda^{2-p})^{\frac{\beta}{2}} \Bigg[ \mint{-}\limits_{Q} |Du|^{p} \, dz \Bigg]^{1-\frac{1}{p}} 
        \nonumber\\ & \qquad +
        C L
        \Bigg[ \mint{-}\limits_{Q} |u-l|^{\beta p} \, dz \Bigg]^{\frac{1}{p}}
        \Bigg[ \mint{-}\limits_{Q} |Du|^{p} \, dz \Bigg]^{1-\frac{1}{p}}
        \leq 
        C \Bigg( \varrho^\beta ( 1+\lambda^{2-p})^{\frac{\beta}{2}} + \bigg[ \mint{-}\limits_{Q} |u-l|^{p} \, dz \bigg]^{\frac{\beta}{p}} \Bigg)
        \Tilde{\Phi}_\lambda^{1-\frac{1}{p}}
        \nonumber\\ & \leq 
        C \varrho^\beta  \Big( ( 1+\lambda^{2-p})^{\frac{\beta}{2}} + 
        \big(  \Tilde{\Phi}_\lambda + \lambda^{p(2-p)} \Tilde{\Phi}_\lambda^{p-1}\big)^{\frac{\beta}{p}} \Big)
       \Tilde{\Phi}_\lambda^{1-\frac{1}{p}}.
    \end{align}
    Combining the estimates (\ref{6.23}), (\ref{6.28}) and (\ref{6.29}) with (\ref{6.22}) proves the statement of the Lemma.
\end{proof}

\section{DiBenedetto \& Friedman regularity theory}\label{SDI}

In this Section, we revisit the famous estimates used by DiBenedetto \& Friedman \cite{DiBenedettoFriedman1,DiBenedettoFriedman2,DiBenedettoFriedman3} to prove the $C^{0,\alpha}$ gradient regularity for solutions to the parabolic $p$-Laplace system. 

\subsection{The degenerate theory}

In this section, we consider weak solutions to the parabolic $p$-Laplacian in the domain $U\times[t_1,t_2)$ with $t_1<t_2$ and $U\subset \R^n$. This revers to functions 
\begin{align}\label{7.101}
    h\in C\big([t_1,t_2]; L^2(U,\R^N)\big) \cap L^p\big(t_1,t_2; W^{1,p}\big(U,\R^N\big) \big),
\end{align}
that satisfy the partial differential equation 
\begin{align}\label{7.102}
    \partial_t h - \Lambda \, \mathrm{div}\big( |Dh|^{p-2} Dh \big) = 0 \qquad \text{in } U\times[t_1,t_2)
\end{align}
for some constant $\Lambda\in  [\nu,L]$ with $\nu,L\in \R$ in the usual weak sense. For this type of solution, we use an excess-decay type estimate for both the sub- and super-quadratic cases, which is based on the famous proof of $C^{1,\alpha}$ continuity for solutions in \cite[Chapter IX]{DiBenedettoDPE}. The excess-decay type estimate allows us to derive an appropriate a priori estimate for the solution. \\
We start by showing that a quantified local $L^p$-estimate implies a quantified upper bound for the modulus of the gradient of the solution. This is the content of the following Lemma which also holds true for Orlicz growth; see \cite[Corollary 5.3]{ok2024regularity}. 
\begin{Lem}\label{Lem7.01}
    Suppose that $p>\frac{2n}{n+2}$, $\Lambda>0$, $z_0\in \R^{n+1}$, $\varrho>0$ and $h$ is a weak solution of (\ref{7.102}) in $Q_r(z_0)$ with 
    \begin{align}\label{7.103}
        \mint{-}\limits_{Q_\varrho(z_0)} |Dh|^p \, dz \leq C_*.
    \end{align}
    Then there exists $C=C(n,N,p,\nu,L,C_*)\geq 1$ such that the $\mathrm{sup}$-estimate
    \begin{align}
        \sup\limits_{Q_{\varrho/2}(z_0)} |Dh| \leq C
    \end{align}
    holds.
\end{Lem} 
\begin{proof}
    Without loss of generality, we assume $z_0=0$. We observe that the considered homogeneous parabolic $p$-Laplace system (\ref{7.102}) fulfills the structure assumptions of \cite[Chapter VIII, 1-(ii)]{DiBenedettoDPE} with $C_0=\nu \min\{1,p-1\}$, $C_1=L$ and $\phi_0=\phi_1=0$. We distinguish between the cases $ p\geq 2$ and $2n/(n+2) < p < 2$. In the super-quadratic case, we use $\sigma=1/2$ in \cite[Chapter VIII, (5.1), page 238]{DiBenedettoDPE} and obtain for the cylinder $Q(\gamma,\vartheta)(z_0) = B_\gamma(x_0) \times \Lambda_{\sqrt{\vartheta}}^{(1)}(t_0) \subset Q_r$ that
    \begin{align}
        \sup\limits_{Q(\gamma/2,\vartheta/2)(z_0)} |Dh| \leq C(n,N,p,\nu,L) \sqrt{\frac{\vartheta}{\gamma^2}} \Bigg[ \mint{-}\limits_{Q(\varrho,\gamma)(z_0)} |Dh|^p \, dz \Bigg]^{\frac{1}{2}} + \bigg( \frac{\gamma^2}{\vartheta} \bigg)^{\frac{1}{p-2}}
    \end{align}
    We apply this with the choices $\gamma=\varrho$, $\vartheta=\varrho^2$ and use the hypothesis (\ref{7.103}) to immediately obtain
    \begin{align}\label{7.012}
        \sup\limits_{Q_{\varrho/2}(z_0)} |Dh| \leq C(n,N,p,\nu,L) \sqrt{C_*} + 1 =: C.
    \end{align}
    We notice that $C\geq 1$ has the stated dependencies of the Lemma. In the sub-quadratic case, we use \cite[Chapter VIII, (5.2), page 238f]{DiBenedettoDPE} with $\sigma=1/2$ to infer for the cylinder $Q(\gamma,\vartheta)(z_0) \subset Q_\varrho$ that
    \begin{align}
        \sup\limits_{Q(\gamma/2,\vartheta/2)(z_0)} |Dh| \leq C(n,N,p,\nu,L) \bigg( \frac{\vartheta}{\gamma^2} \bigg)^{\frac{n}{p(n+2)-2n}} \Bigg[ \mint{-}\limits_{Q(\varrho,\gamma)(z_0)} |Dh|^p \, dz \Bigg]^{\frac{2}{p(n+2)-2n}} + \bigg( \frac{\gamma^2}{\vartheta} \bigg)^{\frac{1}{2-p}}
    \end{align}
    and we can proceed exactly as in (\ref{7.012}). This completes the proof.
\end{proof}
Using the established sup-estimate, one can show a Campanato-type estimate, which is obtained similar to \cite[Theorem 1.3]{BOGELEIN2022113119} and \cite[Theorem 6.1]{ok2024regularity}.
\begin{Lem}\label{Lem7.02}
    Let $p>1$, $\Lambda>0$, $z_0\in \R^{n+1}$, $\varrho>0$ and $h$ be a weak solution of (\ref{7.102}) in $Q_\varrho(z_0)$. If there exists a constant $C_*$ and $\varrho_0 \leq \varrho$ such that 
    \begin{align}
        \sup\limits_{Q_{\varrho_0}(z_0)} |Dh| \leq C_* 
    \end{align}
    holds, then there exist constants $\gamma>0$ and $\beta_1\in (0,1)$ depending on the data $n,N,p,\nu,L$ and $C_*$ such that for all radii $r\in \big(0,\frac{\varrho_0}{2}]$ there holds
    \begin{align}
        \mint{-}\limits_{Q_{r}(z_0)} \big|Dh - (Dh)_{z_0;r} \big|^p \leq \gamma \bigg( \frac{r}{\varrho_0} \bigg)^{\beta_1 p}.
    \end{align}
\end{Lem}
We now apply the preceding Lemma with the choice $\varrho_0=\varrho/2$.
\begin{Lem}\label{Lem7.1v2}
    Suppose, that $\frac{2n}{n+2} < p$, $\max\{2,p\}\leq q$, $\nu,L>0$, $t_1<t_2$ and $U\subset \R^n$. Assume that $h$ is a weak solution to (\ref{7.102}) satisfying (\ref{7.103}) for some cylinder $Q_\varrho(z_0)\subset U\times[t_1,t_2)$ and constant $C_*$. Then there exist $C\geq 1$ $0< \beta_1 <1$ depending on the data $\mathfrak{C},q$ and $C_*$ such that there holds:
    \begin{align}\label{7.105}
        \sup\limits_{Q_{\varrho/4}(z_0)} |Dh| \leq C.
    \end{align}
    Moreover, for any 
    \begin{align}\label{defs}
        s\in \begin{cases}
            [2,p] & \text{for } p\geq 2, \\
            [1,2] & \text{for } \frac{2n}{n+2} < p < 2
        \end{cases}
    \end{align}
    we have the decay estimate
    \begin{align}\label{7.106}
        \frac{1}{r^s} \mint{-}\limits_{Q_r(z_0)} \big| h - (h)_{z_0;r} - (Dh)_{z_0;r} (x-x_0) \big|^s \, dz 
        \leq
        C \bigg( \frac{r}{\varrho} \bigg)^{\min\{s,p,s(p-1)\}\beta_1}
        \qquad \forall 0<r\leq \varrho/4,
    \end{align}
    where $C=C(\mathfrak{C},\Lambda,C_*)$.
\end{Lem}
\begin{proof}
    For simplicity of notation, we again assume that $z_0=0$ and note that the assumptions of Lemma \ref{Lem7.01} are satisfied, implying that (\ref{7.105}) holds for $0<r\leq \varrho/2$. This in turn shows that the assumptions of Lemma \ref{Lem7.02} are satisfied, yielding two constants $C=C(\mathfrak{C},C_*)>0$ and $\beta_1=\beta_1(\mathfrak{C},C_*)\in (0,1)$ such that
    \begin{align}\label{7.11}
        \mint{-}\limits_{Q_{r}} \big|Dh - (Dh)_{r} \big|^p \leq C \bigg( \frac{r}{\varrho} \bigg)^{\beta_1 p}
    \end{align}
    for every $0<r\leq \varrho/4$. Now, we apply the Poincaré inequality (\ref{5.301}) of Lemma \ref{Lem5.3} with the parameter choices $q=s$ and $A= (Dh)_r$, where $s$ is given by (\ref{defs}). We note that the vector field under consideration does not depend on $x$ and $u$, which is why we can disregard the additional terms. Therefore, we obtain
    \begin{align}\label{7.12}
        r^{-s} \mint{-}\limits_{Q_r} \big| h & - (h)_r - (Dh)_r x |^s \, dz
        \leq
        C \Bigg[ \mint{-}\limits_{Q_{r}} \big| Du - (Dh)_r \big|^s \, dz 
        \nonumber \\ & \qquad +
        \bigg( \mint{-}\limits_{Q_{r}} \mathds{1}_{\{p\geq 2\}} \big|(Dh)_r \big|^{p-2} \big| Dh - (Dh)_r \big| 
        + \big| Dh - (Dh)_r \big|^{p-1}  \, dz \bigg)^s  \Bigg]
        \nonumber \\ & \leq
        C \Bigg[ \mint{-}\limits_{Q_{r}} \big| Dh - (Dh)_r \big|^s \, dz 
        +
        \bigg( \mint{-}\limits_{Q_{r}} \big| Dh - (Dh)_r \big|^{p-1} \, dz \bigg)^s  \Bigg],
    \end{align}
    where we used Young's inequality and (\ref{7.105}) to estimate the term that occurs in the super-quadratic case. For the constant, we have by its monotonous dependence on $s$ that $C=C(\mathfrak{C},C_*)$ and it should be mentioned that we presently have $\kappa_M \equiv 1$ and $L \equiv (p-1)L$. If we reduce the exponent of integration by means of (\ref{7.105}) in the sub-quadratic case, we obtain from (\ref{7.11}), (\ref{7.12}) and Hölder's inequality that
    \begin{align}\label{7.14}
        r^{-s} \mint{-}\limits_{Q_r} \big| h & - (h)_r - (Dh)_rx |^s \, dz
        \leq
        C \Bigg[ \mint{-}\limits_{Q_{r}} \big| Dh - (Dh)_r \big|^s \, dz 
        +
        \bigg( \mint{-}\limits_{Q_{r}} \big| Dh - (Dh)_r \big|^{p-1}  \, dz \bigg)^s  \Bigg]
        \nonumber\\ & \leq
        C \Bigg[ \mint{-}\limits_{Q_{r}} \big| Dh - (Dh)_r \big|^{\min\{s,p\}} \, dz 
        +
        \bigg( \mint{-}\limits_{Q_{r}} \big| Dh - (Dh)_r \big|^{p}  \, dz \bigg)^{\frac{s(p-1)}{p}}  \Bigg]
        \nonumber\\ & \leq
        C \bigg( \mint{-}\limits_{Q_{r}} \big| Dh - (Dh)_r \big|^{p}  \, dz \bigg)^{\frac{\min\{s,p,s(p-1)\}}{p}} 
        \leq
        C \bigg( \frac{r}{\varrho} \bigg)^{\min\{s,p,s(p-1)\}\beta_1},
    \end{align}
    where $C=C(\mathfrak{C},C_*)$ and the statement of the Lemma follows.
\end{proof}

\subsection{The linear theory}

In the non-degenerate case, we approximate with the solution of a linear parabolic system with constant coefficients by freezing the $(z,u)$ dependency of the vector field. In this situation, we need the following classical excess-decay estimate; for a statement, see \cite[Lemma 7.3]{bogelein2013regularity}. The proof can be found, for instance, in \cite[Lemma 5.1]{campanato1966equazioni}.
\begin{Lem}\label{Lem7.2}
    Let $h\in L^2\big(\Lambda_\varrho(t_0); W^{1,2}(B_\varrho(x_0),\R^N)\big)$ be a weak solution in the cylinder $Q_\varrho(z_0)$ of a linear parabolic system with constant coefficients, in the sense, that
    \begin{align}\label{7.201}
        \mint{-}\limits_{Q_r(z_0)} h \cdot \partial_t \phi - \mathcal{A}(Dh,D\phi) \, dz = 0 \qquad \forall \phi \in C^\infty_0(Q_r(z_0),\R^N),
    \end{align}
    where the coefficients $\mathcal{A}$ satisfy
    \begin{align}\label{7.202}
        \mathcal{A}(\eta \otimes \zeta, \eta\otimes \zeta) \geq \nu |\eta|^2 |\zeta|^p, \qquad
        \mathcal{A}(\xi,\Tilde{\xi}) \leq L |\xi| |\Tilde{\xi}|,
    \end{align}
    for any $\eta\in \R^n$, $\xi\in \R^N$ and $\xi,\Tilde{\xi} \in \R^{Nn}$. Then $h$ is smooth in $Q_r(z_0)$ and there exits a constant $C_{pa}=C_{pa}(n,N,p,L/\nu)\geq 1$ such that for all $s\geq 1$ and $\vartheta\in (0,1]$ there holds
    \begin{align}\label{7.203}
        \frac{1}{(\vartheta\varrho)^s} \mint{-}\limits_{Q_{\vartheta\varrho}(z_0)} & \big| h - (h)_{z_0;\vartheta\varrho} - (Dh)_{z_0;\vartheta\varrho} (x-x_0) \big|^s \, dz 
        \nonumber\\ &\leq
        C_{pa} \vartheta^s \varrho^{-s} \mint{-}\limits_{Q_{\varrho}(z_0)} \big| h - (h)_{z_0;\varrho} - (Dh)_{z_0;\varrho} (x-x_0) \big|^s \, dz.
    \end{align}
\end{Lem}

\section{Partial gradient regularity}\label{SIN}

In this final chapter we want to proof Theorem \ref{Theo1} and \ref{Theo2}, where we distinguish between the non-degenerate and degenerate regimes. But before we do this, let us start this Section by defining the hybrid excess functional $E_\lambda(z_0,\varrho)$ as follows:
\begin{equation}\label{8.02}
    E_\lambda(z_0,\varrho) = \Phi_\lambda \big(z_0,\varrho,        
    (Du)_{z_0;\varrho}^{(\lambda)}\big) + H_\lambda(z_0,\varrho),
\end{equation}
where the latter term is defined by
\begin{equation}\label{8.04}
    H_\lambda(z_0,\varrho)= \big|(Du)_{z_0;\varrho}^{(\lambda)}\big|^{p} H^*_\lambda\big(z_0,\varrho,(u)_{z_0;\varrho}^{(\lambda)} \big)
    =
     \big|(Du)_{z_0;\varrho}^{(\lambda)}\big|^{p} \Big( \Psi_\lambda^{\frac{\beta}{p}}  \big(z_0,\varrho, (u)_{z_0;\varrho}^{(\lambda)}  \big) + \varrho^\beta (1+\lambda^{2-p})^{\frac{\beta}{2}} \Big)
\end{equation}
Note that this is just the preliminary hybrid excess functional $E^*_\lambda(z_0,\varrho,A)$ (see (\ref{6.02})) for the particular choice $A=(Du)_{z_0;\varrho}^{(\lambda)}$. It will be useful throughout this Section to derive an upper bound for the term $H_\lambda(z_0,\varrho)$. To this end, we start by estimating the term $\Psi_\lambda(z_0,\varrho,(u)_{\varrho}^{(\lambda)})$ of the hybrid excess functional. Thereby, we use that inequality (\ref{5.301}) of Lemma \ref{Lem5.3} applied with $A=0$ implies for every $p>2n/(n+2)$ that
\begin{align}\label{9.01}
    \mint{-}\limits_{Q_\varrho^{(\lambda)}(z_0)} \big| u & - (u)_{\varrho}^{(\lambda)}\big|^p \, dz
    \leq
    C \varrho^p \Bigg[ \mint{-}\limits_{Q^{(\lambda)}_{\varrho}(z_0)} \big|Du \big|^p \, dz 
    +
    \Bigg( \lambda^{2-p} \mint{-}\limits_{Q^{(\lambda)}_{\varrho}(z_0)} \big|Du \big|^{p-1} \, dz 
    \Bigg)^p \Bigg]
    \nonumber\\ & \leq
    C(\mathfrak{C}) \varrho^p \Bigg[ \mint{-}\limits_{Q^{(\lambda)}_{\varrho}(z_0)} \big|Du-(Du)_{z_0;\varrho}^{(\lambda)} \big|^p \, dz 
    +
    \lambda^{p(2-p)} \Bigg( \mint{-}\limits_{Q^{(\lambda)}_{\varrho}(z_0)} \big|Du-(Du)_{z_0;\varrho}^{(\lambda)} \big|^{p} \, dz 
    \Bigg)^{p-1} 
    \nonumber \\ & \qquad +
    \big|(Du)_{z_0;\varrho}^{(\lambda)} \big|^p + \lambda^{p(2-p)} \big|(Du)_{z_0;\varrho}^{(\lambda)} \big|^{p(p-1)} \Bigg].
\end{align}
The $L^p$-norm on the right must now be further estimated in terms of the $V$-function. In the sub-quadratic case, this is accomplished with inequality (\ref{2.14}) of Lemma \ref{Lem2.1}, and in the super-quadratic case, it is trivial. We obtain
\begin{align}\label{9.02}
    \mint{-}\limits_{Q^{(\lambda)}_{\varrho}(z_0)} & \big|Du-(Du)_{z_0;\varrho}^{(\lambda)} \big|^p \, dz 
    \leq
    C(p) \mint{-}\limits_{Q^{(\lambda)}_{\varrho}(z_0)} \Big| V_{|(Du)_{z_0;\varrho}^{(\lambda)}|} \big( Du - (Du)_{z_0;\varrho}^{(\lambda)} \big)\Big|^2 
    \nonumber \\ & \qquad+
    \mathds{1}_{\{p<2\}} \big| (Du)_{z_0;\varrho}^{(\lambda)} \big|^{\frac{p(2-p)}{2}} \Big| V_{|(Du)_{z_0;\varrho}^{(\lambda)}|} \big( Du - (Du)_{z_0;\varrho}^{(\lambda)} \big)\Big|^p \, dz
    \nonumber\\ & \leq
    C(p) \Big[ \Phi_\lambda\big(z_0,\varrho, (Du)_{z_0;\varrho}^{(\lambda)} \big) + \mathds{1}_{\{p<2\}} \big| (Du)_{z_0;\varrho}^{(\lambda)} \big|^{\frac{p(2-p)}{2}} \Phi_\lambda\big(z_0,\varrho, (Du)_{z_0;\varrho}^{(\lambda)} \big)^\frac{p}{2} \Big].
\end{align}
Note that we can replace the cylindrical mean in the previous calculation with any matrix $A\in \R^{Nn}$. Taking into account the two previous inequalities, we have for the additional term (\ref{8.04}) of the hybrid excess functional (\ref{8.02}) that
\begin{align}\label{9.03}
    H_\lambda(z_0,\varrho)& \leq C(\mathfrak{C}) \big|(Du)_{z_0;\varrho}^{(\lambda)} \big|^{p} \varrho^\beta 
    \Bigg( \bigg[
    \Phi_\lambda\big(z_0,\varrho, (Du)_{z_0;\varrho}^{(\lambda)} \big) 
    + 
    \mathds{1}_{\{p<2\}} \big| (Du)_{z_0;\varrho}^{(\lambda)} \big|^{\frac{p(2-p)}{2}} \Phi_\lambda\big(z_0,\varrho, (Du)_{z_0;\varrho}^{(\lambda)} \big)^\frac{p}{2} 
    \nonumber\\ \qquad & +
    \big|(Du)_{z_0;\varrho}^{(\lambda)} \big|^p + \lambda^{p(2-p)} \big|(Du)_{z_0;\varrho}^{(\lambda)} \big|^{p(p-1)} 
    +
    \lambda^{p(2-p)} \Big( \Phi_\lambda\big(z_0,\varrho, (Du)_{z_0;\varrho}^{(\lambda)} \big) 
    \nonumber \\ \qquad & + 
    \mathds{1}_{\{p<2\}} \big| (Du)_{z_0;\varrho}^{(\lambda)} \big|^{\frac{p(2-p)}{2}} \Phi_\lambda\big(z_0,\varrho, (Du)_{z_0;\varrho}^{(\lambda)} \big)^\frac{p}{2} \Big)^{p-1} \bigg]^{\frac{\beta}{p}} 
    + 
    ( 1+\lambda^{2-p})^{\frac{\beta}{2}} \Bigg).
\end{align}
With the previous estimate in mind, we begin with the proof of Theorem \ref{Theo1}.

\subsection{The non-degenerate regime (NDR)}

We now provide a excess-decay estimate for the non-degenerate regime,  which is characterized by (\ref{9.101})-(\ref{9.103}) below.
\begin{Lem}\label{Lem9.1}
    Let $2n/(n+2)< p\neq 2$, $\beta,\beta_0\in (0,1)$, $M,K\geq 1$ and
    \begin{equation*}
        u\in C\big([0,T];L^2\big(\Omega,\mathbb{R}^{N}\big) \big) \cap L^p\big(0,T;W^{1,p}\big(\Omega,\mathbb{R}^{N}\big) \big)
    \end{equation*}
    be a weak solution of system (\ref{Sys1}) in $\Omega_T$ that satisfies the assumptions of Section \ref{Assumptions}. Then, there exist constants
    \begin{align*}
        &\vartheta = \vartheta ( \mathfrak{C} , \beta , \beta_0, K ,\kappa_{8M}) \in (0,\min\{1/16,1/2^{\frac{1}{\beta}}],\\
        &\varepsilon_0 = \varepsilon_0(\mathfrak{C},\beta,\beta_0,K,\kappa_{8M},\omega_{4M}(\,\cdot\,)) \in (0,1] 
    \end{align*}
    for which the following holds true: Let $Q_\varrho^{(\lambda)}\subset \Omega_T$ be a intrinsic geometric cylinder with parameters $\varrho\in (0,1]$ and $0<\lambda\leq 1$ such that there holds
    \begin{align}\label{9.101}
        \big|(Du)_{z_0;\varrho}^{(\lambda)}\big| \leq 2M,
    \end{align}
    \begin{align}\label{9.102}
        \frac{\lambda}{2K} \leq \big|(Du)_{z_0;\varrho}^{(\lambda)}\big| \leq 2K\lambda
    \end{align}
    and the smallness condition
    \begin{align}\label{9.103}
        E_\lambda(z_0,\varrho)
        \leq
        \varepsilon_0 \big|(Du)_{z_0;\varrho}^{(\lambda)}\big|^p
    \end{align}
    is satisfied. Then we have the following excess-decay estimate:
    \begin{align}\label{9.104}
        E_\lambda(z_0,\vartheta\varrho) 
        \leq 
        \vartheta^{\beta \beta_0} E_\lambda (z_0,\varrho).
    \end{align}
\end{Lem}
\begin{proof}
    Without loss of generality we assume $z_0=0$. We abbreviate $A=(Du)_{\varrho}^{(\lambda)}$, $\Phi_\lambda(\varrho) = \Phi_\lambda (0,\varrho, (Du)_{\varrho}^{(\lambda)})$, $E_\lambda(\varrho)=E_\lambda (0,\varrho)$ and $H^*_\lambda(\varrho)=H^*_\lambda (0,\varrho,(u)_\varrho^{(\lambda)})$. Due to $|A|>0$, we have  $E_\lambda(\varrho)>0$. Moreover, we assume that (\ref{9.103}) holds for some $\varepsilon_0\in (0,1]$ with $\varepsilon_0$ to be determined later. To shorten the notation further and to unify the sub- and super-quadratic case, we define
    \begin{equation*}
        \varrho_* = \begin{cases}
            \varrho/4 &\text{for } p<2, \\
            \varrho & \text{for } p>2
        \end{cases}
        \qquad\text{and}\qquad
        \gamma= \bigg( \frac{E_\lambda(\varrho)}{|A|^{p}}\bigg)^{\frac12}
    \end{equation*}
    as well as the auxiliary functions
    \begin{align}\label{9.10}
        v(x,t) =
        \begin{cases}
            u(x,t) - (u)_{\varrho/2}^{(\lambda)} - A x &\text{for } p<2, \\
            u(x,t) - (u)_{\varrho}^{(\lambda)} - A x\varrho & \text{for } p>2
        \end{cases} 
        \qquad \text{for } (x,t)\in Q_\varrho^{(\lambda)}
    \end{align}
    and
    \begin{align}\label{9.11}
        w(x,t) = \big|(Du)_{\varrho}^{(\lambda)}\big|^{\frac{p-2}{2}} \frac{v(x,\lambda^{2-p}t)}{C_1 (E_\lambda(\varrho))^{\frac{1}{2}}}
        \qquad \text{for } (x,t)\in Q_\varrho \equiv Q_\varrho^{(1)},
    \end{align}
    where the constant $C_1\geq 1$ will be specified later. Note that, in contrast to the super-quadratic case, we use the mean of the map $u$ on the cylinder $Q_{\varrho/2}^{(\lambda)}$ instead of $Q_{\varrho}^{(\lambda)}$. This is necessary since the Sobolev-Poincaré inequality of Lemma \ref{Lem5.4} will be applied with $A=(Du)_{\varrho}^{(\lambda)}$ instead of the Poincaré inequality of Lemma \ref{Lem5.3}. We divide the proof into three steps. \\
    \textbf{Step 1: Approximation.} In the first step, we again show that the assumptions of the $\mathcal{A}$-caloric approximation Lemmas \ref{Lem2.12} and \ref{Lem2.13} are satisfied in the respective cases. We start with the \textbf{super-quadratic case}: To establish the $L^2$- and $L^p$-estimate of assumption (\ref{2.1202}), we apply the Poincaré type inequality from Lemma \ref{Lem5.3} with our particular choice of $A$ and note $|A|\leq 2M$. Using (\ref{9.102}) as well as Hölder's and Minkowski's inequality for $s=2,p$, we obtain that
    \begin{align}\label{8.12}
        \mint{-}\limits_{Q_{\varrho}^{(\lambda)}} & \bigg|\frac{v}{\varrho}\bigg|^s \, dz
        \leq 
        C \Bigg[ \mint{-}\limits_{Q_{\varrho}^{(\lambda)}} |Dv|^s \, dz + \Bigg( \lambda^{2-p} \mint{-}\limits_{Q_{\varrho}^{(\lambda)}} |A|^{p-2} |Dv| + |Dv|^{p-1} \, dz 
        +
        \lambda^{2-p} |A|^{p-1} H^*_\lambda(\varrho)
        \Bigg)^s \Bigg]
        \nonumber\\ & \leq
        C (1 + K^{s(p-2)}) \Bigg[\mint{-}\limits_{Q_{\varrho}^{(\lambda)}} |Dv|^s \, dz
        +
        |A|^{s(2-p)} \bigg( \mint{-}\limits_{Q_{\varrho}^{(\lambda)}} |Dv|^p \, dz \bigg)^{\frac{s(p-1)}{p}}
        +
        |A|^{s(1-p)}  E_\lambda^s(\varrho)\Bigg],
    \end{align}
    where we have for the constant $C=C(\mathfrak{C},\kappa_{4M})$. Now we need to rewrite the preceding inequality in terms of $w$. To achieve this, we multiply (\ref{8.12}) by $|A|^{p-s}$, recall the definition (\ref{8.02}) of $E_\lambda(\varrho)$ and use hypothesis (\ref{9.103}) with $K\geq 1$ to estimate
    \begin{align*}
        |A|^{p-s} \mint{-}\limits_{Q_{\varrho}^{(\lambda)}} \bigg|\frac{v}{\varrho}\bigg|^s \, dz 
        &\leq
        C E_\lambda(\varrho) 
        + 
        C \bigg[ \varepsilon_0^{\frac{s(p-2)-(p-s)}{p}} \big( E_\lambda(\varrho) \big)^{-\frac{s(p-2)-(p-s)}{p}} \big( E_\lambda(\varrho) \big)^{\frac{s(p-1)}{p}} 
        \nonumber\\ & \qquad +
        C \varepsilon_0^{1-s} \big(E_\lambda(\varrho)\big)^{1-s} \big(E_\lambda(\varrho)\big)^s
        \bigg]
        \leq
        C E_\lambda(\varrho) 
    \end{align*}
    for a constant $C=C(\mathfrak{C},K,\kappa_{4M})$. We then infer from the definitions of $w$ and $\Phi_\lambda$, as well as the previous inequality for the parameters $s=2$ and $s=p$, that
    \begin{align*}
        \mint{-}\limits_{Q_{\varrho_*}^{(\lambda)}} & \bigg|\frac{w}{\varrho}\bigg|^2 + |Dw|^2 \, dz 
        + \gamma^{p-2}
        \mint{-}\limits_{Q_{\varrho_*}^{(\lambda)}} \bigg|\frac{w}{\varrho}\bigg|^p + |Dw|^p \, dz 
        \nonumber \\ & =
        \frac{|A|^{p-2}}{C_1^2 E_\lambda(\varrho)} 
        \mint{-}\limits_{Q_{\varrho}^{(\lambda)}} \bigg|\frac{v}{\varrho}\bigg|^2 + |Dv|^2 \, dz 
        +
        \frac{1}{C_1^p E_\lambda(\varrho)} 
        \mint{-}\limits_{Q_{\varrho}^{(\lambda)}} \bigg|\frac{v}{\varrho}\bigg|^p + |Dv|^p \, dz 
        \nonumber \\ & \leq
        \frac{CK^{2(p-2)}+1}{C_1^2 }
        +
         \frac{CK^{p(p-2)}+1}{C_1^p } \leq 1,
    \end{align*}
    provided we choose $C_1$ large enough. For the dependency of the constant $C_1$, we obtain $C_1=C_1(\mathfrak{C},K,\kappa_{4M})$. In the \textbf{sub-quadratic case}, we verify the $L^2$-estimate (\ref{2.1302}) of Lemma \ref{Lem2.13}. From the Sobolev-Poincaré inequality of Lemma \ref{Lem5.4} and the upper bound (\ref{9.103}), we infer with (\ref{9.02}) that
    \begin{align}\label{9.12}
        \mint{-}\limits_{Q_{\varrho/2}^{(\lambda)}} \bigg|\frac{v}{\varrho/2}\bigg|^2 \, dz
        &\leq 
        C 
        \Bigg[ \lambda^{2-p} |A|^{p} H_\lambda^*(\varrho)
        +
        \Bigg( \mint{-}\limits_{Q^{(\lambda)}_{\varrho}} |Dv|^p \, dz 
        + 
        |A|^{p} \big(H_\lambda^*(\varrho)\big)^p\Bigg)^{\frac{2}{p}} \Bigg]
        \nonumber \\ &\leq 
        C
        \bigg[ \lambda^{2-p} E_\lambda(\varrho) + \Big( \Phi_\lambda(\varrho) + |A|^{\frac{p(2-p)}{2}} \big( \Phi_\lambda(\varrho) \big)^{\frac{p}{2}}
        + 
        |A|^{\frac{p(2-p)}{2}} E_\lambda(\varrho) \Big)^{\frac{2}{p}} \bigg]
        \nonumber \\ &\leq 
        C
        \big( E_\lambda(\varrho)^{\frac{2}{p}} + |A|^{2-p} E_\lambda(\varrho)
        \big)
        \leq 
        C
        |A|^{2-p} E_\lambda(\varrho),
    \end{align}
    where the dependency of the constant $C$ is given by $C=C(\mathfrak{C},K,\kappa_{4M})$. Next, we use the supremum part of the Caccioppoli inequality of Lemma \ref{Lem5.1} and (\ref{9.102}) to estimate
    \begin{align*}
        \sup\limits_{\Lambda_{\varrho/4}^{(\lambda)}} \mint{-}\limits_{B_{\varrho/4}} \bigg|\frac{v(\,\cdot\,,t)}{\varrho/4}\bigg|^2 \, dx 
        &\leq
        C \mint{-}\limits_{Q_{\varrho}^{(\lambda)}} \lambda^{2-p} |A|^{p-2} \bigg|\frac{v}{\varrho/2}\bigg|^2 + \bigg|\frac{v}{\varrho/2}\bigg|^2 \, dz 
        +
        \lambda^{2-p} H_\lambda(\varrho)
        \leq 
        C |A|^{2-p} E_\lambda(\varrho),
    \end{align*}
    where again $C=C(\mathfrak{C},\kappa_{4M})$. We then infer from the definitions of $w$ and $E_\lambda$, as well as the preceding inequality, the supremum-estimate
    \begin{align}\label{9.14}
        \sup\limits_{\Lambda_{\varrho_*}^{(\lambda)}} \mint{-}\limits_{B_{\varrho_*}} \bigg|\frac{w(\,\cdot\,,t)}{\varrho/4}\bigg|^2 \, dx 
        =
        \frac{|A|^{p-2}}{C_1^2 E_\lambda (\varrho)}
        \sup\limits_{\Lambda_{\varrho/4}^{(\lambda)}} \mint{-}\limits_{B_{\varrho/4}} \bigg|\frac{v(\,\cdot\,,t)}{\varrho/4} \bigg|^2 \, dx 
        \leq
        \frac{C}{C_1^2}.
    \end{align}
    We continue with the estimation of $|V_1(Dw)|$. To this end, we use the definitions (\ref{6.01}) and (\ref{9.11}) of $\Phi_\lambda$ and $w$, respectively, as well as the inequalities $\Phi_\lambda \leq  E_\lambda$, $\varepsilon_0\leq 1$ and $C_1 \geq 1$. This yields with (\ref{9.103}) that
    \begin{align}\label{9.141}
        \mint{-}\limits_{Q_{\varrho_*}^{(\lambda)}} \big| V_1 (Dw) \big|^2 \, dz
        &=
        \frac{|A|^{p-2}}{C_1^2 E_\lambda (\varrho)}
        \mint{-}\limits_{Q_{\varrho/4}^{(\lambda)}} \bigg( 1 + \frac{|A|^{p-2}}{C_1^2 E_\lambda (\varrho)} |Dv|^2 \bigg)^{\frac{p-2}{2}} |Dv|^2 \, dz
        \nonumber \\ & =
        \frac{1}{C_1^2 E_\lambda (\varrho)}
        \mint{-}\limits_{Q_{\varrho/4}^{(\lambda)}} \bigg( |A|^{2} + \frac{|A|^{p}}{C_1^2 E_\lambda (\varrho)} |Dv|^2 \bigg)^{\frac{p-2}{2}} |Dv|^2 \, dz
        \nonumber \\ & \leq
        \frac{1}{C_1^2 E_\lambda (\varrho)}
        \mint{-}\limits_{Q_{\varrho/4}^{(\lambda)}} \bigg( |A|^{2} + \frac{1}{C_1^2 \varepsilon_0} |Dv|^2 \bigg)^{\frac{p-2}{2}} |Dv|^2 \, dz
        \nonumber \\ & \leq
        \frac{4^{n+2}}{C_1^2 E_\lambda (\varrho)}
        \mint{-}\limits_{Q_{\varrho}^{(\lambda)}} \frac{1}{C_1^{p-2}} \big|V_{|A|}(Dv)\big|^2 \, dz
        \leq
        \frac{4^{n+2}}{C_1^p}.
    \end{align}
    If we now join (\ref{9.14}) and (\ref{9.141}), we obtain
    \begin{align}\label{9.142}
        \sup\limits_{\Lambda_{\varrho/4}^{(\lambda)}} \mint{-}\limits_{B_{\varrho/4}} \bigg|\frac{w(\,\cdot\,,t)}{\varrho/4}\bigg|^2 \, dx 
        +
        \mint{-}\limits_{Q_{\varrho/4}^{(\lambda)}} \big| V_1 (Dw) \big|^2 \, dz
        \leq
        \frac{C}{C_1^2} + \frac{4^{n+2}}{C_1^p}\leq 1,
    \end{align}
provided $C_1$ is chosen large enough. Therefore, assumption (\ref{2.1302}) of Lemma \ref{Lem2.13} is satisfied. Furthermore, for the dependence of the constant $C_1$, we obtain $C_1=C_1(\mathfrak{C},K,\kappa_{4M})$. We now verify assumption (\ref{2.1203}) on $Q_{\varrho_*}$ with the help of Lemma \ref{Lem6.1} for both the sub- and super-quadratic case. Since 
    \begin{align*}
        \mint{-}\limits_{Q_{\varrho_*}^{(\lambda)}} (u-v) \cdot \partial_t \phi \, dz = 0 \qquad\forall \phi\in C^1_0(Q_{\varrho_*}^{(\lambda)},\R^N),
    \end{align*}
    we obtain
    \begin{align}\label{8.16}
        \Bigg| \mint{-}\limits_{Q_{\varrho_*}^{(\lambda)}} & v\cdot \partial_t \phi  - (D_\xi a) \big(0,(u)_{\varrho_*}^{(\lambda)},A \big)(Dv,D\phi) \, dz \Bigg| 
        \nonumber \\ & \leq 
        C |A|^{\frac{p-2}{2}} \big(E_\lambda(\varrho) \big)^{\frac{1}{2}} 
        \Bigg[ \bigg( \frac{E_\lambda(\varrho)}{|A|^p} \bigg)^{\left|\frac{1}{2}-\frac{1}{p}\right|} + \omega_{4M} \bigg( \frac{E_\lambda(\varrho)}{|A|^p} \bigg) 
        +
        \bigg( \frac{E_\lambda(\varrho)}{|A|^p} \bigg)^{\frac{1}{2}} \Bigg] \sup\limits_{Q_{\varrho_*}^{(\lambda)}} |D\phi|
    \end{align}
    for all $\phi\in C^1_0\big( Q_{\varrho_*}^{(\lambda)}(z_0),  \mathbb{R}^{N} \big)$. Here, we used in the sub-quadratic case that a simple calculation shows $E_\lambda^*(0,\varrho/4,A)\leq C(n,p) E_\lambda^*(0,\varrho,A)$. Note that for the dependency of the constant on the right-hand side, we have $C=C(p,\kappa_{4M}) L$. Now recall the definition (\ref{9.10}) of $w$. Then we obtain by a transformation in the time variable and the assumptions (\ref{9.102}) and (\ref{9.103}) that (\ref{8.16}) can be rewritten and estimated as follows:
    \begin{align}\label{8.17}
        \Bigg| \mint{-}\limits_{Q_{\varrho_*}} & w\cdot \partial_t \phi  - \frac{(D_\xi a)(0,(u)_{\varrho_*}^{(\lambda)},A)}{\lambda^{p-2}}(Dw,D\phi) \, dz \Bigg| 
        \nonumber\\& \leq 
        \frac{C |A|^{p-2} }{C_1\lambda^{p-2}}
        \Bigg[ \bigg( \frac{E_\lambda(\varrho)}{|A|^p} \bigg)^{\left|\frac{1}{2}-\frac{1}{p}\right|} + \omega_{4M} \bigg( \frac{E_\lambda(\varrho)}{|A|^p} \bigg) 
        +
        \bigg( \frac{E_\lambda(\varrho)}{|A|^p} \bigg)^{\frac{1}{2}} \Bigg] \sup\limits_{Q_{\varrho_*}} |D\phi|
        \nonumber\\& \leq 
        \Bigg[ \bigg( \frac{E_\lambda(\varrho)}{|A|^p} \bigg)^{\left|\frac{1}{2}-\frac{1}{p}\right|} + \omega_{4M} \bigg( \frac{E_\lambda(\varrho)}{|A|^p} \bigg) 
        +
        \bigg( \frac{E_\lambda(\varrho)}{|A|^p} \bigg)^{\frac{1}{2}} \Bigg] \sup\limits_{Q_{\varrho_*}} |D\phi|,
    \end{align}
    for any $\phi\in C^1_0(Q_\varrho,\R^N)$, where the last inequality holds due to (\ref{9.102}) after a possible enlargement of $C_1$, resulting in an additional dependence on $p,L,\kappa_{4M}$ and $K$.  This implies that the chosen $C_1$ depends on $\mathfrak{C}$, $K$ and $\kappa_{4M}$. Now define the elliptic bilinear form $\mathcal{A}$ by
    \begin{align*}
        \mathcal{A}(\xi,\Tilde{\xi}) = \frac{(D_\xi a)(0,(u)_{\varrho_*}^{(\lambda)},A)}{\lambda^{p-2}}(\xi,\Tilde{\xi}) 
        \qquad\text{for }\xi,\Tilde{\xi} \in \R^{Nn}.
    \end{align*}
    From the growth assumptions (\ref{As1}), (\ref{As3}) and (\ref{9.102}), we deduce that $\mathcal{A}$ satisfies the following ellipticity and growth conditions:
    \begin{align*}
        \mathcal{A}(\xi,\xi) \geq \nu (2K)^{p-2} |\xi|^2 , \qquad
        \mathcal{A}(\xi,\Tilde{\xi}) \leq L \kappa_{2M} (2K)^{2-p} |\xi||\Tilde{\xi}|
    \end{align*}
    for all $\xi,\Tilde{\xi} \in \R^{Nn}$. Given $\varepsilon>0$, which we will specified later, let  
    \begin{equation*}
        \delta=\delta(n,N,p,\nu\equiv\nu (2K)^{2-p}, L\equiv L \kappa_{2M} (2K)^{p-2}, \varepsilon)=\delta(\mathfrak{C},K,\kappa_{2M},\varepsilon)\in (0,1]
    \end{equation*}
    be the constants from Lemma \ref{Lem2.12} and \ref{Lem2.13} for the parameter choices $(\nu (2K)^{2-p},L \kappa_{2M} (2K)^{p-2})$ replacing $(\nu,L)$. If we impose the \textbf{smallness condition}
    \begin{align}\label{9.191}
        \bigg( \frac{E_\lambda(\varrho)}{|A|^p} \bigg)^{\left|\frac{1}{2}-\frac{1}{p}\right|} + \omega_{4M} \bigg( \frac{E_\lambda(\varrho)}{|A|^p} \bigg) 
        +
        \bigg( \frac{E_\lambda(\varrho)}{|A|^p} \bigg)^{\frac{1}{2}} \leq \delta,
    \end{align}
    we can apply $\mathcal{A}$-caloric approximation Lemmas \ref{Lem2.12} and \ref{Lem2.13} to $(w,\mathcal{A})$ in the super- and sub-quadratic case, respectively, with the choice of parameters $(\nu (2K)^{2-p},L \kappa_{2M} (2K)^{p-2})$ replacing $(\nu,L)$. This yields the existence of an $\mathcal{A}$-caloric function $h\in L^2\big(\Lambda_{\varrho_*/2};W^{1,2}(B_{\varrho_*/2},\R^N)\big)$ on $Q_{\varrho_*/2}$ satisfying in the sub-quadratic case that
    \begin{align}\label{9.192}
        \mint{-}\limits_{Q_{\varrho_*/2}} & \bigg|\frac{h}{\varrho_*/2}\bigg|^2 + \big|V_1(Dh)\big|^2 \, dz 
        \leq 2^{n+5}
        \qquad\text{and}\qquad
        \mint{-}\limits_{Q_{\varrho_*/2}} \bigg|\frac{w-h}{\varrho_*/2}\bigg|^2 \, dz 
        \leq \varepsilon
    \end{align}
    or in the super-quadratic case that
    \begin{align}\label{8.192}
        \mint{-}\limits_{Q_{\varrho_*/2}} & \bigg|\frac{h}{\varrho_*/2}\bigg|^2 + |Dh|^2 \, dz 
        + \gamma^{p-2}
        \mint{-}\limits_{Q_{\varrho_*/2}} \bigg|\frac{h}{\varrho_*/2}\bigg|^p + |Dh|^p \, dz  \leq 2^{n+3+2p}
    \end{align}
    and
    \begin{align}\label{8.193}
        \mint{-}\limits_{Q_{\varrho_*/2}} & \bigg|\frac{w-h}{\varrho_*/2}\bigg|^2 \, dz 
        + \gamma^{p-2}
        \mint{-}\limits_{Q_{\varrho/2}} \bigg|\frac{w-h}{\varrho_*/2}\bigg|^p \, dz  \leq \varepsilon.
    \end{align}\\
    \textbf{Step 2: Linearization.} In the second step, we prove a decay estimate by appropriately linearizing the solution $u$. Here we use the a priori estimate of Lemma \ref{Lem7.2} for the $\mathcal{A}$-caloric function $h$, where we use the parameters $s\in \{2,\max\{2,p\}\}$ and $\theta\in (0,1/4]$. In the \textbf{super-quadratic case}, we deduce from Lemma \ref{Lem7.2} and inequality (\ref{8.192}) that
    \begin{align}\label{8.194}
        \frac{1}{(2\theta\varrho)^s} \mint{-}\limits_{Q_{2\theta\varrho_*}} & \big| h - (h)_{2\theta\varrho_*} - (Dh)_{2\theta\varrho_*} x \big|^s \, dz 
        \leq
        C_{pa} (4\theta)^s \bigg(\frac{2}{\varrho_*}\bigg)^{s} \mint{-}\limits_{Q_{\varrho_*/2}} \big| h - (h)_{\varrho_*/2} - (Dh)_{\varrho_*/2} x \big|^s \, dz
        \nonumber\\ &\leq
        2^{4s-1} C_{pa} \theta^s \Bigg[ \bigg(\frac{2}{\varrho_*}\bigg)^{s} \mint{-}\limits_{Q_{\varrho_*/2}}  |h|^s + \left|(h)_{\varrho_*/2}\right|^s \, dz + \big| (Dh)_{\varrho_*/2} \big|^s  \Bigg]
        \nonumber\\ &\leq
        2^{4s} C_{pa} \theta^s \Bigg[ \bigg(\frac{2}{\varrho_*}\bigg)^{s} \mint{-}\limits_{Q_{\varrho_*/2}}  |h|^s \, dz + \mint{-}\limits_{Q_{\varrho_*/2}}  |Dh|^s \, dz  \Bigg]
        \leq
        2^{n+8 p} C_{pa} \gamma^{2-s} \theta^s,
    \end{align}
    where $C_{pa}=C_{pa}(\mathfrak{C},K,\kappa_{2M})$. In the \textbf{sub-quadratic case}, we calculate with the help of Hölder's inequality and the inequalities (\ref{2.14}) as well as (\ref{9.192}) that
    \begin{align}\label{9.194}
        \frac{1}{(2\theta\varrho_*)^s} &\mint{-}\limits_{Q_{2\theta\varrho_*}}  \big| h - (h)_{2\theta\varrho_*} - (Dh)_{2\theta\varrho_*} x \big|^s \, dz 
        \leq
        2^{4s} C_{pa} \theta^s \Bigg[ \bigg(\frac{2}{\varrho_*}\bigg)^{s} \mint{-}\limits_{Q_{\varrho_*/2}}  |h|^s \, dz + \bigg( \mint{-}\limits_{Q_{\varrho_*/2}}  |Dh| \, dz \bigg)^s \Bigg]
        \nonumber\\ &\leq
        C(p) C_{pa} \theta^2 \Bigg[ \bigg(\frac{2}{\varrho_*}\bigg)^{s} \mint{-}\limits_{Q_{\varrho_*/2}}  |h|^s \, dz + \bigg( \mint{-}\limits_{Q_{\varrho_*/2}}  \big|V_1(Dh)\big|^2 + \big|V_1(Dh)\big|^{\frac{2}{p}} \, dz \bigg)^s \Bigg]
        \nonumber\\ &\leq
        C(p) C_{pa} \theta^2 \Bigg[ \bigg(\frac{2}{\varrho_*}\bigg)^{s} \mint{-}\limits_{Q_{\varrho_*/2}}  |h|^s \, dz + \bigg( \mint{-}\limits_{Q_{\varrho_*/2}}  \big|V_1(Dh)\big|^2 \, dz \bigg)^s + \bigg( \mint{-}\limits_{Q_{\varrho_*/2}}  \big|V_1(Dh)\big|^2 \, dz \bigg)^{\frac{s}{p}} \Bigg]
        \nonumber\\ &\leq
        C(p) C_{pa} \theta^s,
    \end{align}
    where $C_{pa}=C_{pa}(\mathfrak{C},K,\kappa_{2M})$. The previous two inequalities can now be used to obtain a quantified decay estimate for $w$. From (\ref{9.192}) and (\ref{8.193}), we obtain in both the sub- and super-quadratic case that
    \begin{align}\label{9.195}
        \frac{1}{(2\theta\varrho_*)^s} \mint{-}\limits_{Q_{2\theta\varrho_*}} & \big| w - (h)_{2\theta\varrho_*} - (Dh)_{2\theta\varrho_*} x \big|^s \, dz 
        \nonumber\\ &\leq
        \frac{2^{s-1}}{(2\theta\varrho_*)^s}  \Bigg[\mint{-}\limits_{Q_{2\theta\varrho_*}} | w - h |^s \, dz 
        +
        \mint{-}\limits_{Q_{2\theta\varrho_*}} \big| h - (h)_{2\theta\varrho_*} - (Dh)_{2\theta\varrho_*} x \big|^s \, dz 
        \Bigg]
        \nonumber\\ &\leq
        C(n,p) \Bigg[ \theta^{-n-2-s} (\varrho_*/2)^{-s} \mint{-}\limits_{Q_{\varrho_*/2}} | w - h |^s \, dz + C_{pa} \gamma^{2-s} \theta^s \Bigg]
        \nonumber\\ &\leq
        C(\mathfrak{C},K,\kappa_{2M}) \gamma^{2-s} \big[ \theta^{-n-2-s} \varepsilon +
         \theta^s \big].
    \end{align}
    To further estimate the right-hand side, we choose $\varepsilon:= \theta^{n+4+2p}$, where the parameter $\theta\in (0,1/4]$ remains to be chosen later. If we now define
    \begin{equation*}
        \vartheta= \begin{cases}
            \theta/4 &\text{for }p<2,\\
            \theta &\text{for }p>2,
        \end{cases}
    \end{equation*}
    and recall the definitions (\ref{9.10}) and (\ref{9.11}) of $v$ and $w$ respectively, we can rewrite the latter inequality in terms of $\varrho$ and $u$. This yields that
    \begin{align*}
        \frac{1}{(2\vartheta\varrho)^s} \mint{-}\limits_{Q_{2\vartheta\varrho}^{(\lambda)}} & \Big| u - (u-v) -  C_1 |A|^{\frac{2-p}{2}} \sqrt{E_\lambda(\varrho)} \big[(h)_{2\vartheta\varrho} - (Dh)_{2\vartheta\varrho} x \big] \Big|^s \, dz 
        \nonumber\\ &\leq
        C C_1^s \vartheta^2 \gamma^2 |A|^{s}
        \leq
         C |A|^{s-p}
        \vartheta^s E_\lambda (\varrho),
    \end{align*}
    where we have for the constant $C=C(\mathfrak{C},K,\kappa_{2M})$. Let $L_{2\vartheta\varrho}^{(\lambda)}: \R^n\to \R^N x \mapsto l_{2\vartheta\varrho}^{(\lambda)} + (Dl)_{2\vartheta\varrho}^{(\lambda)}x$ be the unique affine minimizer in accordance to section \ref{SMinaff} and cylinder $Q_{2\vartheta\varrho}^{(\lambda)}$. Then the previous estimate (\ref{9.195}) and an application of Lemma \ref{Lem2.8} in the super-quadratic case imply 
    \begin{align}\label{9.198}
        \mint{-}\limits_{Q_{2\vartheta\varrho}^{(\lambda)}} \bigg|\frac{u-L_{2\vartheta\varrho}^{(\lambda)}}{\vartheta\varrho} \bigg|^s \, dz
        \leq
        C |A|^{s-p}
         \vartheta^s E_\lambda (\varrho),
    \end{align}
    where again $C=C(\mathfrak{C},K,\kappa_{2M})$. In order to proceed further, it will be necessary to replace $A$ by $(Dl)_{2\vartheta\varrho}^{(\lambda)}$ in (\ref{9.198}). For this purpose, note that (\ref{2.5.3}) implies the following estimate for the difference:
    \begin{align*}\label{9.199}
        \big|(Dl)_{2\vartheta\varrho}^{(\lambda)} - (Du)_{\varrho}^{(\lambda)}\big|^2 
        \leq 
        n(n+2) \mint{-}\limits_{Q_{2\vartheta\varrho}^{(\lambda)}} \bigg|\frac{ u - (u)_{2\vartheta\varrho}^{(\lambda)} - (Du)_{\varrho}^{(\lambda)} x}{2\vartheta\varrho} \bigg|^2 \, dz=
        n(n+2) \mint{-}\limits_{Q_{2\vartheta\varrho}^{(\lambda)}} \bigg|\frac{\Tilde{v}}{2\vartheta\varrho} \bigg|^2 \, dz,
    \end{align*}
    where we defined $\Tilde{v}=u - (u)_{2\vartheta\varrho}^{(\lambda)} - (Du)_{\varrho}^{(\lambda)} x$. Let us now estimate the right-hand side further. To this end, we notice that
    \begin{align*}
        \mint{-}\limits_{Q_{2\vartheta\varrho}^{(\lambda)}} & \big| u - (u)_{2\vartheta\varrho}^{(\lambda)}\big|^p \, dz
        \leq
        2^{p-n-2}\vartheta^{-n-2} \mint{-}\limits_{Q_{\varrho}^{(\lambda)}} \big| u - (u)_{\varrho}^{(\lambda)}\big|^p \, dz,
    \end{align*}
    which implies
    \begin{equation*}
        H^*_\lambda (2\vartheta\varrho) \leq C(n,p) \vartheta^{-\frac{\beta(n+2)}{p}} H^*_\lambda (\varrho)
        \qquad\text{and similarly}\qquad
        H^*_\lambda (4\vartheta\varrho) \leq C(n,p) \vartheta^{-\frac{\beta(n+2)}{p}} H^*_\lambda (\varrho).
    \end{equation*}
    We proceed by distinguishing the sub- and super-quadratic case. In the \textbf{super-quadratic case}, we use the Poincaré inequality from Lemma \ref{Lem5.3} with the parameter choice $s=2$ exactly as in (\ref{8.12}) to obtain that
    \begin{align}\label{8.1991}
        \mint{-}\limits_{Q_{2\vartheta\varrho}^{(\lambda)}} & \bigg|\frac{\Tilde{v}}{2\vartheta\varrho} \bigg|^2 \, dz
        \leq 
        C \Bigg[ \mint{-}\limits_{Q_{2\vartheta\varrho}^{(\lambda)}} |Dv|^2 \, dz 
        + 
        \Bigg( \lambda^{2-p} \mint{-}\limits_{Q_{2\vartheta\varrho}^{(\lambda)}} |A|^{p-2} |Dv| 
        + 
        |Dv|^{p-1} \, dz 
        +
        \lambda^{2-p} |A|^{p-1}  H_\lambda^*( 2\vartheta\varrho)
        \Bigg)^2 \Bigg]
        \nonumber\\ & \leq
        \frac{C}{\vartheta^{n+2}} \mint{-}\limits_{Q_{\varrho}^{(\lambda)}} |Dv|^2 \, dz
        +
        C |A|^{2(2-p)} \bigg( \frac{1}{\vartheta^{n+2}} \mint{-}\limits_{Q_{\varrho}^{(\lambda)}} |Dv|^p \, dz \bigg)^{\frac{2(p-1)}{p}}
        +
        C  |A|^{2(1-p)} \bigg( \frac{E_\lambda(\varrho)}{\vartheta^{\frac{(n+2)\beta}{p}}} \bigg)^2
        \nonumber\\ &\leq
        \frac{C }{\vartheta^{2(n+2)}} \Big[|A|^{2-p} E_\lambda(\varrho)
        +
        C |A|^{2(2-p)} \big( E_\lambda(\varrho) \big)^{\frac{2(p-1)}{p}}
        +
        |A|^{2(1-p)} \big( E_\lambda(\varrho) \big)^2 \Big]
        \leq
        \frac{C}{\vartheta^{2(n+2)}} \varepsilon_0 |A|^2.
    \end{align}
    In the \textbf{sub-quadratic case}, we estimate with the Sobolev-Poincaré inequality from Lemma \ref{Lem5.4} exactly as in (\ref{9.12}), but with different radii. We obtain with (\ref{9.103}) that
    \begin{align}\label{9.1991}
        \mint{-}\limits_{Q_{2\vartheta\varrho}^{(\lambda)}} & \bigg|\frac{\Tilde{v}}{2\vartheta\varrho} \bigg|^2 \, dz
        \leq 
        C 
        \Bigg[ \lambda^{p-2} |A|^{p} H^*_\lambda (4\vartheta\varrho) 
        +
        \Bigg( \mint{-}\limits_{Q^{(\lambda)}_{4\vartheta\varrho}} |Dv|^p \, dz 
        + 
        |A|^{p} \big( H^*_\lambda (4\vartheta\varrho)\big)^p \Bigg)^{\frac{2}{p}} \Bigg]
        \nonumber \\ &\leq 
        C 
        \Bigg[ \vartheta^{-\frac{\beta(n+2)}{p}} \lambda^{p-2} E_\lambda(\varrho) + \Bigg( \vartheta^{-n-2} \mint{-}\limits_{Q^{(\lambda)}_{\varrho}} |Dv|^p \, dz 
        + 
        \vartheta^{-\beta(n+2)} |A|^{p(1-p)} E^p_\lambda(\varrho) \Bigg)^{\frac{2}{p}} \Bigg]
        \nonumber \\ & \leq 
        \frac{C}{\vartheta^{2(n+2)}}
        \bigg[ \lambda^{p-2} E_\lambda(\varrho) + \Big(  \Phi_\lambda(\varrho) + \lambda^{\frac{p(2-p)}{2}} \big( \Phi_\lambda(\varrho) \big)^{\frac{p}{2}}
        + 
        |A|^{p(1-p)} E^p_\lambda(\varrho) \Big)^{\frac{2}{p}}\bigg]
        \nonumber \\ &\leq 
        \frac{C}{\vartheta^{2(n+2)}}
        \Big[ \varepsilon_0 |A|^{2} +  \Big( \varepsilon_0 |A|^{p} + \varepsilon_0^{\frac{p}{2}} |A|^{p} + \varepsilon_0^p |A|^{p} \Big)^{\frac{2}{p}}
        \Big]
        \leq 
        \frac{C}{\vartheta^{2(n+2)}} \varepsilon_0
        |A|^{2}.
    \end{align}
    Using the derived inequality (\ref{8.1991}) and (\ref{9.1991}) yields 
    \begin{align}\label{9.1992}
        \big|(Dl)_{2\vartheta\varrho}^{(\lambda)} - A \big|^2 
        &\leq
        Cn(n+2) \vartheta^{-2(n+2)} \varepsilon_0
        |A|^{2}
        \leq
        |A|^2,
    \end{align}
    provided that the \textbf{smallness condition}
    \begin{align}\label{9.1993}
        C n(n+2) \vartheta^{-2(n+2)} \varepsilon_0 \leq 1
    \end{align}
    is satisfied. Note that the constant has the dependency $C=C(\mathfrak{C},K,\kappa_{4M})$. The inequality (\ref{9.1992}) in particular implies that 
    \begin{equation}\label{9.1994}
        \big| 
        (Dl)^{(\lambda)}_{2\vartheta\varrho} \big|\leq 2 |A| 
        \qquad\text{and}\qquad
        |A| \leq 2 \big| (Dl)^{(\lambda)}_{2\vartheta\varrho} \big|.
    \end{equation}
    \textbf{Step 3: Excess improvement.} In the third and final step, we use the estimate (\ref{9.198}) to obtain the desired excess improvement of the Lemma. We start again with the estimation of the right-hand side of the Caccioppoli inequality of Lemma \ref{Lem5.1}. To do this, we observe with inequalities (\ref{9.198}) \& (\ref{9.1994}) as well as hypotheses (\ref{9.102}) \& (\ref{9.103}) in a first step that
    \begin{align*}
        \mint{-}\limits_{Q_{2\vartheta\varrho}^{(\lambda)}} & \big|u-(u)_{2\vartheta\varrho}^{(\lambda)} \big|^p  \, dz
        \leq 
        2^{p-1} (2\vartheta\varrho)^p \mint{-}\limits_{Q_{2\vartheta\varrho}^{(\lambda)}} \bigg|\frac{u-L_{2\vartheta\varrho}^{(\lambda)}}{2\vartheta\varrho} \bigg|^p \, dz
        +
        2^{p-1} (2\vartheta\varrho)^p \big| (Dl)^{(\lambda)}_{2\vartheta\varrho} \big|^p
        \\ & \leq
        C (2\vartheta\varrho)^p\Big( \mathds{1}_{\{p<2\}} \big( |A|^{2-p} \vartheta^2  E_\lambda (\varrho) \big)^{\frac{p}{2}} + \mathds{1}_{\{p>2\}}\vartheta^p E_\lambda(\varrho)+ |A|^p \Big)
        \leq 
        C (2\vartheta\varrho)^p \lambda^p
        \leq 
        C (2\vartheta\varrho)^p.
    \end{align*}
    Using the inequalities (\ref{9.102}), (\ref{9.198}) and (\ref{9.1994}), then yields that
    \begin{align*}
        &\mint{-}\limits_{Q_{2\vartheta\varrho}^{(\lambda)}} \bigg| V_{|(Dl)^{(\lambda)}_{2\vartheta\varrho}|}\bigg(\frac{u-L_{2\vartheta\varrho}^{(\lambda)}}{2\vartheta\varrho} \bigg) \bigg|^2
        + \lambda^{p-2}
        \bigg| \frac{u-L_{2\vartheta\varrho}^{(\lambda)}}{2\vartheta\varrho} \bigg|^2
        \, dz +
        L \big| 
        (Dl)^{(\lambda)}_{2\vartheta\varrho} \big|^{p} \Bigg( \mint{-}\limits_{Q_{2\vartheta\varrho}^{(\lambda)}} \big|u-(u)_{2\vartheta\varrho}^{(\lambda)} \big|^p  \, dz \Bigg)^{\frac{\beta}{p}}
        \nonumber\\ & \qquad + L \big| 
        (Dl)^{(\lambda)}_{2\vartheta\varrho} \big|^{p} (2\vartheta\varrho)^\beta (1+\lambda^{2-p})^{\frac{\beta}{2}}
        \nonumber\\ & \leq C
        \mint{-}\limits_{Q_{2\vartheta\varrho}^{(\lambda)}} \mathds{1}_{\{p>2\}} \bigg| \frac{u-L_{2\vartheta\varrho}^{(\lambda)}}{2\vartheta\varrho} \bigg|^p + \Big(\big|(Dl)^{(\lambda)}_{2\vartheta\varrho}|^{p-2} + \lambda^{p-2} \Big)
        \bigg| \frac{u-L_{2\vartheta\varrho}^{(\lambda)}}{2\vartheta\varrho} \bigg|^2 \, dz
        +
        C |A|^{p} (2\vartheta\varrho)^\beta (1+\lambda^{2-p})^{\frac{\beta}{2}}
        \nonumber\\ & \leq 
        C \vartheta^\beta E_\lambda (\varrho),
    \end{align*}
    where $C=C(\mathfrak{C},K,\kappa_{8M})$ and the dependency on $\kappa_{8M}$ is due to (\ref{9.1994}) and (\ref{9.101}). Then the Caccioppoli inequality (\ref{Lem5.1}) yields
    \begin{align}\label{9.1996}
        \mint{-}\limits_{Q_{\vartheta\varrho}^{(\lambda)}} \big| V_{|(Dl)^{(\lambda)}_{2\vartheta\varrho}|}\big(Du - (Dl)^{(\lambda)}_{2\vartheta\varrho} \big) \big|^2
        \, dz
        \leq 
        C \vartheta^\beta E_\lambda (\varrho)
    \end{align}
    and Lemma \ref{Lem2.4} allows us to replace $(Dl)^{(\lambda)}_{2\vartheta\varrho}$ on the left-hand side by $(Du)^{(\lambda)}_{\vartheta\varrho}$. This results in
    \begin{align}\label{9.1997}
        \Phi_\lambda(\vartheta\varrho) & =
        \mint{-}\limits_{Q_{\vartheta\varrho}^{(\lambda)}} \big| V_{|A|}(Du -A ) \big|^2
        \, dz
        \leq C(p)
        \mint{-}\limits_{Q_{\vartheta\varrho}^{(\lambda)}} \big| V_{|(Dl)^{(\lambda)}_{2\vartheta\varrho}|}\big(Du - (Dl)^{(\lambda)}_{2\vartheta\varrho} \big) \big|^2
        \, dz
        \nonumber\\ & \leq
        C \vartheta^\beta E_\lambda (\varrho)
        =: C_f \vartheta^\beta E_\lambda (\varrho)
    \end{align}
    with the obvious definition of the constant $C_f=C_f(\mathfrak{C},K,\kappa_{8M})$. It remains to estimate the other terms of $E_\lambda(\vartheta\varrho)$. To this end, we use (\ref{9.1996}), (\ref{9.103}) and (\ref{9.02}) and Hölder's inequality to estimate the difference
    \begin{align}\label{9.1998}
        \big|(Dl)_{2\vartheta\varrho}^{(\lambda)} & - A\big| 
        \leq 
        \mint{-}\limits_{Q_{\vartheta\varrho}^{(\lambda)}} \big|Du -(Dl)_{2\vartheta\varrho}^{(\lambda)}  \big| \, dz
        \leq
        C(p) \Bigg[ \mint{-}\limits_{Q_{\vartheta\varrho}^{(\lambda)}} \big| V_{|(Dl)^{(\lambda)}_{2\vartheta\varrho}|}\big(Du - (Dl)^{(\lambda)}_{2\vartheta\varrho} \big) \big|^2
        \, dz 
        \nonumber\\ & \qquad + \mathds{1}_{\{p<2\}} \big| (Dl)_{2\vartheta\varrho}^{(\lambda)} \big|^{\frac{p(2-p)}{2}} \bigg( \mint{-}\limits_{Q_{\vartheta\varrho}^{(\lambda)}} \big| V_{|(Dl)^{(\lambda)}_{2\vartheta\varrho}|}\big(Du - (Dl)^{(\lambda)}_{2\vartheta\varrho} \big) \big|^2
        \, dz \bigg)^\frac{p}{2} \Bigg]^{\frac{1}{p}} 
        \nonumber \\ &\leq
        C(p) \Big[ C_f \vartheta^{\beta} \varepsilon_0 |A|^p
        + 
        \mathds{1}_{\{p<2\}} C_f^{\frac{p}{2}} \vartheta^{\frac{\beta p}{2}} \varepsilon^\frac{p}{2}_0 |A|^p \Big]^{\frac{1}{p}} 
        \leq
        \frac{1}{4}|A|.
    \end{align}
    For the last inequality additionally assumed that the \textbf{smallness assumption}
    \begin{align}\label{9.1999}
         \Big[ C_f
        + 
        C_f^{\frac{p}{2}}  \Big]^{\frac{1}{p}}
        \varepsilon_0^\frac{1}{2} \leq
         \frac{1}{4}
    \end{align}
    is satisfied. From the inequalities (\ref{9.1994}) and (\ref{9.1998}) we deduce that $|A| \leq 4 |B|$ and $|B| \leq 4 |A|$, where we use the abbreviation $B= (Du)_{\vartheta\varrho}^{(\lambda)}$. With this in mind, we can estimate the untreated term $H_\lambda(\vartheta\varrho)$ (see (\ref{9.02})) of the complete excess functional. By a matter of fact the smallness assumptions (\ref{9.103}) and (\ref{9.1999}) in combination with (\ref{9.1997}) yield
    \begin{align}\label{9.19991}
        H_\lambda(\vartheta\varrho) & \leq C (\vartheta\varrho)^\beta  |B|^{p} \bigg( \Phi_{\lambda}(\vartheta\varrho) 
        + 
        \mathds{1}_{\{p<2\}}|B|^{\frac{p(2-p)}{2}} \Phi^\frac{p}{2}_{\lambda}(\vartheta\varrho)
        +
        (1+\lambda^{2-p})^{\frac{\beta}{2}}
        \nonumber\\ & \qquad +
        \Big[ |B|^p \big( 1 + \lambda^{p(2-p)} |B|^{p(p-2)} \big) 
        + 
        \lambda^{p(2-p)} \big( \Phi_\lambda (\vartheta\varrho) + 
        \mathds{1}_{\{p<2\}} |B|^{\frac{p(2-p)}{2}} \Phi^\frac{p}{2}_{\lambda}(\vartheta\varrho) \big)^{p-1} \Big) \Big]^{\frac{\beta}{p}}  \bigg)
        \nonumber \\ & \leq
        C 4^p \vartheta^\beta \varrho^\beta |A|^p \bigg( 
         \big( C_f \vartheta^\beta \varepsilon_0 + \mathds{1}_{\{p<2\}} (C_f \vartheta^\beta \varepsilon_0)^{\frac{p}{2}} 4^p \big) |A|^p 
         + 
         (1+\lambda^{2-p})^{\frac{\beta}{2}}
        +
        \Big[ 4^p |A|^p \big( 1 
        \nonumber\\ & \qquad 
        + \lambda^{p(2-p)} 4^{p(p-2)} |A|^{p(p-2)} \big) 
        + 
        \lambda^{p(2-p)} |A|^{p(p-1)} \Big( C_f \vartheta^\beta \varepsilon_0 + \mathds{1}_{\{p<2\}} ( C_f \vartheta^\beta \varepsilon_0 )^{\frac{p}{2}} 4^p \Big)^{p-1} \Big]^{\frac{\beta}{p}} \bigg)
        \nonumber \\ & \leq
        C  \vartheta^\beta |A|^p \varrho^\beta (1+\lambda^{2-p})^{\frac{\beta}{2}}
        \leq
        C \vartheta^\beta E_\lambda(\varrho).
    \end{align}
    Combining (\ref{9.1997}) and (\ref{9.19991}) results in
    \begin{align*}
        E_\lambda (\vartheta\varrho) 
        \leq 
        C_e \vartheta^\beta E_\lambda (\varrho),
    \end{align*}
    where $C_e= C_e(\mathfrak{C},K,\kappa_{8M})$ and $\vartheta\in (0,1/4]$. To conclude the proof of the Lemma, we fix the constants $\vartheta$ and $\varepsilon_0$. Given $\beta_0\in (0,1)$, we first choose $\vartheta\in (0,1/4]$ such that $C_e\vartheta^\beta \leq \vartheta^{\beta\beta_0}$ and $\vartheta\leq 2^{-2/(\beta_0\beta)}\leq 2^{-p/(\beta_0\beta)}$. This fixes $\vartheta$ in dependence on $\mathfrak{C},K,\beta$ and $\kappa_{8M}$. As mentioned before this fixes $\varepsilon=\theta^{n+4+2p}$ in dependence on the same parameters and therefore $\delta$, also in dependence on the same parameters. Here we note that we could have used $\kappa_{8M}$ instead of $\kappa_{2M}$ for the upper bound of the bilinear form $\mathcal{A}$. Finally, we have to ensure that the smallness conditions (\ref{9.191}), (\ref{9.1993}) and (\ref{9.1999}) are satisfied. This can be achieved, if we require $E_\lambda(\varrho)\leq \varepsilon_0 |A|^p$ with a sufficiently small constant $\varepsilon_0$ depending on $\mathfrak{C},K,\beta,\kappa_{8M}$ and $\omega_{4M}(\,\cdot\,)$, which concludes the proof.
\end{proof}
In the following we want to iterate the obtained Lemma. Since the only prerequisite for the proof is Lemma \ref{Lem9.1}, we obtain exactly as in \cite[Lemma 8.2 and Lemma 9.2]{bogelein2013regularity} with the parameter $\beta$ replaced by $\beta_0\beta/2$ and $ E_\lambda$ replacing $\Phi_\lambda$ (note $\Phi_\lambda\leq E_\lambda$) the following
\begin{Lem}\label{Lem9.2}
    Let $2n/(n+2)< p\neq 2$, $\beta, \beta_0\in (0,1)$, $M,K\geq 1$ and
    \begin{equation*}
        u\in C\big([0,T];L^2\big(\Omega,\mathbb{R}^{N}\big) \big) \cap L^p\big(0,T;W^{1,p}\big(\Omega,\mathbb{R}^{N}\big) \big)
    \end{equation*}
    be a weak solution of system (\ref{Sys1}) in $\Omega_T$ that satisfies the assumptions of Section \ref{Assumptions}. Then there exits a positive constants
    \begin{align*}
        & \varepsilon_1= \varepsilon_1(\mathfrak{C},K,\beta,\beta_0,\kappa_{8M}, \omega_{4M}(\,\cdot\,)) \leq 1, \\
        &C=C(\mathfrak{C},K,\beta,\beta_0,\kappa_{8M}),
    \end{align*}
    for which the following is true: Let $Q_{\varrho}^{(\lambda)}\Subset\Omega_T$ be some intrinsic geometric cylinder with parameters $\varrho\in (0,1]$ and $0<\lambda\leq 1$ such that there holds
    \begin{align}\label{9.201}
        \big|(Du)_{z_0;\varrho}^{(\lambda)}\big| \leq M,
    \end{align}
    \begin{align}\label{9.202}
        \frac{\lambda}{K} \leq \big|(Du)_{z_0;\varrho}^{(\lambda)}\big| \leq K\lambda
    \end{align}
    and the smallness condition
    \begin{align}\label{9.203}
        E_\lambda (z_0,\varrho)
        \leq
        \varepsilon_0 \big|(Du)_{z_0;\varrho}^{(\lambda)}\big|^p
    \end{align}
    is satisfied. Then the limit
    \begin{align}\label{9.204}
        \Gamma_{z_0} \equiv \lim\limits_{r\downarrow 0} (Du)_{z_0,r}^{(\lambda)}
    \end{align}
    exists and the following decay estimate holds:
    \begin{align}\label{9.205}
        \mint{-}\limits_{Q_r^{(\lambda)}(z_0)} |Du - \Gamma_{z_0} |^p \, dz \leq C \Big( \frac{r}{\varrho} \Big)^{\beta_0 \beta} E_\lambda (z_0,\varrho)
        \qquad \forall 0<r\leq \varrho.
    \end{align}
    Moreover, we have
    \begin{align}\label{9.206}
        \frac{\lambda}{2K}\leq |\Gamma_{z_0}| \leq 2K\lambda.
    \end{align}
\end{Lem}

\subsection{The degenerate regime (DR)}

We now consider the \textbf{degenerate regime}.
\begin{Lem}\label{Lem9.3}
    Let $2n/(n+2)< p\neq 2$, $M_0\geq 1$, $\beta\in (0,1)$, $\chi,\chi_1\in (0,1]$ and
    \begin{equation*}
        u\in C\big([0,T];L^2\big(\Omega,\mathbb{R}^{N}\big) \big) \cap L^p\big(0,T;W^{1,p}\big(\Omega,\mathbb{R}^{N}\big) \big)
    \end{equation*}
    be a weak solution of system (\ref{Sys1}) in $\Omega_T$ that satisfies the assumptions of Section \ref{Assumptions}. Then there exist constants 
    \begin{align*}  
    &\alpha_1=\alpha_1(\mathfrak{C},M_0,\beta,\kappa_2)\in (0,\beta) ,\\
    &\mathfrak{m}=\mathfrak{m}(\mathfrak{C},M_0,\kappa_2), \\
    & C_d = C_d(\mathfrak{C},M_0,\kappa_\mathfrak{m}) \geq 1,\\
    &\vartheta=\vartheta(\mathfrak{C},M_0,\beta,\kappa_\mathfrak{m} ,\chi_1) \in (0,1/16],\\
    &K = K(\mathfrak{C},M_0,\beta,\kappa_\mathfrak{m}, \chi_1 , \eta(\,\cdot\,)) \geq M_0,\\
    &\varepsilon_2 = \varepsilon_2(\mathfrak{C},M_0,\beta,\kappa_\mathfrak{m}
    ,\chi , \chi_1, \eta(\,\cdot\,)) \in (0,1)
    \end{align*}
    such that the following is true: Let $Q_{\varrho}^{(\lambda)}\Subset\Omega_T$ be a intrinsic geometric cylinder with parameters $\varrho>0$ and $\lambda\in (0,1]$, which satisfy
    \begin{equation}\label{9.300}
        \varrho^\beta ( 1 + \lambda^{2-p})^{\frac{\beta}{2}} \leq \varepsilon_2,
    \end{equation}
    such that there holds
    \begin{align}\label{9.301}
        \big|(Du)_{\varrho}^{(\lambda)} \big| \leq M_0 \lambda,
    \end{align}
    \begin{align}\label{9.302}
        \chi \big|(Du)_{\varrho}^{(\lambda)} \big|^p \leq E_\lambda (\varrho)
        \qquad\text{or}\qquad
        \big|(Du)_{\varrho}^{(\lambda)} \big| \leq \frac{\lambda}{K}
    \end{align}
    and the smallness condition
    \begin{align}\label{9.303}
        E_\lambda(\varrho) \leq \min\{ \lambda^p , \varepsilon_2 \}
    \end{align}
    is fulfilled. Then we have for every $\Tilde{\lambda} \in [ \vartheta^{\alpha_1} \lambda,\lambda]$ that $Q_{\vartheta\varrho}^{(\Tilde{\lambda})}\Subset Q_{\varrho}^{(\lambda)}$,
    \begin{align}\label{9.304}
        E_{\Tilde{\lambda}}(\vartheta \varrho) \leq \chi_1 \Tilde{\lambda}^p
        \qquad\text{and}\qquad
        \big|(Du)_{\vartheta\varrho}^{(\Tilde{\lambda})} \big| \leq C_d \lambda.
    \end{align}
    Here we again abbreviated $Q_{\varrho}^{(\lambda)} \equiv Q_{\varrho}^{(\lambda)}(z_0)$,
    \begin{align*}
        (Du)_{\varrho}^{(\lambda)} \equiv (Du)_{z_0;\varrho}^{(\lambda)}
        ,\qquad
        \Phi_\lambda(\varrho) \equiv \Phi_\lambda \big(z_0,\varrho, (Du)_{z_0;\varrho}^{(\lambda)} \big).
        \quad\text{and}\quad
        E_\lambda(\varrho) \equiv E_\lambda (z_0,\varrho)
    \end{align*}
\end{Lem}
\begin{proof}
    As always we will assume without loss of generality that $z_0=0$ and use the abbreviations stated in the Lemma together with $\Tilde{\Phi}_\lambda(\varrho)= \Tilde{\Phi}_\lambda(0,\varrho)$ and $H_\lambda^*(\varrho) = H_\lambda^*(0,\varrho,(u)_{\varrho}^{(\lambda)})$. In order to combine both the sub- and super-quadratic case, we again define 
    \begin{equation*}
        \varrho_* = \begin{cases}
            \varrho/4 &\text{for } p<2, \\
            \varrho & \text{for } p>2
        \end{cases}
    \end{equation*}
    and the auxiliary function
    \begin{align}\label{9.34}
        w(x,t) = 
        \begin{cases}
            \frac{u(x,\lambda^{2-p}t) - (u)^{(\lambda)}_{\varrho/2}}{C_2 \lambda} 
         &\text{for } p<2, \\
            \frac{u(x,\lambda^{2-p}t) - (u)^{(\lambda)}_{\varrho}}{C_2 \lambda} 
         & \text{for } p>2
        \end{cases}
        \qquad \text{for } (x,t)\in Q_\varrho \equiv Q_\varrho^{(1)},
    \end{align}
    where $C_2$ denotes a constant that will be fixed during the proof. We start with the observation that due to assumptions (\ref{9.301}) and (\ref{9.303}), the following is true
    \begin{align}\label{9.31}
        \Tilde{\Phi}_\lambda(\varrho) 
        &\leq 
        2^{p-1} \Bigg[ \mint{-}\limits_{Q^{(\lambda)}_{\varrho}} \big| Du- (Du)^{(\lambda)}_{\varrho} \big|^p \, dz + \big|(Du)^{(\lambda)}_{\varrho} \big|^p \Bigg]
        \nonumber \\ & \leq
        C(p) \Big[ E_\lambda(\varrho) + \mathds{1}_{\{p<2\}} \big|(Du)^{(\lambda)}_{\varrho} \big|^{\frac{p(2-p)}{2}}  E^{\frac{p}{2}}_\lambda(\varrho) + \big|(Du)^{(\lambda)}_{\varrho} \big|^p \Big]
        \nonumber \\ & \leq 
        C(p)\big(1 + \mathds{1}_{\{p<2\}} M_0^{\frac{p(2-p)}{2}} + M_0^p \big) \lambda^p 
        =: C_1(p,M_0) \lambda^p
    \end{align}
    with the obvious definition of $C_1\geq 1$. We now estimate $\Tilde{\Phi}_\lambda(\varrho)$ a second time, distinguishing between two cases. In the first case we assume that (\ref{9.302})\textsubscript{1} holds. In analogy to (\ref{9.31}) and using (\ref{9.301}), (\ref{9.303}) as well as $\lambda\leq 1$ this results in
    \begin{align}\label{9.32}
        \Tilde{\Phi}^{\frac{1}{p}}_\lambda(\varrho)
        \leq
        C(p)
        \Big[ E^{\frac{1}{p}}_\lambda(\varrho) + \mathds{1}_{\{p<2\}} \big|(Du)^{(\lambda)}_{\varrho} \big|^{\frac{2-p}{2}}  E^{\frac{1}{2}}_\lambda(\varrho) + \big|(Du)^{(\lambda)}_{\varrho} \big| \Big]
        \leq
        C(p) \big(1+\chi^{-\frac{1}{p}}\big) \big( E_\lambda(\varrho) \big)^{\frac{1}{p}},
    \end{align}
    while in the second case (\ref{9.302})\textsubscript{2}, we have
    \begin{align}\label{9.33}
        \Tilde{\Phi}^{\frac{1}{p}}_\lambda(\varrho)
        \leq
        C(p) \big( E_\lambda(\varrho) \big)^{\frac{1}{p}} + \frac{\lambda}{K}.
    \end{align}
    Similar to Lemma \ref{Lem9.1}, we apply the $p$-caloric approximation Lemma \ref{Lem4.1} or \ref{Lem4.2} to the re-scaled function $w$. To this end, we must assure the assumptions of aforementioned Lemma. Our initial step involves verifying whether $w$ satisfies the Poincaré inequality (\ref{4.103}). To this end, let $Q_{r}^{(\mu)}(z')\subset Q_\varrho$ be an arbitrary intrinsic geometric cylinder with center $z'=(x',t')\in \R^{n+1}$ and parameters $r,\mu>0$. Applying the Poincaré inequalities of Lemma \ref{Lem5.3} with the choice $A=0$ to $u$ on $Q_{r}^{(\lambda\mu)}(\Tilde{z})\subset Q^{(\lambda)}_\varrho$, where the center of the cylinder is given by $\Tilde{z}=(x',\lambda^{2-p}t')$, allows us to infer for the parameters $q\in [1,p]$ that
    \begin{align}\label{9.35}
        \mint{-}\limits_{Q_{r}^{(\lambda\mu)}(\Tilde{z})} \big| u & - (u)_{\Tilde{z};r}^{(\lambda\mu)} \big|^q \, dz
        \leq
        C(\mathfrak{C}) r^q \Bigg[ \mint{-}\limits_{Q_{r}^{(\lambda\mu)}(\Tilde{z})} |Du|^q \, dz 
        +
        \bigg((\mu\lambda)^{2-p} \mint{-}\limits_{Q_{r}^{(\lambda\mu)}(\Tilde{z})} |Du|^{p-1}  \, dz \bigg)^q \Bigg].
    \end{align}
    This inequality will be used to deduce that both (\ref{4.101}) and (\ref{4.103}) are satisfied. By scaling (\ref{9.35}) to $w$ on $Q_{r}^{(\mu)}(z')$, we obtain that
    \begin{align*}%
        \mint{-}\limits_{Q_{r}^{(\mu)}(z')} \big| w & - (w)_{z';r}^{(\mu)} \big|^q \, dz
        \leq
        C_p r^q \Bigg[ \mint{-}\limits_{Q_{r}^{(\mu)}(z')} |Dw|^q \, dz 
        +
        \bigg(\mu^{2-p} \mint{-}\limits_{Q_{r}^{(\mu)}(z')} |Dw|^{p-1}  \, dz \bigg)^q \Bigg],
    \end{align*}
    ensuring that the Poincaré inequality (\ref{4.103}) holds with the constant $C_p=C_p(\mathfrak{C},C_2)$. We continue by showing that the $L^p$ estimate (\ref{4.101}) is satisfied. To this end, we estimate $H_\lambda^*$. To start, we obtain with the help of the inequalities (\ref{9.01}), (\ref{9.31}) and $C_1\geq 1$ that
    \begin{align}\label{9.36}
        \mint{-}\limits_{Q_\varrho^{(\lambda)}} \bigg| \frac{u - (u)_{\varrho}^{(\lambda)}}{\varrho}\bigg|^p \, dz
         &\leq
        C(\mathfrak{C}) \big[ \Tilde{\Phi}_\lambda(\varrho) 
        +
        \lambda^{p(2-p)} \Tilde{\Phi}_\lambda(\varrho)^{p-1} \big]
        \leq
        C(\mathfrak{C}) C_1^{p} \lambda^p.
    \end{align}
    Now we distinguish the cases. In the \textbf{super-quadratic case}, we immediately obtain with the definition (\ref{9.34}) of $w$ and (\ref{9.31}) that inequality (\ref{9.36}) implies
    \begin{align*}
        \mint{-}\limits_{Q_{\varrho_*}} \bigg| \frac{w}{\varrho} \bigg|^p \, dz
        +
        \mint{-}\limits_{Q_{\varrho_*}} |Dw|^p \, dz
        &=
        \frac{1}{C_2^p \lambda^p} \Bigg[ \mint{-}\limits_{Q^{(\lambda)}_\varrho} \bigg| \frac{u - (u)^{(\lambda)}_{\varrho}}{\varrho} \bigg|^p \, dz
        +
        \mint{-}\limits_{Q^{(\lambda)}_\varrho} |Du|^p \, dz \Bigg]
        \leq
        \frac{C(\mathfrak{C}) C_1^{p} + C_1}{C_2^p}
        \leq 1,
    \end{align*}
    provided we have chosen $C_2=C_2(\mathfrak{C},C_1)\geq 1$ large enough. This determines $C_2$ as a constant depending on $\mathfrak{C}$ and $M_0$. In the \textbf{sub-quadratic case}, we use the Sobolev-Poincaré inequality of Lemma \ref{Lem5.4} with arbitrarily chosen $A\in \R^{Nn}\setminus \{0\}$ that satisfies $|A|=\lambda \leq 1$ and the estimate
    \begin{equation}\label{9.361}
        H_\lambda^*(\varrho)
        \leq
        C(\mathfrak{C}) \big[ C_1^\beta \varrho^\beta + \varrho^\beta (1+\lambda^{2-p})^{\frac{\beta}{2}} \big] \leq C(\mathfrak{C}) C_1,
    \end{equation}
    which is easily derived from the definition (\ref{5.01}), assumption (\ref{9.300}), (\ref{9.36}) and $C_1\geq 1$. Lemma \ref{Lem5.4} then allows us to estimate
    \begin{align}\label{9.37}
         \mint{-}\limits_{Q_{\varrho /2}^{(\lambda)}} & \bigg| \frac{u - (u)_{\varrho/2}^{(\lambda)}}{\varrho} \bigg|^2 \, dz
        \leq
        2 \mint{-}\limits_{Q_{\varrho/2}^{(\lambda)}} \bigg| \frac{u - (u)_{\varrho/2}^{(\lambda)}- A x}{\varrho} \bigg|^2 \, dz
        +
        2|A|^2
        \nonumber \\ & \leq 
        C
        \Bigg[ \lambda^2 H_\lambda^*(\varrho) + \Bigg( \mint{-}\limits_{Q^{(\lambda)}_{\varrho} } |Du-A|^p \, dz 
        + 
         \lambda^{p} \big( H_\lambda^*(\varrho)\big)^p \Bigg)^{\frac{2}{p}}\Bigg] + 2 \lambda^2
        \nonumber \\ & \leq 
        C \bigg[ C_1  \lambda^2 + \Big( \Tilde{\Phi}_\lambda (\varrho) + |A|^p
        + 
         \lambda^{p} C_1^{p} \Big)^{\frac{2}{p}} \bigg] + 2 \lambda^2
        \nonumber \\ & \leq 
        C \bigg[ C_1 \lambda^2  + \big( (1+C_1+C_1^{p})\lambda^p
         \big)^{\frac{2}{p}} \bigg] + 2 \lambda^2
        \leq C C_1^2 \lambda^2,
    \end{align}
    where we have for the dependency of the constant $C=C(\mathfrak{C},\kappa_2)$. Since $p<2$, this additionally provides an $L^p$-estimate for $u - (u)_{\varrho/2}^{(\lambda)}$ by Hölder's inequality. Recalling the definition (\ref{9.34}) of $w$ and (\ref{9.31}), the preceding estimates implies
    \begin{align*}
        \mint{-}\limits_{Q_{\varrho_*}} \bigg| \frac{w}{\varrho} \bigg|^p \, dz
        +
        \mint{-}\limits_{Q_{\varrho_*}} |Dw|^p \, dz
        &=
        \frac{1}{C_2^p \lambda^p} \Bigg[ \mint{-}\limits_{Q^{(\lambda)}_{\varrho/4}} \bigg| \frac{u - (u)^{(\lambda)}_{\varrho/2}}{\varrho} \bigg|^p \, dz
        +
        \mint{-}\limits_{Q^{(\lambda)}_{\varrho/4}} |Du|^p \, dz \Bigg]
        \nonumber\\ & \leq
        \frac{2^{n+2} C^{\frac{p}{2}} C_1^p + 4^{n+2} C_1}{C_2^p}
        \leq 1,
    \end{align*}
    provided we have chosen $C_2=C_2(\mathfrak{C}, \kappa_2,C_1)\geq 1$ large enough. This determines $C_2$ as a constant depending on $\mathfrak{C}, \kappa_2$ and $M_0$. Furthermore, we ensure that sup-bound (\ref{4.201}) is satisfied. Thereby, we employ the Caccioppoli inequality from Lemma \ref{Lem5.1} with $Dl=0$ (we use for the Lemma's parameters $M=1$ and $\kappa_1=1$), (\ref{9.37}) and Hölder's inequality to find
    \begin{align}\label{9.381}
        \sup\limits_{t\in \Lambda_{\varrho/4}^{(\lambda)}} \mint{-}\limits_{B_{\varrho/4}} \bigg| \frac{u - (u)_{\varrho/2}^{(\lambda)}}{\varrho} \bigg|^2 \, dx
        &\leq
        C \mint{-}\limits_{Q_{\varrho/4}^{(\lambda)}} \lambda^{2-p} \bigg| \frac{u - (u)_{\varrho/2}^{(\lambda)}}{\varrho} \bigg|^p 
        +
        \bigg| \frac{u - (u)_{\varrho/2}^{(\lambda)}}{\varrho} \bigg|^2 \, dz
        \leq
        C \lambda^2,
    \end{align}
    where we have for the constant $C=C(\mathfrak{C},C_1)$. Let us now rewrite this in terms of $w$. Taking into account its definition (\ref{9.34}), we obtain
    \begin{align}\label{9.382}
        \sup\limits_{t\in \Lambda_{\varrho_*}^{(\lambda)}} \mint{-}\limits_{B_{\varrho_*}} \bigg| \frac{w}{\varrho} \bigg|^2 \, dx
        =
        \frac{1}{C_2^2 \lambda^2}
        \sup\limits_{t\in \Lambda_{\varrho/4}^{(\lambda)}} \mint{-}\limits_{B_{\varrho/4}} \bigg| \frac{v}{\varrho} \bigg|^2 \, dx
         \leq
        \frac{C}{C_2^2} \leq 1,
    \end{align}
    where the last inequality holds after an eventual enlargement of $C_2$, which does not alter its dependencies. We now continue by treating the sub- and super-quadratic case simultaneously. By Lemma \ref{Lem6.2} we know that $u$ is approximately $p$-caloric, in the sense that for any $\delta>0$ and $\phi\in C^1_0\big(Q^{(\lambda)}_{\varrho_*},\R^N\big)$ there holds:
    \begin{align}\label{9.39}
        \Bigg| \mint{-}\limits_{Q_{\varrho_*}^{(\lambda)}} & u \cdot \partial_t \phi - \Tilde{a}(0,(u)_{\varrho_*}^{(\lambda)}) |Du|^{p-2} Du \cdot D\phi \, dz \Bigg| 
        \leq
        C(p,L) \Bigg[ \delta + \frac{\big(\Tilde{\Phi}_\lambda(\varrho_*)\big)^{\frac{1}{p}}}{\eta(\delta)}
        \nonumber \\ & \qquad +
        \varrho^\beta \bigg( (1+\lambda^{2-p})^{\frac{\beta}{2}} + 
        \big(  \Tilde{\Phi}_\lambda(\varrho_*) + \lambda^{p(2-p)} \Tilde{\Phi}_\lambda^{p-1}(\varrho_*)\big)^{\frac{\beta}{p}} \bigg) \Bigg]
        \big(\Tilde{\Phi}_\lambda(\varrho_*)\big)^{1-\frac{1}{p}} 
        \sup\limits_{Q_{\varrho_*}^{(\lambda)}} |D\phi|
        \nonumber \\ & \leq
        C(\mathfrak{C}) C_1^{\frac{p-1}{p}} \lambda^{p-1} \Bigg[ \delta + \frac{\big(\Tilde{\Phi}_\lambda(\varrho)\big)^{\frac{1}{p}}}{\eta(\delta)}
        +
        \varrho^\beta \big( (1+\lambda^{2-p})^{\frac{\beta}{2}} + C_1 \big) \Bigg] \sup\limits_{Q_{\varrho_*}^{(\lambda)}} |D\phi|,
    \end{align}
    where we used (\ref{9.31}), $\Tilde{\Phi}_\lambda(\varrho_*)\leq 4^{n+2}\Tilde{\Phi}_\lambda(\varrho)$, $C_1\geq 1$ and $\lambda\leq 1$ for the last inequality. If we keep 
    \begin{align*}
        \mint{-}\limits_{Q_{\varrho_*}^{(\lambda)}} (u-C_2\lambda w)\cdot \partial_t \phi \, dz = 0 \qquad\forall \phi\in C^1_0(Q_{\varrho_*}^{(\lambda)},\R^N)
    \end{align*}
    in mind, inequality (\ref{9.39}) can be rewritten as follows:
    \begin{align*}
        I:= \Bigg| \mint{-}\limits_{Q_{\varrho_*}} & w \cdot \partial_t \phi - C_2^{p-2} \Tilde{a}(0,(u)_{\varrho_*}^{(\lambda)}) |Dw|^{p-2} Dw \cdot D\phi \, dz \Bigg| 
        \nonumber \\ & \leq
        \frac{C(\mathfrak{C}) C_1^{\frac{p-1}{p}}}{C_2} \Bigg[ \delta + \frac{\big(\Tilde{\Phi}_\lambda(\varrho)\big)^{\frac{1}{p}}}{\eta(\delta)}
        +
        C_1 \varrho^\beta (1+\lambda^{2-p})^{\frac{\beta}{2}} \Bigg] \sup\limits_{Q_{\varrho_*}} |D\phi|
    \end{align*}
    for all $\phi\in C^1_0(Q_{\varrho_*},\R^N)$. By enlargement, we can assume $C_2\geq C(p,L) C_1^{\frac{p-1}{p}} \geq 1$ without changing the dependencies on $\mathfrak{C}, \kappa_2$ and $M_0$. We further estimate the right-hand side, distinguishing two cases. If (\ref{9.32}) holds we have
    \begin{align}\label{9.392}
        I\leq \Bigg[ \delta + C(p) \frac{1+\chi^{-\frac{1}{p}}}{\eta(\delta)} \big( E_\lambda(\varrho) \big)^{\frac{1}{p}}
        +
        C_1 \varrho^\beta (1+\lambda^{2-p})^{\frac{\beta}{2}} \Bigg] \sup\limits_{Q_{\varrho_*}} |D\phi|,
    \end{align}
    whereas using (\ref{9.33}) and $\lambda\leq 1$ we find
    \begin{align*}
        I\leq \Bigg[ \delta + \frac{\big( C(p) E_\lambda(\varrho) \big)^{\frac{1}{p}} + \frac{1}{K}}{\eta(\delta)} 
        +
        C_1 \varrho^\beta (1+\lambda^{2-p})^{\frac{\beta}{2}} \Bigg] \sup\limits_{Q_{\varrho_*}} |D\phi|.
    \end{align*}
    With the aim of applying the $p$-caloric approximation Lemmas \ref{Lem4.1} or \ref{Lem4.2}, we define the vector field
    \begin{align*}
        \mathcal{A}(\xi) = C_2^{p-2} \Tilde{a}(0,(u)_{\varrho_*}^{(\lambda)}) |\xi|^{p-2} \xi
        \qquad\text{for } \xi\in \R^{Nn},
    \end{align*}
    for which one easily calculates that the hypothesis of the Lemma is satisfied when the data $(\mu,\nu,L,L/\nu)$ is replaced by $(0,C(p)^{-1}C_2^{p-2}\nu, C_2^{p-2}L, C(p)L/\nu)$ for an appropriate constant $C(p)$. Overall, we have demonstrated that, depending on the sub- and super-quadratic case, $w$ satisfies all but one of the assumptions of Lemma \ref{Lem4.1} or \ref{Lem4.2} respectively, on $Q_{\varrho_*}$. It remains to establish (\ref{4.102}). To this end, let $\varepsilon>0$ to be chosen later and $\delta_0=\delta_0(\mathfrak{C},C_2,C_p,\varepsilon)\in (0,1]$ be the constant of Lemma \ref{Lem4.1} or \ref{Lem4.2}. Note that due to the dependencies of $C_2$ and $C_p$, we have $\delta_0=\delta_0(\mathfrak{C},M_0,\varepsilon)$. Let us now fix
    \begin{align}\label{9.395}
        \delta\leq \frac{\delta_0}{3}
    \end{align}
    and note that $\delta$ depends on $\mathfrak{C},M_0$ and $\varepsilon$. Choosing the latter will later determine $\eta(\delta)$. Furthermore, we assume that either
    \begin{align}\label{9.396}
        C(p)\frac{1+\chi^{-1/p}}{\eta(\delta)} \big( E_\lambda(\varrho)\big)^{\frac{1}{p}} \leq \frac{\delta_0}{3}
    \end{align}
    or 
    \begin{align}\label{9.397}
        \frac{\big( C(p) E_\lambda(\varrho) \big)^{\frac{1}{p}}} {\eta(\delta)} \leq \frac{\delta_0}{6}
        \qquad\text{and}\qquad
        \frac{1}{K\eta(\delta)} \leq \frac{\delta_0}{6}
    \end{align}
    is fulfilled and note that (\ref{9.396}) corresponds to (\ref{9.302})\textsubscript{1}, whereas (\ref{9.397}) corresponds to (\ref{9.302})\textsubscript{2}. Finally, we state a \textbf{smallness assumption for the radius $\varrho$}. Namely that
    \begin{align}\label{9.398}
        \varrho^\beta (1+\lambda^{2-p})^{\frac{\beta}{2}} \leq  \frac{\delta_0}{3C_1} .
    \end{align}
    Combining (\ref{9.392})-(\ref{9.398}), we notice that the $p$-caloric approximation Lemmas \ref{Lem4.1} or \ref{Lem4.2} are applicable. This implies the existence of a $p$-caloric function
    \begin{align*}
        h\in  C^0\big( \Lambda_{\varrho_*/2} ; L^2(B_{\varrho_*/2},\R^N)\big) \cap L^p \big( \Lambda_{\varrho_*/2} ; W^{1,p}(B_{\varrho_*/2},\R^N)\big)
    \end{align*}
    such that
    \begin{align}\label{9.3992}
        \int\limits_{Q_{\varrho_*/2} } \bigg| \frac{w-h}{\varrho_*/2} \bigg|^p 
        +
        \bigg| \frac{w-h}{\varrho_*/2} \bigg|^2 \, dz
        \leq \varepsilon.
    \end{align}
    Now apply Lemma \ref{Lem7.1v2} to the $(\mathcal{A},p)$-caloric function $h$ on $Q_{\varrho/8}$ with the particular choices $\nu C_2^{p-2} \leq \Lambda = C_2^{p-2} \Tilde{a}(0,(u)_{\varrho}^{(\lambda)}) \leq C_2^{p-2} L$ and $C_*=H(n,p)$. We once more denote with $\beta_1=\beta_1(\mathfrak{C}, \kappa_2,M_0,\beta)$ the corresponding Hölder exponent from Lemma \ref{Lem7.1v2} and assume that $p\beta_1 < \beta$ holds without loss of generality. For $\theta\in (0,1/16]$ to be chosen later, we apply the Lemma with $r=2\theta\varrho_*$ and obtain from (\ref{7.105}) that
    \begin{align}\label{9.3993}
        \sup\limits_{Q_{2\theta\varrho_*}} |Dh| \leq C.
    \end{align}
    Moreover, for any 
    \begin{equation*}
        s\in\begin{cases}
            [1,2] &\text{for }p<2,\\
            [2,p] &\text{for }p>2,
        \end{cases}
    \end{equation*}
    we have
    \begin{align}\label{9.3994}
        \frac{1}{(2\theta\varrho_*)^s} \mint{-}\limits_{Q_{2\theta\varrho_*}} \big| h - (h)_{2\theta\varrho_*} - (Dh)_{2\theta\varrho_*} x \big|^s \, dz 
        \leq
        C \theta^{\min\{s,p,s(p-1)\}\beta_1},
    \end{align}
    where in both inequalities $C=C(\mathfrak{C}, \kappa_2,M_0)$. We will use the previous decay estimate with parameters $s=2,p$ and note that we have for 
    \begin{equation*}
        \nu = \begin{cases}
            p(p-1) &\text{for }p<2,\\
            2 &\text{for }p>2
        \end{cases}
    \end{equation*}
    that $\nu \leq \min\{s,p,s(p-1)\}$. Then we obtain from (\ref{9.3992}) and (\ref{9.3994}) the estimate
    \begin{align*}
        \frac{1}{(2\theta\varrho_*)^s}  \mint{-}\limits_{Q_{2\theta\varrho_*}} & \big| w - (h)_{2\theta\varrho_*} - (Dh)_{2\theta\varrho_*} x \big|^s \, dz 
        \nonumber\\ & \leq
        \frac{2^{s-1}}{(2\theta\varrho_*)^s}
        \mint{-}\limits_{Q_{2\theta\varrho_*}} | w - h |^s + \big| h - (h)_{2\theta\varrho_*} - (Dh)_{2\theta\varrho_*} x \big|^s \, dz 
        \nonumber\\ & \leq
        2^{s-1} \Bigg[ \theta^{-s} \frac{|Q_{\varrho_*/2}|}{|Q_{2\theta\varrho_*}|} \mint{-}\limits_{Q_{\varrho_*/2}} \bigg| \frac{w - h}{\varrho_*/2} \bigg|^s \, dz
        +
        C \theta^{\nu \beta_1} \Bigg]
        \leq
        2^{s-1} \big[ \theta^{-n-2-s} \varepsilon
        +
        C \theta^{\nu\beta_1} \big]
        \nonumber\\ & \leq
        2^{s-1} \big[ \theta^{2}
        +
        C \theta^{\nu\beta_1} \big]
        \leq
        C \theta^{\nu\beta_1},
    \end{align*}
    where $C=C(\mathfrak{C}, \kappa_2,M_0)\geq 1$ and we chose $\varepsilon=\theta^{n+6+p}$. Scaling back to $u$ on $Q_{2\theta\varrho}^{(\lambda)}$ we arrive at the following estimate
    \begin{align}\label{9.3997}
        \frac{1}{(2\theta\varrho_*)^s}  \mint{-}\limits_{Q_{2\theta\varrho_*}^{(\lambda)}} & | u - l |^s \, dz 
        \leq
        C \lambda^s \theta^{\nu\beta_1},
    \end{align}
    for a constant $C$ depending only on $\mathfrak{C}, \kappa_2$ and $M_0$. Here we abbreviated
    \begin{align*}
        \R^n\ni x \mapsto l(x):= u - C_2 \lambda ( w + (h)_{2\theta\varrho_*} + (Dh)_{2\theta\varrho_*} x ).
    \end{align*}
    Due to (\ref{9.3993}) we have for the modulus of the gradient 
    \begin{align}\label{9.3998}
        |Dl| = C_2 \lambda \big| (Dh)_{2\theta\varrho_*} \big| \leq
        C_2 \lambda \sup\limits_{Q_{2\theta\varrho_*}^{(\mu)}} |Dh| \leq C\lambda =: C_3 \lambda ,
    \end{align}
    with the obvious definition of $C_3 = C_3(\mathfrak{C}, \kappa_2,M_0)$. If we use (\ref{9.300}) and (\ref{9.3997}) for $s=p$ and $s=2$ as well as $\lambda\leq 1$, we can estimate the right-hand side of the Caccioppoli inequality of Lemma \ref{Lem5.1} as follows:
    \begin{align*}
        \mint{-}\limits_{Q_{2\theta\varrho_*}^{(\lambda)}} & \bigg| V_{|Dl} \bigg( \frac{u-l}{2\theta\varrho_*} \bigg) \bigg|^2
        +
        \lambda^{p-2}
        \bigg| \frac{u-l}{2\theta\varrho_*} \bigg|^2\, dz 
        +
        L |Dl|^{p} \Bigg[\Bigg( \mint{-}\limits_{Q_{2\theta\varrho_*}^{(\lambda)}} |u-l(0)|^p  \, dz \Bigg)^{\frac{\beta}{p}}
        + (2\theta\varrho_*)^\beta (1+\lambda^{2-p})^{\frac{\beta}{2}} \Bigg] 
        \\ & \leq
        C \mint{-}\limits_{Q_{2\theta\varrho_*}^{(\lambda)}} \bigg| \frac{u-l}{2\theta\varrho_*} \bigg|^p 
        +
        \big( \mathds{1}_{\{p>2\}}  |Dl|^{p-2} + \lambda^{p-2} \big)
        \bigg| \frac{u-l}{2\theta\varrho_*} \bigg|^2\, dz 
        \\ & \qquad + 
        C |Dl|^{p} (2\theta\varrho_*)^\beta \Bigg( \mint{-}\limits_{Q_{2\theta\varrho_*}^{(\lambda)}} \bigg| \frac{u-l}{2\theta\varrho_*} \bigg|^p \, dz + |Dl|^p \Bigg)^{\frac{\beta}{p}}
        + 
        C |Dl|^{p} \theta^\beta
        \leq
        C(\mathfrak{C}, \kappa_2,M_0) \lambda^p \theta^{\nu\beta_1}.
    \end{align*}
    And the Caccioppoli inequality then yields
    \begin{align}\label{9.39991}
        \mint{-}\limits_{Q_{\theta\varrho_*}^{(\lambda)}} & \big| V_{|Dl|} ( Du - Dl) \big|^2\, dz 
        \leq
        C(\mathfrak{C},M_0,\kappa_{4C_3}) \lambda^p \theta^{\nu\beta_1},
    \end{align}
    where the dependence on $\kappa_{4C_3}$ stems from the fact that the constant of the Caccioppoli inequality depends on $|Dl|$ and we used the monotonicity of $\kappa$ with $4C_3\geq 2$. To shorten the notation, we define $\mathfrak{m}=4C_3$ and simply write $\kappa_{\mathfrak{m}}$ instead of $\kappa_{4C_3}$. By recalling the dependencies of $C_3$, we realize $\mathfrak{m}=\mathfrak{m}(\mathfrak{C}, \kappa_2,M_0)$. We now adjust the scaling parameter of the cylinders. To get started, we define
    \begin{align*}
        \vartheta = \begin{cases}
            \theta/4 &\text{for }p<2,\\
            \theta^{1+\frac{\beta_1}{2(n+2)}} &\text{for }p>2,
        \end{cases}
        \qquad\text{and}\qquad
        \alpha_1= \begin{cases}
            \nu\beta_1/4 &\text{for }p<2,\\
            \frac{\beta_1}{2(n+2)p} &\text{for }p>2,
        \end{cases}.
    \end{align*}
    Then for every $\Tilde{\lambda} \in [\vartheta^{\alpha_1}\lambda,\lambda]$ there holds
    \begin{align*}
        Q_{\vartheta\varrho}^{(\Tilde{\lambda})}\subset Q_{\theta\varrho_*}^{(\lambda)} 
        \qquad\text{and}\qquad
        \frac{|Q_{\theta\varrho_*}^{(\lambda)}|}{|Q_{\vartheta\varrho}^{(\Tilde{\lambda})}|} \leq  4^p \theta^{-(\mathds{1}_{\{p<2\}} (2-p)\alpha_1 + \mathds{1}_{\{p>2\}}\beta_1/2)} .
    \end{align*}
    This allows us to adapt the scaling parameter of the cylinders in (\ref{9.39991}) as follows:
    \begin{align}\label{9.39992}
        \mint{-}\limits_{Q_{\vartheta\varrho}^{(\Tilde{\lambda})}} & \big| V_{|Dl|} ( Du - Dl) \big|^2\, dz 
        \leq
        4^p \theta^{-(\mathds{1}_{\{p<2\}} (2-p)\alpha_1 + \mathds{1}_{\{p>2\}}\beta_1/2)}
        \mint{-}\limits_{Q_{\theta\varrho_*}^{(\lambda)}} \big| V_{|Dl|} ( Du - Dl) \big|^2\, dz 
        \nonumber\\ &\leq
        C \lambda^p \theta^{\nu \beta_1 -(\mathds{1}_{\{p<2\}} (2-p)\alpha_1 + \mathds{1}_{\{p>2\}}\beta_1/2)}
        \leq
        C \Tilde{\lambda}^p \theta^{\nu \beta_1 -(\mathds{1}_{\{p<2\}} (2-p)\alpha_1 + \mathds{1}_{\{p>2\}}\beta_1/2+p\alpha_1)}
        \leq
        C\Tilde{\lambda}^p \vartheta^{\frac{\nu\beta_1}{2}}.
    \end{align}
    To obtain an estimate for $\Phi_{\Tilde{\lambda}}(\vartheta\varrho)$, we have to replace $|Dl|$ by $\big| (Du)_{\vartheta\varrho}^{(\Tilde{\lambda})} \big|$. But this can be achieved by means of Lemma \ref{Lem2.4}. We obtain
    \begin{align}\label{9.39993}
        \Phi_{\Tilde{\lambda}}(\vartheta\varrho)
        &=
        \mint{-}\limits_{Q_{\vartheta\varrho}^{(\Tilde{\lambda})}} \big| V_{|(Du)_{\vartheta\varrho}^{(\Tilde{\lambda})}|} \big( Du - (Du)_{\vartheta\varrho}^{(\Tilde{\lambda})} \big) \big|^2\, dz 
        \leq
        C
        \mint{-}\limits_{Q_{\vartheta\varrho}^{(\Tilde{\lambda})}} \big| V_{|Dl|} ( Du - Dl) \big|^2\, dz 
        \nonumber\\ & \leq
        C \vartheta^{\frac{\nu \beta_1}{2}} \Tilde{\lambda}^p
        =: C_5(\mathfrak{C},M_0,\kappa_{\mathfrak{m}}) \vartheta^{\frac{\nu \beta_1}{2}} \Tilde{\lambda}^p,
    \end{align}
    with the obvious definition of $C_5$. We proceed by estimating $\big| (Du)_{\theta\varrho}^{(\Tilde{\lambda})} \big|$ with the help of inequalities (\ref{9.3998}), (\ref{9.39992}) and (\ref{9.02}), resulting in 
    \begin{align}\label{9.39994}
        \big| (Du)_{\vartheta\varrho}^{(\Tilde{\lambda})} \big| 
        &\leq 
        |Dl| + \bigg(\mint{-}\limits_{Q_{\vartheta\varrho}^{(\Tilde{\lambda})}} \big| Du - Dl \big|^p\, dz \bigg)^{\frac{1}{p}}
        \leq
        C_3 \lambda +  C \Bigg[ \mint{-}\limits_{Q_{\vartheta\varrho}^{(\Tilde{\lambda})}} \big| V_{|Dl|} ( Du - Dl) \big|^2\, dz 
        \nonumber\\ & \qquad + 
        \mathds{1}_{\{p<2\}}|Dl|^{\frac{p(2-p)}{2}} \bigg( \mint{-}\limits_{Q_{\vartheta\varrho}^{(\Tilde{\lambda})}} \big| V_{|Dl|} ( Du - Dl) \big|^2\, dz \bigg)^{\frac{p}{2}} \Bigg]
        \leq
        C_3 \lambda + C \lambda =: C_d \lambda
    \end{align}
    with the obvious definition of $C_d=C_d(\mathfrak{C},M_0,\kappa_{\mathfrak{m}})$. Let us now estimate the additional term (\ref{8.04}) of $E_{\Tilde{\lambda}}(\vartheta\varrho)$ with the help of inequality (\ref{9.03}). We obtain with assumption (\ref{9.300}), the estimates (\ref{9.39993}) and (\ref{9.39994}) as well as $\beta_1<\beta$ and $\Tilde{\lambda}\leq \lambda\leq 1$ the chain of inequalities
    \begin{align}\label{9.39995}
        H_{\Tilde{\lambda}}(\vartheta\varrho) & \leq C \big|(Du)_{\vartheta\varrho}^{(\Tilde{\lambda})} \big|^{p} (\vartheta\varrho)^\beta 
        \Bigg( \bigg[
        \Phi_{\Tilde{\lambda}}(\vartheta\varrho) 
        + 
        \mathds{1}_{\{p<2\}} \big| (Du)_{\vartheta\varrho}^{(\Tilde{\lambda})} \big|^{\frac{p(2-p)}{2}} \Phi_{\Tilde{\lambda}}(\vartheta\varrho)^\frac{p}{2} 
         \nonumber\\ & \qquad +
        \big|(Du)_{\vartheta\varrho}^{(\Tilde{\lambda})} \big|^p + \Tilde{\lambda}^{p(2-p)} \big|(Du)_{\vartheta\varrho}^{(\Tilde{\lambda})} \big|^{p(p-1)} 
        +
        \Tilde{\lambda}^{p(2-p)} \Big( \Phi_{\Tilde{\lambda}}\big(\vartheta\varrho, (Du)_{\vartheta\varrho}^{(\Tilde{\lambda})} \big) 
        \nonumber \\ & \qquad +  \mathds{1}_{\{p<2\}}
        \big| (Du)_{\vartheta\varrho}^{(\Tilde{\lambda})} \big|^{\frac{p(2-p)}{2}} \Phi_{\Tilde{\lambda}}\big(z_0,\varrho, (Du)_{\vartheta\varrho}^{(\Tilde{\lambda})} \big)^\frac{p}{2} \Big)^{p-1} \bigg]^{\frac{\beta}{p}} 
        + 
        ( 1+\Tilde{\lambda}^{2-p})^{\frac{\beta}{2}} \Bigg)
        \nonumber\\ & \leq
        C \vartheta^\beta \vartheta^{\mathds{1}_{\{p>2\}}\alpha_1 \beta (2-p)} \lambda^p
        \leq
        C \vartheta^{\beta - (p + \mathds{1}_{\{p>2\}}p)\alpha_1} \Tilde{\lambda}^p
        \nonumber\\ & \leq
        C \vartheta^{\beta - \beta_1} \Tilde{\lambda}^p
        =: 
        C_6(\mathfrak{C},M_0,\kappa_{\mathfrak{m}}) \vartheta^{\beta - \beta_1} \Tilde{\lambda}^p,
    \end{align}
    with the obvious definition of $C_6$. We proceed by choosing $\vartheta = \vartheta(\mathfrak{C},M_0,\beta,\kappa_\mathfrak{m} ,\chi_1) \in (0,1/16]$ such that
    \begin{align}\label{9.39996}
        C_5 \vartheta^{\frac{\nu\beta_1}{2}} \leq \frac{\chi_1}{2} \qquad\text{and}\qquad C_6 \vartheta^{\beta-\beta_1} \leq \frac{\chi_1}{2}.
    \end{align}
    Since $\varepsilon = \vartheta^{n+6+p}$ and $\delta_0 = \delta_0 (\mathfrak{C} , M_0 , \varepsilon)$, both are fixed by this choice and have the same dependencies as $\vartheta$. Due to (\ref{9.395}) this fixes $\delta$ and therefore $\eta(\delta)$. With these specifications in mind (\ref{9.396}) can be rewritten in the following form
    \begin{align*}
         E_\lambda(\varrho) \leq \bigg( \frac{\delta_0 \eta(\delta)}{3 C(p) \big( 1+\chi^{-\frac{1}{p}}\big)}  \bigg)^p,
    \end{align*}
    whereas (\ref{9.397}) yields
    \begin{align*}
         E_\lambda(\varrho) \leq \Big( \frac{\delta_0 \eta(\delta)}{6 C(p)}  \Big)^p
         \qquad\text{and}\qquad
         K\geq \frac{6}{\delta_0 \eta(\delta)}.
    \end{align*}
    Then the right-hand sides of the preceding inequalities define the constants $\varepsilon_2$ and $K$ in dependence of $\mathfrak{C} , M_0 , \beta, \kappa_{\mathfrak{m}} , \chi , \chi_1$ and $\eta(\,\cdot\,)$. Furthermore, the choice (\ref{9.39996}) of $\vartheta$ as well as estimates (\ref{9.39993}), (\ref{9.39994}) and (\ref{9.39995}) imply that
    \begin{align}\label{9.39997} 
        E_{\Tilde{\lambda}}(\vartheta \varrho) \leq \big( C_5 \vartheta^{\frac{\beta_1}{2}} + C_6 \vartheta^{\beta-\beta_1} \big) \Tilde{\lambda}^p \leq \chi_1 \Tilde{\lambda}^p
        \qquad\text{and}\qquad
        \big|(Du)_{\vartheta\varrho}^{(\Tilde{\lambda})} \big| \leq C_d \lambda
    \end{align}
    holds, which proves the assertion (\ref{9.304}).  \\
    By further decreasing $\varepsilon_2$ we can achieve that (\ref{9.300}) implies (\ref{9.398}). This finally concludes the proof of the Lemma.
\end{proof}

\subsection{Combining the degenerate and the non-degenerate regime}

In this subsection we join both the degenerate and non-degenerate case. The procedure will be as follows: Iterate Lemma \ref{Lem9.3} as long as we are in the degenerate regime. Once we are in the non-degenerate regime we obtain a suitable excess-decay estimate from Lemma \ref{Lem9.2}. \\
By $\alpha_1$ we denote the Hölder exponent from Lemma \ref{Lem9.3} corresponding to the choice of $M_0=1$, and if $\alpha_1 > 1/p$ we replace $\alpha_1$ by $1/p$; note $\alpha_1=\alpha_1(\mathfrak{C},\kappa_2)\in (0,1)$.
\begin{Lem}\label{Lem9.5}
    Let $2n/(n+2)< p\leq 2$. $M_1\geq 1$, $\alpha\in (0,\alpha_1)$ and
    \begin{equation*}
        u\in C\big([0,T];L^2\big(\Omega,\mathbb{R}^{N}\big) \big) \cap L^p\big(0,T;W^{1,p}\big(\Omega,\mathbb{R}^{N}\big) \big)
    \end{equation*}
    be a weak solution of system (\ref{Sys1}) in $\Omega_T$ that satisfies the assumptions of Section \ref{Assumptions}. Then, there exist constants 
    \begin{align*}
        &\mathfrak{m}=\mathfrak{m}(\mathfrak{C},M_1, \kappa_2),\\
        &C_d=C_d(\mathfrak{C},M_1,\kappa_{\mathfrak{m}}), \\ &\varepsilon_3=\varepsilon_3(\mathfrak{C},M_1,\kappa_{\mathfrak{m}}, \kappa_{8M}, \alpha, \omega_{4M}(\,\cdot\,),\eta(\,\cdot\,))\in (0,1] ,\\
        &\vartheta = \vartheta (\mathfrak{C} , M_1 , \kappa_{\mathfrak{m}} , \kappa_{8M}, \alpha, \omega_{4M}(\,\cdot\,))\in (0,1/16], \\
        &K = K ( \mathfrak{C} , M_1 , \kappa_{\mathfrak{m}}, \kappa_{8M}, \alpha, \omega_{4M}(\,\cdot\,), \eta(\,\cdot\,)) \geq 1, \\
        &C=C(\mathfrak{C},M_1,\kappa_{\mathfrak{m}}, \kappa_{8M}, \alpha, \omega_{4M}(\,\cdot\,), \eta(\,\cdot\,)), \\
        &\varrho_0 = \varrho_0(\mathfrak{C},M_1,\kappa_{\mathfrak{m}}, \kappa_{8M}, \alpha, \omega_{4M}(\,\cdot\,),\eta(\,\cdot\,)) \in (0,1/2],
    \end{align*}
    where $M:=\max\{M_1,C_d\}$, for which the following is true: Let $Q_{\varrho}\Subset\Omega_T$ be a parabolic cylinder with parameter $\varrho\in (0,\varrho_0]$ such that there holds
    \begin{align}\label{9.501}
        \big|(Du)_{z_0;\varrho}\big| \leq M_1
    \end{align}
    and
    \begin{align}\label{9.502}
        E_1 \big(z_0,\varrho,L_{z_0;\varrho} \big) \leq \varepsilon_3.
    \end{align}
    Then the limit
    \begin{align}\label{9.503}
        \Gamma_{z_0} \equiv \lim\limits_{\varrho\downarrow 0} (Du)_{z_0;\varrho},
    \end{align}
    a constant $m\in \N_0 \cup \infty$ and positive numbers $\{ \lambda_l \}^m_{l=0}$ exist that satisfy
    \begin{align}\label{9.504}
        \lambda_0 = 1, \qquad
        \lambda_1 \leq 1 \qquad\text{and}\qquad 
        \lambda_l \leq \vartheta^{(l-2)\alpha} \qquad\text{for }l \in \{2,\ldots,m\}
    \end{align}
    and we have the decay estimate
    \begin{align}\label{9.505}
        \mint{-}\limits_{Q_{\vartheta^l\varrho}^{(\lambda_l)}(z_0)} |Du - \Gamma_{z_0}|^p \, dz \leq C \lambda_l^p 
        \qquad
        \forall l\in \{0,\ldots,m\}.
    \end{align}
    Moreover, in the case $m<\infty$ there holds
    \begin{align}\label{9.506}
        \mint{-}\limits_{Q_{r}^{(\lambda_m)}(z_0)} |Du - \Gamma_{z_0}|^p \, dz \leq C \bigg( \frac{r}{\vartheta^m \varrho} \bigg)^{\alpha \beta p} \lambda_m^p 
        \qquad
        \forall 0<r\leq \vartheta^m \varrho
    \end{align}
    and
    \begin{align}\label{9.507}
        \frac{\lambda_m}{2K} \leq |\Gamma_{z_0}| \leq 2K \lambda_m,
    \end{align}
    whereas in the case $m=\infty$ we have $\Gamma_{z_0}=0$.
\end{Lem}
\begin{proof}
    As usual we set $z_0=0$ and abbreviate $\Phi_\lambda (0,\varrho,(Du)^{(\lambda)}_{\varrho})$ by $\Phi_\lambda (\varrho)$ as well as $E_\lambda (z_0,\varrho)$ by $E_\lambda (\varrho)$. To further shorten the notation we neglect the dependency of the constants of $\mathfrak{C}$ in the notation. Our objective is to iterate Lemma \ref{Lem9.3} until the assumptions of Lemma \ref{Lem9.2} are fulfilled. Since the application of Lemma \ref{Lem9.3} with $M_0=M_1$ in all iteration steps would lead to a Hölder exponent $\alpha_1$ depending on $M_1$ and thus on the gradient, we use the Lemma in the first iteration step with $M_0=M_1$ and thereafter with $M_0=1$. \\  
    We start by specifying all constants occurring in Lemma \ref{Lem9.2} and \ref{Lem9.3} in such a way that the statements of both Lemmas apply independently of the case $M_0=M_1$ or $M_0=1$. First let $\alpha_1(\kappa_2)$ denote the Hölder-exponent from Lemma \ref{Lem9.3} according to the choice $M_0=1$ and $\alpha_1^*=\alpha_1^*(\kappa_2, M_1)$ the corresponding exponent for the choice $M_0=M_1$. Without changing the dependency, we can assume $\alpha_1^*\leq \alpha_1$. Furthermore, if $\alpha_1>1/p$, we replace $\alpha_1$ by $1/p$. Next, we choose $\alpha\in (0,\alpha_1)$ in Lemma \ref{Lem9.3} and let $\beta_0:=p\alpha<1$ in Lemma \ref{Lem9.2}. Due to $\alpha<\alpha_1\leq 1/p$, this choice is admissible. Let
    \begin{align}\label{9.51}
        \mathfrak{m} = \mathfrak{m}(\kappa_2, M_1)
        \qquad\text{and}\qquad
        C_d = C_d(M_1,\kappa_{\mathfrak{m}})
    \end{align}
    be the constants of Lemma \ref{Lem9.3} according to the choice $M_0=M_1$. In the following we will neglect the dependency of $\mathfrak{m}$ on $\kappa_2$. Without loss of generality we can assume $\mathfrak{m}(M_1)\geq \mathfrak{m}(1)$ resulting in $\kappa_{\mathfrak{m}(M_1)}\geq \kappa_{\mathfrak{m}(1)}$ and $C_d(M_1,\kappa_{\mathfrak{m}(M_1)}) \geq C_d(1,\kappa_{\mathfrak{m}(1)})$. Thus Lemma \ref{Lem9.3} continues to hold for $M_0=1$ with the constants given in (\ref{9.51}). Define 
    \begin{align*}
        M:= \max\{ M_1 , C_d \},
    \end{align*}
    which fixes the constant 
    \begin{align*}
        \varepsilon_1(C_d) = \varepsilon_1(C_d,\beta,\alpha,\kappa_{8M}, \omega_{4M}(\,\cdot\,))
    \end{align*}
    of Lemma \ref{Lem9.2}. Now we choose
    \begin{align*}
        \chi_1= \varepsilon_1(C_d)
    \end{align*}
    in Lemma \ref{Lem9.3}. This fixes the constants
    \begin{align}\label{9.55}
        \vartheta = \vartheta(M_1,\kappa_{\mathfrak{m}},\chi_1,\alpha)\in (0,1/16], \quad
        K = K(M_1,\kappa_{\mathfrak{m}},\chi_1,\alpha,\eta(\,\cdot\,)) \geq M.
    \end{align}
    Furthermore, we can assume $K(M_1,\kappa_{\mathfrak{m}(M_1)}) \geq K(1,\kappa_{\mathfrak{m}(1)})$ and $\vartheta(M_1,\kappa_{\mathfrak{m}(M_1)}) \leq \vartheta(1,\kappa_{\mathfrak{m}(1)})$ , which implies that Lemma \ref{Lem9.3} is also applicable with the constants defined in (\ref{9.55}) for the choice $M_0=1$. Next we choose
    \begin{align*}
        \varepsilon_1(K) = \varepsilon_1(K,\beta,\alpha,\kappa_{8M}, \omega_{4M}(\,\cdot\,))
    \end{align*}
    according to Lemma \ref{Lem9.2} and let
    \begin{align*}
        \chi= \varepsilon_1(K).
    \end{align*}
    This fixes
    \begin{align*}
        \varepsilon_2 = \varepsilon_2 (M_1,\kappa_{\mathfrak{m}},\chi,\chi_1,\alpha,\eta(\,\cdot\,))\in (0,1]
    \end{align*}
    in Lemma \ref{Lem9.3}. Again we assume $ \varepsilon_2 (M_1,\kappa_{\mathfrak{m}(M_1)}) \leq  \varepsilon_2 (1,\kappa_{\mathfrak{m}(1)})$ implying that Lemma \ref{Lem9.3} is usable in the case $M_0=1$ with $\varepsilon_2 (M_1,\kappa_{\mathfrak{m}(M_1)}) $ instead of $\varepsilon_2 (1,\kappa_{\mathfrak{m}(1)}) $. Let $C_+$ be the constant of Lemma \ref{Lem2.4} and finally define
    \begin{align*}
        \varepsilon_3 = \varepsilon_2^{n_0},\quad\text{where} \quad
        n_0 = 1+\frac{\ln(2^{-1} C_+^{-1} \vartheta^{k(n+2)+\mathds{1}_{\{p<2\}}\alpha_1(2-p)})}{\ln(\varepsilon_2)}
        \quad\text{and}\quad 
        k= \left \lceil{\frac{\ln(\varepsilon_2)}{\alpha p \ln(\vartheta)}}\right \rceil.
    \end{align*}
    Furthermore, assume for the radius $\varrho\leq (\varepsilon_2/(16C_*))^{1/\beta}$, where $C_*$ is the constant on the right-hand side of estimate (\ref{9.03}). By definition, we have $n_0>1$ and $k\in\N$ such that there holds:
    \begin{align}\label{9.591}
        \varepsilon_3 =  \frac{1}{2C_+} \vartheta^{k(n+2)+\mathds{1}_{\{p<2\}} \alpha_1(2-p)} \varepsilon_2
        \qquad\text{and}\qquad
        \vartheta^{k\alpha p}\leq \varepsilon_2
        \qquad\text{and}\qquad
        \varrho^\beta\leq \frac{\varepsilon_2}{16 C_*}.
    \end{align}
    This concludes the set up and we are finally in the position to proof the Lemma. As mentioned this will be achieved by iterating, distinguishing the cases if
    \begin{align}\label{D1s} 
        \chi \big| (Du)_{\vartheta^i \varrho}^{(\lambda_i)} \big|^p \leq E_{\lambda_i}(\vartheta^i \varrho) 
        \qquad\text{or}\qquad
        \big| (Du)_{\vartheta^i \varrho}^{(\lambda_j)} \big| \leq \frac{\lambda_i}{K} \tag{$D_i$}
    \end{align}
    holds at step $i\in \N_0$ or not. If (\ref{D1s}) holds we call it the \textbf{degenerate case}, whereas when (\ref{D1s}) fails, namely when $>$ occurs in both inequalities, we call it the \textbf{non-degenerate case}. As soon as the non-degenerate case occurs, Lemma \ref{Lem9.2} will be applicable by choice of constants and the iteration stops. In the degenerate case at level $i$ we apply Lemma \ref{Lem9.3} and obtain excess decay for certain scaling parameters of the next cylinder. In order to proceed, we to distinguish whether on the next cylinder the \textbf{alternative degenerate case}
    \begin{align}\label{D2s} 
        \chi_1 \big| (Du)_{\vartheta^i \varrho}^{(\lambda_i)} \big|^p \leq E_{\lambda_i}(\vartheta^i \varrho) 
        \qquad\text{or}\qquad
        \big| (Du)_{\vartheta^i \varrho}^{(\lambda_i)} \big| \leq \frac{\lambda_i}{C_d} \tag{$AD_i$}
    \end{align}
    holds. If this is not the case, we again have by choice of constants that Lemma \ref{Lem9.2} will be applicable after switching the scaling parameter $\lambda_i$ and the iteration stops, whereas if this is the case we obtain a excess decay estimate. We visualize the procedure with the following picture:
\begin{figure}[ht]\label{fig1}
\tikzset{
  treenode/.style = {shape=rectangle, rounded corners,
                     draw, align=center,
                     top color=white, bottom color=blue!20},
  root/.style     = {treenode, font=\Large, bottom color=red!30},
  root2/.style     = {treenode, font=\footnotesize, bottom color=green!30},
  env/.style      = {treenode, font=\footnotesize},
  dummy/.style    = {circle,draw}
}
\begin{tikzpicture}
  [
    grow                    = right,
    sibling distance        = 8em,
    level distance          = 11em,
    edge from parent/.style = {draw, -latex},
    every node/.style       = {font=\footnotesize},
    sloped
  ]
  \node [root] {$Q_{\vartheta^i\varrho}^{(\lambda_i)}$}
    child { node [root2] {(\ref{9.506})}
      edge from parent node[above] {$\neg$(\ref{D1s}):}
      node[below]  {Lemma \ref{Lem9.2}} }
    child { node [env] {$Q_{\vartheta^{i+1}\varrho}^{(\Tilde{\lambda})}$ with \\
    $\Tilde{\lambda}\in [\vartheta^{\alpha_1}\lambda_{i},\lambda_{i}]$ \\
    \& (\ref{9.505})}
      child { 
        node [env] {Distinguish: \\$\big| (Du)_{\vartheta^{i+1} \varrho}^{(\vartheta^{\alpha_1}\lambda_{i})}\big| $}
        child { 
            node [root2] {(\ref{9.506}) with \\
            $\lambda_{i+1}=\vartheta^{\alpha_1}\lambda_{i}$}
                edge from parent node [above] {$\leq C_d\vartheta^{\alpha_1}\lambda_{i}$:}
                    node [below] {Lemma \ref{Lem9.2}} }
        child { 
            node [root2] {(\ref{9.506}) with \\
            $\lambda_{i+1}\in [\vartheta^{\alpha_1}\lambda_{i},\lambda_{i}]$}
                edge from parent node [above] {$>C_d\vartheta^{\alpha_1}\lambda_{i}$:}
                    node [below] {Lemma \ref{Lem9.2}}}
            edge from parent node[above] {$\neg$(\ref{D2s}):} 
                }
      child { 
        node [root] {$Q_{\vartheta^{i+1}\varrho}^{(\lambda_{i+1})}$}
        child{
            node [root] {...}
            edge from parent node[above] {Continue} node[below] {iteration}
            }
            edge from parent node[above] {(\ref{D2s}):}
                node[below]  {$\lambda_{i+1}=\vartheta^{\alpha_1}\lambda_{i}$}
                }
         edge from parent node[above] {(\ref{D1s}):}
             node[below]  {Lemma \ref{Lem9.3}} };
\end{tikzpicture}
\caption{Here, the cylinders represent the scaling parameters used in each case under consideration and the arrows indicate the conditions and Lemmas that allow us to move to the next step. The steps at which the iteration progresses are highlighted in red, while it stops at the green steps. The intermediate steps are highlighted in blue.}
\end{figure}
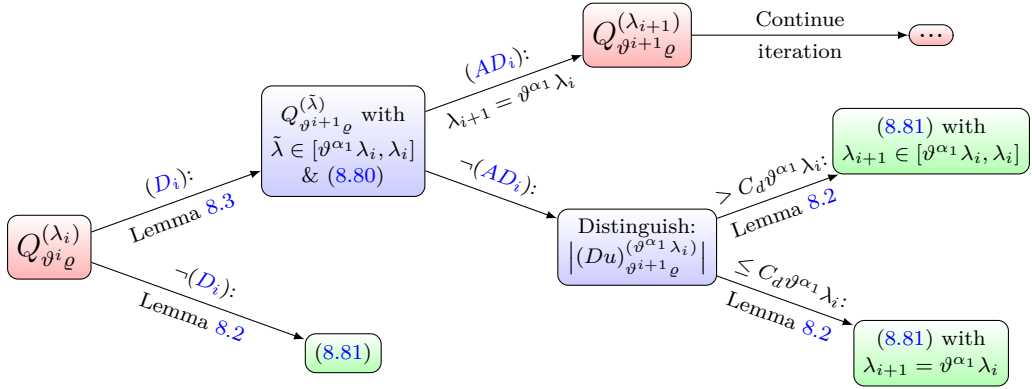\\
    We initialize the iteration by setting
    \begin{align*}
        \lambda_0=1
    \end{align*}
    and obtain immediately from (\ref{9.591}) that (\ref{9.300}) is satisfied for every iteration step $i\in \N_0$, since the later constructed cylinders $Q_{\vartheta^i\varrho}^{(\lambda_i)}$ will be nested. \\
    Assume that we have the \textbf{non-degenerate case} at $i=0$, meaning that (\ref{D1s}) does not hold for $i=0$. To treat this case we verify the assumptions of Lemma \ref{Lem9.2} for the choice of parameters $M_1$, $K$ and $\lambda=\lambda_0 = 1$. We start by establishing that (\ref{9.201}) holds, which is an immediate consequence of (\ref{9.501}). Additionally, hypothesis (\ref{9.202}), i.e. 
    \begin{align*}
        \frac{\lambda_0}{K} \leq \big| (Du)_{\varrho}^{(\lambda_0)} \big| \leq K \lambda_0
    \end{align*}
    is satisfied due to $M_1\leq K$ and the second inequality of $\neg$(\ref{D1s}) for $i=0$. The remaining assumption (\ref{9.203}) is again an immediate consequence of the first inequality of $\neg$(\ref{D1s}) for $i=0$, since we have chosen $\chi=\varepsilon_1(K)$. Then Lemma \ref{Lem9.2} yields that the limit 
    \begin{align*}
        \Gamma_{0}\equiv \lim\limits_{r\downarrow 0} (Du)_r
    \end{align*}
    exists and we set $m=0$. To establish (\ref{9.504})-(\ref{9.507}) for this particular choice of $m$, we first observe that (\ref{9.504}) is trivially satisfied. Moreover, (\ref{9.205}) and (\ref{9.502}) with our choice of $\beta_0$ allow us to deduce that
    \begin{align*}
        \mint{-}\limits_{Q_r^{(\lambda_0)}} |Du - \Gamma_{0} |^p \, dz 
        \leq 
        C \Big( \frac{r}{\varrho} \Big)^{\alpha\beta p} E_{\lambda_0}(\varrho)
        \leq
        C \Big( \frac{r}{\varrho} \Big)^{\alpha\beta p} \lambda_0^p
        \qquad \forall 0<r\leq \varrho,
    \end{align*}
    where the dependency of $C$ is given by $C=C(\kappa_{8M},K,\beta,\alpha)$. This shows (\ref{9.505}) and (\ref{9.506}) for $m=0$. Furthermore, (\ref{9.507}) can be deduced from (\ref{9.206}). Therefore, the Lemma follows for $m=0$ in the non-degenerate case. \\
    Let us now assume that the \textbf{degenerate case} (\ref{D1s}) holds for $i=0$. This case is treated by applying Lemma \ref{Lem9.3} with parameters $M_0=M_1$ and $\lambda=\lambda_0=1$. We again start with the observation that due to (\ref{9.501}) hypothesis (\ref{9.301}) is satisfied. Furthermore, we have by definition of the degenerate case (\ref{D1s}) that assumption (\ref{9.302}) holds. The smallness condition (\ref{9.303}) is easily verified by (\ref{9.502}) and definition (\ref{9.591}) of $\varepsilon_3$. Therefore, we can apply Lemma \ref{Lem9.3}, which yields that for every $\Tilde{\lambda}\in [\vartheta^{\alpha^*_1}\lambda_0,\lambda_0]$ we have $Q_{\vartheta\varrho}^{(\Tilde{\lambda})}\Subset Q_{\varrho}^{(\lambda)}$,
    \begin{align}\label{9.594}
        E_{\Tilde{\lambda}}(\vartheta \varrho) \leq \chi_1 \Tilde{\lambda}^p
        \qquad\text{and}\qquad
        \big|(Du)_{\vartheta\varrho}^{(\Tilde{\lambda})} \big| \leq C_d \lambda.
    \end{align}
    We define $\lambda_1 = \vartheta^{\alpha^*_1}\lambda_0$. At this stage we distinguish, if the alternative degenerate case (\ref{D2s}) holds at level $i=1$ or not. Note that in this case Lemma \ref{Lem9.2} with $C_d$ instead of $K$ must be applied. In the \textbf{alternative non-degenerate case} at level $i=1$, i.e. when (\ref{D2s}) fails to hold for $i=1$, we distinguish two subcases:\\
    \textbf{Case 1}: If we have
    \begin{align*}
        \big|(Du)_{\vartheta\varrho}^{(\lambda_1)} \big| \leq C_d \lambda_1 \leq C_d\leq M,
    \end{align*}
    then (\ref{9.201}) of Lemma \ref{Lem9.2} is satisfied. Furthermore, the second inequality of $\neg$(\ref{D2s}) for $i=1$ implies that 
    \begin{align*}
        \frac{\lambda_1}{C_d} \leq \big| (Du)_{\varrho}^{(\lambda_1)} \big| \leq C_d \lambda_1.
    \end{align*}
    Thus we know that (\ref{9.202}) of Lemma \ref{Lem9.2} holds with $K = C_d$. Finally, the smallness assumption (\ref{9.203}) is an immediate consequence of of the first inequality  of $\neg$(\ref{D2s}) for $i=1$ by choice of $\chi_1=\varepsilon_1(C_d)$. In summary, we showed that Lemma \ref{Lem9.2} is applicable with the parameters $(\vartheta\varrho, \lambda_1,C_d)$ instead of $(\varrho,\lambda,K)$. \\
    \textbf{Case 2}: If we have
    \begin{align*}
        \big|(Du)_{\vartheta\varrho}^{(\lambda_1)} \big| \geq C_d \lambda_1,
    \end{align*}
    then the continuity of the map
    \begin{align*}
        \lambda \mapsto \big|(Du)_{\vartheta\varrho}^{(\lambda)} \big|
    \end{align*}
    and (\ref{9.594})\textsubscript{2} implies the existence of $\lambda_1^*\in [\lambda_1,\lambda]$ such that 
    \begin{align*}
        \big|(Du)_{\vartheta\varrho}^{(\lambda^*_1)} \big| = C_d \lambda^*_1 \leq C_d \leq M.
    \end{align*}
    By abuse of notation, we redefine $\lambda_1$ to be equal to $\lambda_1^*$. From (\ref{9.594})\textsubscript{1} we can deduce that 
    \begin{align*}
        E_{\lambda_1}(\vartheta \varrho) \leq \chi_1 \lambda_1^p 
        =
        \frac{\chi_1}{C_d^p} \big|(Du)_{\vartheta\varrho}^{(\lambda_1)} \big|^p \leq 
        \chi_1 \big|(Du)_{\vartheta\varrho}^{(\lambda_1)} \big|^p
    \end{align*}
    and once more Lemma \ref{Lem9.2} is applicable with the parameters $(\vartheta\varrho, \lambda_1,C_d)$ replacing $(\varrho,\lambda,K)$. \\
    Since Lemma \ref{Lem9.2} can be applied in both cases, we obtain that the limit $\Gamma_0$ exists and set $m=1$. It remains to establish (\ref{9.504})-(\ref{9.507}) for this particular choice of $m$. First observe that (\ref{9.504}) is satisfied, since $\lambda_1\in [\vartheta^{\alpha^*_1}\lambda_0,\lambda_0]$ holds even after replacing $\lambda_1$. Moreover, (\ref{9.205}) and (\ref{9.594})\textsubscript{3} allow us to deduce for every $ 0<r\leq \vartheta\varrho$ that
    \begin{align}\label{9.595}
        \mint{-}\limits_{Q_r^{(\lambda_1)}} |Du - \Gamma_{0} |^p \, dz 
        \leq 
        C \Big( \frac{r}{\vartheta\varrho} \Big)^{\alpha\beta p} E_{\lambda_1} (\vartheta\varrho)
        \leq
        C \Big( \frac{r}{\vartheta\varrho} \Big)^{\alpha\beta p} \lambda_1^p
        \leq
        C \Big( \frac{r}{\vartheta\varrho} \Big)^{\alpha\beta p} \lambda_1^p,
    \end{align}
    where the dependency of $C$ is given by $C=C(\kappa_{8M+1},C_d,\beta,\alpha)$ and we keep in mind that $\lambda_1$ actually depends on the considered case. This shows (\ref{9.506}) for $m=1$. Furthermore, (\ref{9.506}) immediately follows from (\ref{9.206}), since $C_d\leq K$. It remains to show (\ref{9.505}) for $l=0$. This will be achieved by means of the preceding inequality (\ref{9.595}) and the smallness assumption (\ref{9.502}). We estimate
    \begin{align}\label{9.596}
        \bigg( \mint{-}\limits_{Q_{\varrho}^{(\lambda_0)}} & |Du - \Gamma_{0} |^p \, dz \bigg)^{\frac{1}{p}}
        \nonumber \\ & \leq
        \bigg( \mint{-}\limits_{Q_{\varrho}^{(\lambda_0)}} \big| Du - (Du)_{\varrho}^{(\lambda_0)}  \big|^p \, dz \bigg)^{\frac{1}{p}}
        +
        \big| (Du)_{\varrho}^{(\lambda_0)} - (Du)_{\vartheta\varrho}^{(\lambda_1)} \big|
        +
        \big| (Du)_{\vartheta\varrho}^{(\lambda_1)} - \Gamma_{0} \big|
        \nonumber \\ & \leq
        C \Big[ E_1(\varrho) + \mathds{1}_{\{p<2\}} \big|(Du)^{(\lambda_0)}_{\varrho}\big|^{\frac{p(2-p)}{2}} E_1^{\frac{p}{2}}(\varrho)
        \Big]^{\frac{1}{p}}
        +
        \bigg( \mint{-}\limits_{Q_{\vartheta\varrho}^{(\lambda_1)}} \big| Du - (Du)_{\varrho}^{(\lambda_0)}  \big|^p \, dz \bigg)^{\frac{1}{p}}
        \nonumber \\ & \qquad +
        \bigg( \mint{-}\limits_{Q_r^{(\lambda_1)}} |Du - \Gamma_{0} |^p \, dz \bigg)^{\frac{1}{p}}
         \leq
        C \lambda_0 
        +
        C \vartheta^{-\frac{n+2}{p}} \lambda_1^{\frac{p-2}{p}} \lambda_0
        +
        C\lambda_1 \leq C \lambda_0.
    \end{align}
    This concludes the non-degenerate case. \\
    For the first iteration step it remains to treat the \textbf{alternative degenerate case} (\ref{D2s}) at level $i=1$. Here we do not have to adapt $\lambda_1$ as in the preceding case and hence we have 
    \begin{align}\label{9.597}
        \lambda_1 = \vartheta^{\alpha^*_1}\lambda_0.
    \end{align}
    Furthermore, the inequalities (\ref{D2s})\textsubscript{1} and (\ref{9.594})\textsubscript{1} with the choice $\Tilde{\lambda}= \lambda_1$ imply
    \begin{align}\label{9.5971}
        \big|(Du)_{\vartheta\varrho}^{(\lambda_1)} \big| 
        \leq 
        \bigg(\frac{E_{\lambda_1}(\vartheta\lambda)}{\chi_1}\bigg)^{\frac{1}{p}}
        \leq 
        \lambda_1,
    \end{align}
    whereas (\ref{D2s})\textsubscript{2} trivially implies the same inequality. Since this is the only case where we could not obtain the results of Lemma \ref{Lem9.5} from Lemma \ref{Lem9.2}, we have to continue with the iteration scheme. \\
    Now we perform the \textbf{iteration step $j\to j+1$} when $j\geq 1$: Hence we have that the non-degenerate case does not occur up to level $j$. From the preceding iteration steps we know that the companions of (\ref{9.595}), (\ref{9.597}) and (\ref{9.5971}) hold for $l\in \{0,\ldots,j\}$, which implies
    \begin{align}\label{9.598}
        \begin{cases}
            E_{\lambda_l}(\vartheta^l\varrho) \leq \chi_1 \lambda_l^p & \text{for } l \in \{0,\ldots,j\},\\
            \big| (Du)_{\varrho}^{(\lambda_0)} \big|\leq M_1 &\text{and}\quad 
            \big| (Du)_{\vartheta^l \varrho}^{(\lambda_l)} \big| \leq \lambda_l \qquad
            \text{for } l \in \{1,\ldots,j\},\\
            \lambda_1 = \vartheta^{\alpha^*} \lambda_0 
            &\text{and}\quad 
            \lambda_l = \vartheta^\alpha \lambda_{l-1} \qquad
            \text{for } l \in \{2,\ldots,j\}.
        \end{cases}
    \end{align}
    We infer from (\ref{9.02}) and the preceding inequality (\ref{9.598})\textsubscript{1} that for any $l\in \{0,\ldots,j\}$ there holds
    \begin{align}\label{9.5981}
        \mint{-}\limits_{Q_{\vartheta^l \varrho}^{(\lambda_l)}} \big| Du - (Du)_{\vartheta^l \varrho}^{(\lambda_l)} \big|^p \, dz
        \leq
        C(p) \Big[ E_{\lambda_l}(\vartheta^l \varrho) + \mathds{1}_{\{p<2\}}\big|(Du)_{\vartheta^l \varrho}^{(\lambda_l)} \big|^{\frac{p(2-p)}{2}} E^{\frac{p}{2}}_{\lambda_l}(\vartheta^l \varrho) \Big] \leq C(p) \lambda_l^p.
    \end{align}
    Now we again distinguish between the non-degenerate and degenerate case at level $j$. If we assume the \textbf{non-degenerate case}, that is $\neg$(\ref{D1s}) for $i=j$, then due to (\ref{9.598})\textsubscript{2} and $\lambda_j\leq 1$ we have that hypothesis (\ref{9.201}) of Lemma \ref{Lem9.2} is satisfied with the parameters $(\vartheta^j\varrho,\lambda_j,K,1)$ replacing $(\varrho,\lambda,K,M)$. Furthermore, the second inequality of $\neg$(\ref{D1s}), $K\geq 1$ and (\ref{9.598})\textsubscript{2} ensures
    \begin{align*}
        \frac{\lambda_j}{K} \leq \big| (Du)_{\vartheta^j \varrho}^{(\lambda_j)} \big| \leq \lambda_j K
    \end{align*}
    and we deduce that (\ref{9.202}) holds. The first inequality of $\neg$(\ref{D1s}) for $i=j$ implies assumption (\ref{9.203}) due to our particular choice of $\chi=\varepsilon_1(K)$. Therefore, Lemma \ref{Lem9.2} is applicable and we infer that the limit $\Gamma_{0}$ exists. Now we set $m=j$ and verify the remaining assertions (\ref{9.504})-(\ref{9.506}) of the Lemma for this particular choice. We start with (\ref{9.504}), which is a easy consequence of (\ref{9.598})\textsubscript{3}. Let us now establish (\ref{9.505}) for $l=m$ and (\ref{9.506}) for our choice of $m$, with the help of the conclusion (\ref{9.205}) of Lemma \ref{Lem9.2}. If we use (\ref{9.598})\textsubscript{1} and recall $\beta_0 = \alpha p$, we obtain
    \begin{align}\label{9.599}
        \mint{-}\limits_{Q_r^{(\lambda_m)}} |Du - \Gamma_{0} |^p \, dz 
        \leq 
        C \Big( \frac{r}{\vartheta^m \varrho} \Big)^{\alpha\beta p} E_{\lambda_m} (\vartheta^m \varrho)
        \leq
        C \Big( \frac{r}{\vartheta^m \varrho} \Big)^{\alpha\beta p} \lambda_m^p
        \qquad \forall 0<r\leq \vartheta^m \varrho,
    \end{align}
    where $C=C(\kappa_{8M},K,\beta,\alpha)$. The second statement (\ref{9.206}) of Lemma \ref{Lem9.2} implies (\ref{9.507}). To conclude the non-degenerate case we only need to verify (\ref{9.505}) for $l\in \{0,\ldots,m-1\}$. To this end, we use the inequalities (\ref{9.5981}) and (\ref{9.599}) to estimate
    \begin{align}\label{9.5991}
        \bigg( \mint{-}\limits_{Q_{\vartheta^l \varrho}^{(\lambda_l)}} & |Du - \Gamma_{0} |^p \, dz \bigg)^{\frac{1}{p}}
        \nonumber \\ & \leq
        \bigg( \mint{-}\limits_{Q_{\vartheta^l \varrho}^{(\lambda_l)}} \big| Du - (Du)_{\vartheta^l \varrho}^{(\lambda_l)} \big|^p \, dz \bigg)^{\frac{1}{p}}
        +
        \sum\limits_{i=l}^{m-1}
        \big| (Du)_{\vartheta^i \varrho}^{(\lambda_i)}- (Du)_{\vartheta^{i+1} \varrho}^{(\lambda_{i+1})} \big|
        +
        \big| (Du)_{\vartheta^m\varrho}^{(\lambda_m)} - \Gamma_{0} \big|
        \nonumber \\ & \leq
        C(p) \lambda_l
        +
        \sum\limits_{i=l}^{m-1}
        \bigg( \mint{-}\limits_{Q_{\vartheta^{i+1}\varrho}^{(\lambda_{i+1})}} \big| Du - (Du)_{\vartheta^{i}\varrho}^{(\lambda_i)}  \big|^p \, dz \bigg)^{\frac{1}{p}}
        +
        \bigg( \mint{-}\limits_{Q_r^{(\lambda_m)}} |Du - \Gamma_{0} |^p \, dz \bigg)^{\frac{1}{p}}
        \nonumber \\ & \leq
        C\lambda_l
        +
        C \vartheta^{-\frac{n+2}{p}} \sum\limits_{i=l}^{m-1} \bigg( \frac{\lambda_i}{\lambda_{i+1}} \bigg)^{\frac{2-p}{p}} \lambda_i
        +
        C\lambda_m 
        \leq 
        C\lambda_l
        +C
        \sum\limits_{i=l}^{m} \lambda_i
        \leq
        C \lambda_l.
    \end{align}
    Therefore the assertion of the Lemma holds for $m=j$. \\
    If on the contrary the \textbf{degenerate case} (\ref{D1s}) holds for $i=j$, we want to apply Lemma \ref{Lem9.3} with the parameters $(\vartheta^j\varrho, \lambda_j,1)$ replacing $(\varrho,\lambda,M_0)$. We now verify that the hypotheses of the Lemma are satisfied. We start with assumption (\ref{9.301}), which is trivially fulfilled by (\ref{9.598})\textsubscript{2}. Since (\ref{9.302}) is implied by the case under consideration (\ref{D1s}), only the smallness condition (\ref{9.303}) needs to be validated. The first upper bound, i.e. $E_{\lambda_j}\leq \lambda_j^p$, is ensured by (\ref{9.598})\textsubscript{1}. Therefore, we only need to establish $E_{\lambda_j}(\vartheta^j \varrho) \leq \varepsilon_2$, for which we distinguish two cases. In the first case $j>k$ we obtain by choice (\ref{9.591}) of $\varepsilon_2$ that $\lambda_j^p\leq \vartheta^{(j-1)\alpha p}\leq \vartheta^{k\alpha p}\leq \varepsilon_2$. For the treatment of the second case $j\leq k$ we use Lemma \ref{Lem2.4}, the hypothesis (\ref{9.503}), the already existing estimates (\ref{9.598})\textsubscript{1\&2} together with the inequalities $\vartheta^{j\alpha_1} \leq \lambda_j\leq 1$, $\vartheta< 1/4$, (\ref{9.03}) and (\ref{9.591}) to estimate
    \begin{align}\label{9.59911}
        E_{\lambda_j}(\vartheta^j \varrho) 
        &\leq
        \mint{-}\limits_{Q_{\vartheta^{j}\varrho}^{(\lambda_{j})}}  \Big| V_{| (Du)_{\vartheta^{j}\varrho}^{(\lambda_j)}|} \Big( Du - (Du)_{\vartheta^{j}\varrho}^{(\lambda_j)} \Big) \Big|^2 \, dz
        + 
        C_*(\vartheta^j \varrho)^\beta  \big| (Du)_{\vartheta^j \varrho}^{(\lambda_j)} \big|^{p} \Bigg( \bigg[
        \Phi_{\lambda_j}(\vartheta^j\varrho) 
        \nonumber \\ & \qquad + 
        \mathds{1}_{\{p<2\}} \big| (Du)_{\vartheta^j \varrho}^{(\lambda_j)} \big|^{\frac{p(2-p)}{2}} \Phi_{\lambda_j}(\vartheta^j\varrho)^\frac{p}{2} 
        +
        \big| (Du)_{\vartheta^j \varrho}^{(\lambda_j)} \big|^p + \lambda^{p(2-p)} \big| (Du)_{\vartheta^j \varrho}^{(\lambda_j)} \big|^{p(p-1)} 
        \nonumber \\ & \qquad +
        \lambda^{p(2-p)} \Big( \Phi_{\lambda_j}(\vartheta^j\varrho) 
        + 
        \mathds{1}_{\{p<2\}} \big| (Du)_{\vartheta^j \varrho}^{(\lambda_j)} \big|^{\frac{p(2-p)}{2}} \Phi_{\lambda_j}(\vartheta^j\varrho)^\frac{p}{2} \Big)^{p-1} \bigg]^{\frac{\beta}{p}} + ( 1+\lambda_j^{2-p})^{\frac{\beta}{2}} \Bigg)
        \nonumber \\ & \leq
        C_+\lambda_j^{p-2} \vartheta^{-j(n+2)} \Phi_1(\varrho)
        +
        C_* 8\varrho^\beta
        \leq
        C(p) \vartheta^{-k[(n+2)+\mathds{1}_{\{p<2\}} \alpha_1(2-p)]} \varepsilon_3 + 8 C_* \varrho^\beta
        \leq 
        \varepsilon_2.
    \end{align}
    This allows us to apply Lemma \ref{Lem9.3} with the parameters $(M_0,\alpha_1)$ replacing $(1,\alpha_1)$, which yields for every $\Tilde{\lambda}\in [\vartheta^{\alpha_1}\lambda_j,\lambda_j]$ that $Q_{\vartheta\varrho}^{(\Tilde{\lambda})}\Subset Q_{\varrho}^{(\lambda_j)}$,
    \begin{align}\label{9.5993}
        E_{\Tilde{\lambda}}(\vartheta^{j+1} \varrho) \leq \chi_1 \Tilde{\lambda}^p 
        \qquad\text{and}\qquad
        \big|(Du)_{\vartheta^{j+1} \varrho}^{(\Tilde{\lambda})} \big| \leq C_d \lambda_j.
    \end{align}
    We set $\lambda_{j+1}= \vartheta^{\alpha_1}\lambda_j$ and distinguish again the alternative non-degenerate and degenerate case at the level $j+1$. In the \textbf{alternative non-degenerate case}, i.e. when (\ref{D2s}) fails for $i=j+1$, we once more consider two sub-cases: \\
    \textbf{Case 1}: If we have
    \begin{align*}
        \big|(Du)_{\vartheta^{j+1} \varrho}^{(\lambda_{j+1})} \big| \leq C_d \lambda_{j+1} \leq C_d\leq M,
    \end{align*}
    then (\ref{9.201}) of Lemma \ref{Lem9.2} is satisfied. Furthermore, the second inequality of $\neg$(\ref{D2s}) for $i=j+1$ implies that
    \begin{align*}
        \frac{\lambda_{j+1}}{C_d} \leq \big| (Du)_{\vartheta^{j+1} \varrho}^{(\lambda_{j+1})} \big| \leq C_d \lambda_{j+1}.
    \end{align*}
    and we observe that (\ref{9.202}) of Lemma \ref{Lem9.2} is satisfied for $K = C_d$. Finally, the smallness assumption (\ref{9.203}) is an immediate consequence of the first inequality of $\neg$(\ref{D2s}) for $i=1$ by choice of $\chi_1=\varepsilon_1(C_d)$. In summary, we showed that Lemma \ref{Lem9.2} is applicable with $(\vartheta\varrho, \lambda_{j+1},C_d)$ replacing $(\varrho,\lambda,K)$. \\
    \textbf{Case 2}: If we have on the contrary that
    \begin{align*}
        \big|(Du)_{\vartheta^{j+1} \varrho}^{(\lambda_{j+1})} \big| \geq C_d \lambda_{j+1},
    \end{align*}
    then (\ref{9.5993})\textsubscript{2} implies, similar to the preceding, the existence of $\lambda_{j+1}^*\in [\lambda_{j+1},\lambda_j]$ such that 
    \begin{align*}
        \big|(Du)_{\vartheta^{j+1} \varrho}^{(\lambda^*_{j+1})} \big| = C_d \lambda^*_{j+1} \leq C_d \leq M.
    \end{align*}
    By abuse of notation, we once more redefine $\lambda_{j+1}$ to be equal to $\lambda_{j+1}^*$. Then (\ref{9.5993})\textsubscript{1} implies 
    \begin{align*}
        E_{\lambda_{j+1}}(\vartheta^{j+1} \varrho) \leq \chi_1 \lambda_{j+1}^p 
        =
        \frac{\chi_1}{C_d^p} \big|(Du)_{\vartheta^{j+1} \varrho}^{(\lambda_1)} \big|^p \leq 
        \chi_1 \big|(Du)_{\vartheta^{j+1} \varrho}^{(\lambda_1)} \big|^p
    \end{align*}
    and Lemma \ref{Lem9.2} is applicable with the parameters $(\vartheta\varrho, \lambda_{j+1},C_d)$ replacing $(\varrho,\lambda,K)$. \\
    Since Lemma \ref{Lem9.2} can be applied in both cases, we can deduce that the limit $\Gamma_0$ again exists and set $m=j+1$. It remains to establish (\ref{9.504})-(\ref{9.507}) for this particular choice of $m$. We observe that (\ref{9.504}) holds, since $\lambda_{j+1}\in [\vartheta^{\alpha_1}\lambda_j,\lambda_j] $. From (\ref{9.205}) and (\ref{9.5993})\textsubscript{3} we further obtain for every $0<r\leq \vartheta^{j+1}\varrho$, similar to (\ref{9.595}), that there holds
    \begin{align*}
        \mint{-}\limits_{Q_r^{(\lambda_{j+1})}} |Du - \Gamma_{0} |^p \, dz 
        \leq 
        C \Big( \frac{r}{\vartheta^{j+1}\varrho} \Big)^{\alpha\beta p} E_{\lambda_1} (\vartheta^{j+1}\varrho)
        \leq
        C \Big( \frac{r}{\vartheta^{j+1}\varrho} \Big)^{\alpha\beta p} \lambda_{j+1}^p,
    \end{align*}
    where $C=C(\kappa_{8M},C_d,\beta,\alpha)$. This shows (\ref{9.506}) with $m=j+1$ and (\ref{9.505}) for $l=m$. Furthermore, (\ref{9.507}) follows again from (\ref{9.206}), where we once more use that $C_d\leq K$. It remains to show (\ref{9.505}) for $l\in \{ 0,\ldots,m-1 \}$. Since this can be done in exactly the same manner as (\ref{9.5991}), we skip the computation. Therefore, the Lemma follows for the parameter $m=j+1$.\\
    It remains to treat the \textbf{alternative degenerate case} \ref{D2s} for $i=j+1$. Here we do not have to adapt $\lambda_{j+1}=\vartheta^{\alpha_1}\lambda_j$ and obtain similar to the case $j+1=1$ that
    \begin{align}
        \big|(Du)_{\vartheta\varrho}^{(\lambda_{j+1})} \big|
        \leq 
        \lambda_{j+1},
    \end{align}
    which concludes the proof of the iteration step. \\
    We are now in the position to finish the proof of the Lemma. We call $m$ the \textit{switching-index} from the degenerate to the non-degenerate case. If $m<\infty$, the iteration scheme stops at the switching-index $m$ and the Lemma follows immediately from the already established. It remains to treat the case $m=\infty$. To do this, we observe that for any $0<r\leq \varrho$ we can find $j\in \N_0$ such that $\vartheta^{j+1}< r \leq \vartheta^j \varrho$. Since we have for the scaling parameter $\lambda_j\leq 1$, there holds $Q_r\subset Q_{\vartheta^j \varrho}^{(\lambda_j)}$. This allows us to estimate the cylindrical mean with the help of inequality (\ref{9.5981}) as follows:
    \begin{align*}
        \big| (Du)_r \big| 
        &\leq
        \big| (Du)_r - (Du)_{\vartheta^j \varrho}^{(\lambda_j)} \big| + \big| (Du)_{\vartheta^j \varrho}^{(\lambda_j)} \big|
        \leq
        \bigg( \mint{-}\limits_{Q_r} \big| Du - (Du)_{\vartheta^j \varrho}^{(\lambda_j)} \big|^p \, dz \bigg)^{\frac{1}{p}} 
        +
        M_1 \lambda_j
        \nonumber\\ & \leq
        \bigg( \frac{|Q_{\vartheta^j \varrho}^{(\lambda_j)}|}{|Q_r|} \mint{-}\limits_{Q_{\vartheta^j \varrho}^{(\lambda_j)}} \big| Du - (Du)_{\vartheta^j \varrho}^{(\lambda_j)} \big|^p \, dz \bigg)^{\frac{1}{p}} 
        +
        M_1 \lambda_j
        \leq
        \big( \vartheta^{-(n+2)} \lambda_j^{2-p} \lambda_j^p \big)^{\frac{1}{p}} 
        +
        M_1 \lambda_j
        \nonumber\\ &  \leq 
        \vartheta^{-\frac{n+2}{p}} (1+M_1) \lambda_j \vartheta^{\frac{2j\alpha}{p}}
        \leq 
        \vartheta^{-\frac{n+2+2\alpha}{p}} (1+M_1) \lambda_j \bigg( \frac{r}{\varrho} \bigg)^{\frac{2j\alpha}{p}}.
    \end{align*}
    and we deduce 
    \begin{align*}
        \Gamma_0 \equiv \lim\limits_{r\downarrow 0} (Du)_r = 0.
    \end{align*}
    It remains to show (\ref{9.505}) for $j\in \N$. To this end, we again use (\ref{9.598}) to derive
    \begin{align*}
        \mint{-}\limits_{Q_{\vartheta^j \varrho}^{(\lambda_j)}} | Du-\Gamma_0|^p \, dz
        &\leq 
        2^{p-1} \Bigg( \mint{-}\limits_{Q_{\vartheta^j \varrho}^{(\lambda_j)}} \big| Du - (Du)_{\vartheta^j \varrho}^{(\lambda_j)} \big|^p \, dz \Bigg) 
        +
        \big| (Du)_{\vartheta^j \varrho}^{(\lambda_j)} \big|^p
        \nonumber\\ & \leq
        2^{p-1} (1+M_1^p) \lambda_j^p.
    \end{align*}
    This concludes the proof of the Lemma.
\end{proof}
We are now in the setting to proof the main Theorem \ref{Theo1} in the sub-quadratic case. 

\subsection{Proof of partial regularity in the super-quadratic case}

As in the proof of the preceding Lemma \ref{Lem9.5} we omit the dependencies on $\mathfrak{C}$ of the constants. We denote with $\alpha_1\in (0,1/p]$ the constant in Lemma \ref{Lem9.5} and define $\alpha=\alpha_1/p \in (0,\alpha_1)$. This Lemma will now be applied for some $z_0\in \Omega_T \setminus (\Sigma_1 \cup \Sigma_2)$, where $\Sigma_1$ and $\Sigma_2$ were specified in Theorem \ref{Theo2}. Since $z_0 \notin \Sigma_2$, there exists $M_1\geq 1$ and $\varrho_1>0$ such that $Q_{2\varrho_1}(z_0)\subset\Omega_T$ and $|(Du)_{z_0;\varrho}|< M_1$ for every $0<\varrho \leq \varrho_1$. Now we check that the smallness assumptions (\ref{9.502}) is satisfied for a particular $0<\varrho \leq \min\{\varrho_0,\varrho_1\}$, where $\varrho_0$ denotes the constant of Lemma \ref{Lem9.5}. To this end, let $\varepsilon_3$ be the smallness parameter of the same Lemma and recall the definition (\ref{8.02}) of $E_1(z_0,\varrho)$. Then we infer from the upper bound (\ref{9.03}), where we denote with $C_*$ the corresponding constant, for $0<\varrho \leq \min\{\varrho_0,\varrho_1\}$ the estimate
\begin{align*}
    &E_1(z_0,\varrho) \leq \Phi_1 \big(z_0,\varrho,        
    (Du)_{z_0;\varrho}\big) + C_* \varrho^\beta \big| (Du)_{z_0;\varrho} \big|^{p} \Bigg( \bigg[
    \Phi_\lambda\big(z_0,\varrho, (Du)_{z_0;\varrho} \big) 
    \nonumber \\ & \qquad + 
    \mathds{1}_{\{p<2\}}\big| (Du)_{z_0;\varrho} \big|^{\frac{p(2-p)}{2}} \Phi_1\big(z_0,\varrho, (Du)_{z_0;\varrho} \big)^\frac{p}{2} 
    +
    \big|(Du)_{z_0;\varrho}\big|^p +  \big|(Du)_{z_0;\varrho} \big|^{p(p-1)} 
    \nonumber \\ & \qquad +
    \Big( \Phi_1\big(z_0,\varrho, (Du)_{z_0;\varrho} \big) 
    + 
    \mathds{1}_{\{p<2\}} \big| (Du)_{z_0;\varrho} \big|^{\frac{p(2-p)}{2}} \Phi_1\big(z_0,\varrho, (Du)_{z_0;\varrho} \big)^\frac{p}{2} \Big)^{p-1} \bigg]^{\frac{\beta}{p}} + 2^{\frac{\beta}{2}} \Bigg)
    \nonumber\\ & \leq
    \Phi_1 \big(z_0,\varrho,        
    (Du)_{z_0;\varrho}\big) + C_* \varrho^\beta M_1^p \Bigg( \bigg[
    \Phi_\lambda\big(z_0,\varrho, (Du)_{z_0;\varrho} \big) 
    + 
    \mathds{1}_{\{p<2\}} M_1^{\frac{p(2-p)}{2}} \Phi_1\big(z_0,\varrho, (Du)_{z_0;\varrho} \big)^\frac{p}{2} 
    \nonumber \\ & \qquad +
    M_1^p +  M_1^{p(p-1)} 
    +
    \Big( \Phi_1\big(z_0,\varrho, (Du)_{z_0;\varrho} \big) 
    + 
    \mathds{1}_{\{p<2\}} M_1^{\frac{p(2-p)}{2}} \Phi_1\big(z_0,\varrho, (Du)_{z_0;\varrho} \big)^\frac{p}{2} \Big)^{p-1} \bigg]^{\frac{\beta}{p}} + 2^{\frac{\beta}{2}} \Bigg).
\end{align*}
Since $z_0\notin \Sigma_1$, we furthermore have by definition (\ref{6.01}) of $\Phi_1^{\frac{\beta}{p}}\big( z_0,\varrho, (Du)_{z_0;\varrho} \big)$
\begin{align*}
    \limsup\limits_{\varrho\downarrow 0} \Phi_1\big( z_0,\varrho, (Du)_{z_0;\varrho} \big) = 0.
\end{align*}
This ensures the existence of $\varrho_2(\varepsilon_3)>0$ such that
\begin{align*}
    \Phi_1\big( z_0,\varrho, (Du)_{z_0;\varrho} \big) \leq \frac{\varepsilon_3}{2} \leq M_1^p \qquad \forall 0<\varrho\leq \varrho_2
\end{align*}
and we conclude that the smallness assumption (\ref{9.502}) is satisfied for
\begin{align*}
    \varrho= \frac{1}{4} \min \Bigg\{\varrho_0,\varrho_1,\varrho_2, \bigg(\frac{\varepsilon_3}{16 C_* M_1^{p+p\beta}} \bigg)^{\frac{1}{\beta}} \Bigg\}.
\end{align*}
From measure theory we further know that the mappings $z\mapsto (Du)_{z;\varrho}$ and $z\mapsto E_1(z,\varrho)$ are continuous. This implies the existence of $0<R\leq \varrho/2$ such that $\big|(Du)_{z;\varrho}\big|<M_1$ and $E_1(z,\varrho)<\varepsilon_3$ holds for all $z\in Q_R(z_0)$. By choice of $R$ we furthermore have $Q_R(z)\subset Q_{2\varrho}(z_0)\subset\Omega_T$ for every $z\in Q_R(z_0)$. Thus, Lemma \ref{Lem9.5} is applicable for every $z\in Q_R(z_0)$ with the particular choices $\alpha$ and $M_1$ from the beginning of this section. From here on, we can proceed as in \cite[Section 8.4 and 9.4]{bogelein2013regularity} to obtain the Hölder continuity for the Lebesgue representative $z\mapsto \Gamma_z$ of $Du$.  This completes the proof of Theorem \ref{Theo1} and \ref{Theo2}.

\phantom{a}\\
\lyxaddress{$\hspace*{1em}$\textbf{Fabian Bäuerlein}\\
 Fachbereich Mathematik, Universität Salzburg \\
 Hellbrunner Str. 34, 5020 Salzburg, Austria.\\
 \textit{E-mail address}: fabian.baeuerlein@plus.ac.at}

\end{document}